\documentclass{amsproc}
\usepackage{graphicx}

\usepackage{lmodern,microtype}
\usepackage[T1]{fontenc}
\usepackage[utf8]{inputenc}
\usepackage[UKenglish]{babel}

\usepackage{booktabs}
\usepackage{pinlabel}
\newcommand{\mybg}[2]{%
	{\begin{tikzpicture}
		\node [fill=white,fill opacity=0.8,text opacity=1, rounded corners=2pt,inner sep=1pt] {\textcolor{#1}{#2}};
		\end{tikzpicture}}
}

\usepackage{xcolor}
\colorlet{myblue}{black}
\colorlet{lightred}{red!20!white}
\colorlet{lightblue}{blue!30!white}
\colorlet{darkgreen}{green!70!black}
\colorlet{lightgreen}{green!50!white}
\colorlet{llightgray}{lightgray!50!white}
\colorlet{gold}{yellow!90!black!70!red}
\usepackage{caption}
\usepackage[labelfont=up,margin=5pt]{subcaption}

\usepackage{tikz}
\usepackage{tikz-cd}
\tikzcdset{diagrams={nodes={inner sep=2pt}}}
\usetikzlibrary{decorations,calligraphy}
\usetikzlibrary{decorations.pathmorphing}
\usetikzlibrary{decorations.pathreplacing}
\usetikzlibrary {arrows.meta} 
\tikzcdset{arrow style=tikz, diagrams={>=latex}}
\tikzset{>=latex}
\tikzstyle{t}=[fill=white,inner sep=1.5pt,rounded corners]
\tikzstyle{d}=[dashed,dash pattern={on 3pt off 1.5pt}]
\tikzstyle{g} = [line width=30pt,llightgray,line cap=round,style={shorten >=-10pt,shorten <=-10pt}]
\tikzstyle{gt} = [line width=25pt,llightgray,line cap=round,style={shorten >=-10pt,shorten <=-10pt}]
\tikzstyle{gb} = [rounded corners,ultra thick,lightgray]
\tikzstyle{gd} = [rounded corners,ultra thick,lightgray,dashed]
\tikzstyle{ppscale}=[scale=0.6,xscale=1.2]
\usepackage{pinlabel}

\usepackage[section]{placeins}
\usepackage[bookmarks=true,pdfencoding=auto,hidelinks]{hyperref}
\usepackage[nameinlink,capitalise]{cleveref}
\let\cref\Cref
\Crefname{equation}{}{}
\Crefname{subsection}{Subsection}{Subsections}

\usepackage{soul}

\usepackage{amsmath}
\usepackage{amsfonts}
\usepackage{amssymb}
\usepackage{amsthm}
\usepackage{nicefrac}
\usepackage{mathtools}
\usepackage{enumitem}
\mathtoolsset{mathic=true}
\usepackage{bm} 
\usepackage{marvosym}

\newtheorem{theorem}{Theorem}[section]
\newtheorem{lemma}[theorem]{Lemma}
\newtheorem{conjecture}[theorem]{Conjecture}
\crefname{conjecture}{Conjecture}{Conjectures}
\newtheorem{question}[theorem]{Question}

\newtheorem{corollary}[theorem]{Corollary}
\newtheorem{observation}[theorem]{Observation}
\newtheorem{proposition}[theorem]{Proposition}

\newtheorem*{namedthm}{\namedthmname}
\newcounter{namedthm}
\makeatletter
\newenvironment{named}[1]
{\def\namedthmname{#1}%
	\refstepcounter{namedthm}%
	\namedthm\def\@currentlabel{#1}}
{\endnamedthm}
\makeatother

\theoremstyle{definition}

\newtheorem{definition}[theorem]{Definition}
\newtheorem{example}[theorem]{Example}

\theoremstyle{remark}
\newtheorem{remark}[theorem]{Remark}

\numberwithin{equation}{section}

\DeclareMathOperator{\HFT}{HFT}
\newcommand{\HFKminus}{\operatorname{HFK}^-}
\DeclareMathOperator{\HFKhat}{\widehat{HFK}}
\DeclareMathOperator{\HFhat}{\widehat{HF}}
\DeclareMathOperator{\Khr}{\widetilde{Kh}}
\DeclareMathOperator{\BNr}{\widetilde{BN}}
\DeclareMathOperator{\Kh}{Kh}
\DeclareMathOperator{\HF}{HF}
\DeclareMathOperator{\CFD}{\widehat{CFD}}
\DeclareMathOperator{\BSD}{\widehat{BSD}}

\DeclareMathOperator{\BSAD}{\widehat{BSAD}}
\newcommand{\typeA}[2]{\prescript{}{#1}{#2}}
\newcommand{\HFcurve}[1]{\gamma^{\operatorname{HF}}_{#1}}
\newcommand{\Khcurve}[1]{\gamma^{\operatorname{Kh}}_{#1}}

\newcommand{\BSADY}{\typeA{\algT}{\mathcal{Y}}^{\algS}}
\newcommand{\capL}{\mathbf{v}}
\newcommand{\capR}{\mathbf{u}}
\newcommand{\ccon}{\mathbf{d}}
\newcommand{\s}{\mathbf{s}}
\newcommand{\sS}{\overline{\mathbf{s}}}
\newcommand{\sHF}{\mathbf{s}^{\operatorname{HF}}}
\newcommand{\sSHF}{\sS^{\operatorname{HF}}}
\newcommand{\sKh}{\mathbf{s}^{\operatorname{Kh}}}
\newcommand{\rr}{\mathbf{r}}
\newcommand{\rHF}{\mathbf{r}^{\operatorname{HF}}}
\newcommand{\rKh}{\mathbf{r}^{\operatorname{Kh}}}

\setuldepth{r}
\DeclareMathOperator{\rgr}{\text{\ul{\normalfont gr}}}
\DeclareMathOperator{\rGr}{\text{\ul{\normalfont Gr}}}
\DeclareMathOperator{\gr}{gr}
\DeclareMathOperator{\Gr}{Gr}
\newcommand{\HFKgr}[2]{[#1;#2]}
\newcommand{\HFTgr}[4]{\prescript{#1}{}{#4}^{#2}_{#3}}

\newcommand{\algT}{\mathcal{A}}
\newcommand{\algS}{\mathcal{B}}
\newcommand{\Z}{\mathbb{Z}} 
\newcommand{\R}{\mathbb{R}}
\newcommand{\Q}{\mathbb{Q}}
\newcommand{\QPI}{\operatorname{\mathbb{Q}P}^1}  
\newcommand{\F}{\mathbb{F}_2} 
\DeclareMathOperator{\id}{id} 
\DeclareMathOperator{\im}{im} 
\newcommand{\half}{\tfrac{1}{2}}
\newcommand{\hhalf}{\frac{1}{2}}
\newcommand{\hZ}{\half\Z}

\DeclareMathOperator{\am }{\alpha_\mu}
\DeclareMathOperator{\al }{\alpha_\lambda}
\DeclareMathOperator{\aaa}{\alpha_\mathit{a}}
\DeclareMathOperator{\aab}{\alpha_\mathit{b}}
\DeclareMathOperator{\aac}{\alpha_\mathit{c}}
\DeclareMathOperator{\aad}{\alpha_\mathit{d}}
\DeclareMathOperator{\ba }{\beta_\mathit{a}}
\DeclareMathOperator{\bb }{\beta_\mathit{b}}
\DeclareMathOperator{\bc }{\beta_\mathit{c}}
\DeclareMathOperator{\bd }{\beta_\mathit{d}}
\newcommand{\arcA}{\textcolor{red}{a}}
\newcommand{\arcB}{\textcolor{blue}{b}}
\newcommand{\arcC}{\textcolor{darkgreen}{c}}
\newcommand{\arcD}{\textcolor{gold}{d}}
\newcommand{\I}  {{\normalfont\texttt{1}}}
\newcommand{\II} {{\normalfont\texttt{2}}}
\newcommand{\III}{{\normalfont\texttt{3}}}
\newcommand{\IV} {{\normalfont\texttt{4}}}
\newcommand{\Za}{\mathcal{Z}_{\alpha}}
\newcommand{\Zb}{\mathcal{Z}_{\beta}}
\newcommand{\FZb}{\mathcal{F}(\Zb)}
\newcommand{\FZa}{\mathcal{F}(\Za)}

\newcommand{\ellmax}{\ell_{\text{max}}}
\newcommand{\Sset}{S_{\text{\LeftScissors}}}
\DeclareMathOperator{\interior}{int}
\newcommand{\PI}{\mathcal{P}_1}
\newcommand{\PIII}{\mathcal{P}_3}
\newcommand{\vc}[1]{\vcenter{\hbox{#1}}}
\def\co{\colon\thinspace\relax}
\newcommand{\Aa}[1]{\operatorname{A}(#1)}

\begin{document}
\title{Heegaard Floer multicurves of double tangles}
\author{Claudius Zibrowius}
\address{Durham University, Department of Mathematical Sciences, United Kingdom}
\email{\href{mailto:cbz20@posteo.net}{cbz20@posteo.net}}
\thanks{}
\urladdr{\url{https://cbz20.raspberryip.com/}}
\hypersetup{pdfauthor={\authors},pdftitle={\shorttitle}}
\date{\today}

\begin{abstract}
	We describe a simple formula for computing the Heegaard Floer multicurve invariant of double tangles from the Heegaard Floer multicurve invariant of knot complements.
A comparison with a similar multicurve invariant for Conway tangles in the setting of Khovanov homology confirms that knot Floer homology and Khovanov homology behave very differently under satellite operations, echoing recent observations from \cite{LZ}. 
We also obtain a new characterisation of L-space knots in terms of Heegaard Floer A-link satellites. 
Along the way, we find the first example of a satellite knot whose knot Floer homology is thin.  
\end{abstract}
\maketitle

\section{Introduction}

\begin{figure}[b]
	\centering
	\includegraphics[height=3.4cm]{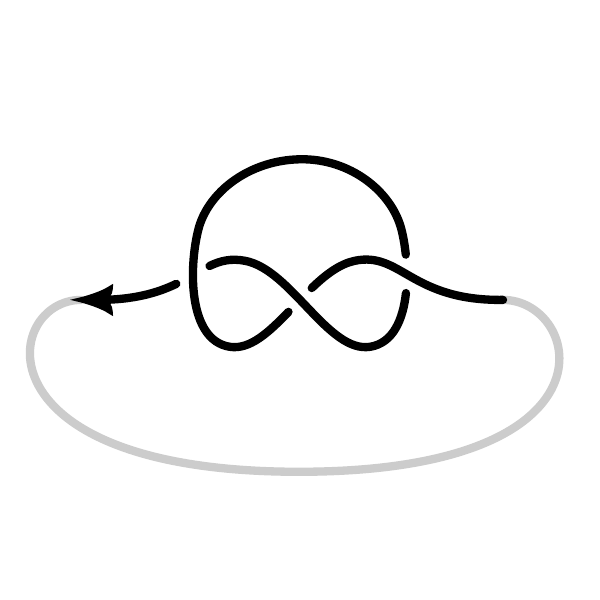}
	\hspace{1cm}
	\centering
	\labellist
	\pinlabel \(\texttt{1}\) at 65 225
	\pinlabel \(\texttt{2}\) at 65 65
	\pinlabel \(\texttt{3}\) at 225 65
	\pinlabel \(\texttt{4}\) at 225 225
	\endlabellist
	\includegraphics[height=3.4cm]{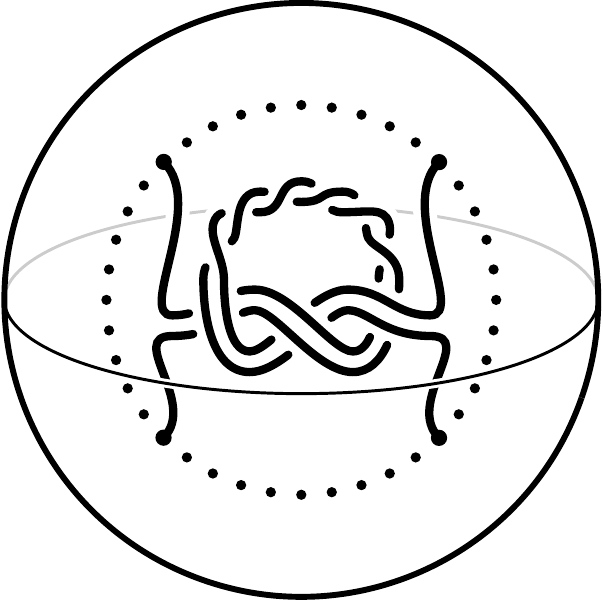}
	\caption{The double tangle (right) of the right-handed trefoil knot (left).%
	}\label{fig:double-tangle}
\end{figure}

The double tangle \(T_K\) associated with an oriented knot \(K\subset S^3\) is the Conway tangle obtained by cutting the knot \(K\) open to a tangle with two ends and adding another tangle strand parallel to the first, as illustrated in \cref{fig:double-tangle}. 
This family of Conway tangles plays a key role for understanding satellites of wrapping number 2: 
Any such satellite knot \(P(K)\) can be written as
\[
P(K)=T_K\cup T_P,
\]
the union of the double tangle \(T_K\) associated with the companion knot \(K\) and a second Conway tangle \(T_P\) corresponding to the pattern \(P\).
This perspective allows us to investigate satellite formulae for any knot invariants that admit glueable generalisations to Conway tangles.
For example, in joint work with Lukas Lewark, I recently applied this strategy to prove such satellite formulae for Rasmussen invariants \cite{LZ}, using the multicurve tangle invariant generalising Bar-Natan homology \cite{KWZ}.

The focus of this paper lies on a similar multicurve invariant: the tangle invariant \(\gamma_T\coloneqq\HFT(T)\) from Heegaard Floer theory \cite{pqMod,pqSym}. 
The main goal is to show that \(\gamma_{T_K}\) can be read off easily from yet another multicurve invariant, namely the Hanselman-Rasmussen-Watson multicurve \(\gamma_K\coloneqq \HFhat(S^3\smallsetminus\mathring{\nu}(K))\) associated with the knot exterior \cite{HRW,HRWcompanion}. 
Before explaining how this works, we briefly recall those features of the two multicurve invariants that are relevant to this story.
The reader familiar with both invariants may skip straight to the statement of the \ref{thm:main:intro} (\cpageref{thm:main:intro}), which is illustrated in \cref{fig:executive-summary}.

\begin{figure}[p]
	\centering	
		\labellist

		\pinlabel \(\R^2\smallsetminus\Z^2\) at 178 430
		\pinlabel	\(\downarrow\) at 178 405
		\pinlabel	\(S^1\times S^1\smallsetminus \{z\}\) at 178 380
		
		\pinlabel \(\R^2\smallsetminus\Z^2\) at 568 430
		\pinlabel	\(\downarrow\) at 568 405
		\pinlabel	\(S^2\smallsetminus \{\text{points labelled \(\I\), \(\II\), \(\III\), and \(\IV\)}\}\) at 568 380
		
		\pinlabel \mybg{red}{\(\alpha_a\)} at 438 167
		\pinlabel \mybg{red}{\(\alpha_b\)} at 568 50
		\pinlabel \mybg{red}{\(\alpha_c\)} at 703 167
		\pinlabel \mybg{red}{\(\alpha_d\)} at 568 303
		
		\pinlabel \(\textcolor{red}{\lambda_K}\) at 270 57
		\pinlabel \(\textcolor{red}{\mu_K}\) at 182 21
		\pinlabel \(\textcolor{darkgreen}{z}\) at 148 70
		\pinlabel \(\textcolor{darkgreen}{\sigma_1}\) at 128 50
		\pinlabel \(\textcolor{darkgreen}{\sigma_3}\) at 168 89
		\pinlabel \(\textcolor{darkgreen}{\sigma_2}\) at 172 50
		
		\pinlabel \(\textcolor{red}{\lambda_K}\) at 180 910
		\pinlabel \(\textcolor{red}{\mu_K}\) at 140 948
		
		\pinlabel \(\textcolor{blue}{\gamma_K}\) at 260 127
		\pinlabel \(\textcolor{blue}{\gamma_{T_K}}\) at 520 320
		
		\pinlabel \(\textcolor{darkgreen}{\sigma_2}\) at 142 918
		\pinlabel \(\textcolor{darkgreen}{\sigma_1}\) at 94 918
		\pinlabel \(\textcolor{darkgreen}{\sigma_3}\) at 142 870
		
		\pinlabel \(\tilde{\capR}_1\) at 306 745
		\pinlabel \(\tilde{\capL}_1\) at 210 682
		\pinlabel \(\tilde{\ccon}_2\) at 166 611

		\pinlabel \(\tilde{\s}_2\) at 537 850
		\pinlabel \(\tilde{\sS}_2\) at 591 708
		\pinlabel \(\tilde{\rr}_4\) at 527 609

		\small
		\pinlabel \(\textcolor{darkgreen}{\texttt{1}}\) at 485 259
		\pinlabel \(\textcolor{darkgreen}{\texttt{2}}\) at 485 95
		\tiny
		\pinlabel \(\textcolor{darkgreen}{\texttt{3}}\) at 651 96
		\pinlabel \(\textcolor{darkgreen}{\texttt{4}}\) at 653 260	
		
		\pinlabel \(\textcolor{darkgreen}{q_1}\) at 530 796
		\pinlabel \(\textcolor{darkgreen}{q_1}\) at 486 752
		\pinlabel \(\textcolor{darkgreen}{p_1}\) at 486 796
		\pinlabel \(\textcolor{darkgreen}{p_1}\) at 530 752
		\pinlabel \(\textcolor{darkgreen}{q_4}\) at 606 796
		\pinlabel \(\textcolor{darkgreen}{q_4}\) at 650 752
		
		\pinlabel \(\textcolor{darkgreen}{p_2}\) at 530 676
		\pinlabel \(\textcolor{darkgreen}{p_2}\) at 486 632
		\pinlabel \(\textcolor{darkgreen}{q_2}\) at 486 676
		\pinlabel \(\textcolor{darkgreen}{q_2}\) at 530 632
		\pinlabel \(\textcolor{darkgreen}{p_3}\) at 606 676
		\pinlabel \(\textcolor{darkgreen}{p_3}\) at 650 632
		
		\pinlabel \(\textcolor{darkgreen}{\texttt{2}}\) at 508 894
		\pinlabel \(\textcolor{darkgreen}{\texttt{3}}\) at 628 894
		\pinlabel \(\textcolor{darkgreen}{\texttt{1}}\) at 508 774
		\pinlabel \(\textcolor{darkgreen}{\texttt{4}}\) at 628 774
		\pinlabel \(\textcolor{darkgreen}{\texttt{2}}\) at 508 654
		\pinlabel \(\textcolor{darkgreen}{\texttt{3}}\) at 628 654
		\pinlabel \(\textcolor{darkgreen}{\texttt{1}}\) at 508 534
		\pinlabel \(\textcolor{darkgreen}{\texttt{4}}\) at 628 534
		\pinlabel \(\textcolor{red}{\aaa}\) at 497 731	
		\pinlabel \(\textcolor{red}{\aab}\) at 567 644
		\pinlabel \(\textcolor{red}{\aac}\) at 640 690	
		\pinlabel \(\textcolor{red}{\aad}\) at 566 782
		\endlabellist
		\includegraphics[width=\textwidth]{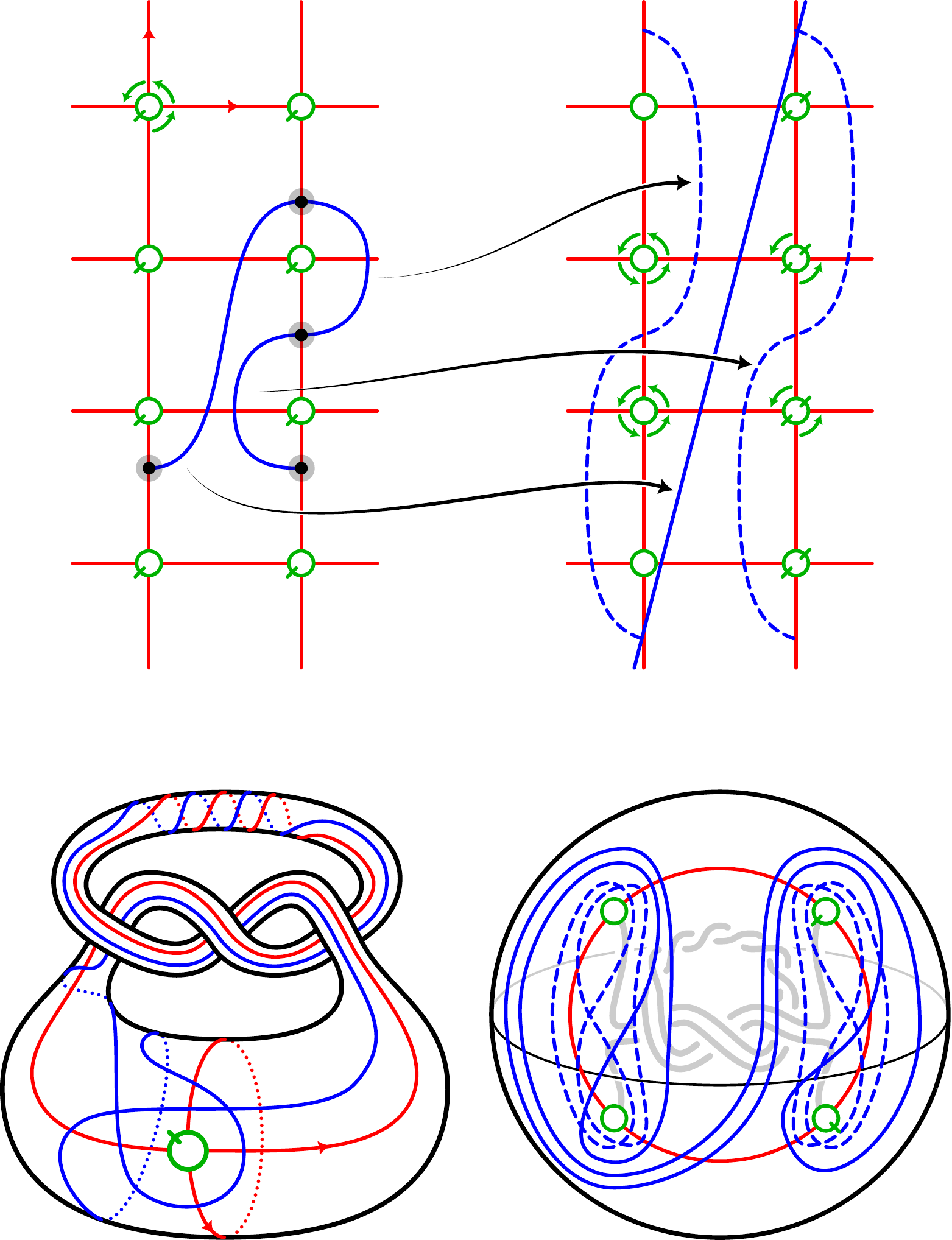}
	\caption{Executive summary of the \ref{thm:main:intro} (\cpageref{thm:main:intro}): 
		To construct \(\gamma_{T_K}\), first cut \(\gamma_K\) along the generators \(\bullet\) of \(\HFKhat(K)\) into curve segments. This always produces a single curve segment \(\ccon_{2\tau(K)}\) of slope \(2\tau(K)\) as well as some number of curve segments \(\capR_{\ell}\) and \(\capL_{\ell}\) that are right and left caps, respectively, of some length \(\ell\in\Z^{>0}\). 
		Then \(\gamma_{T_K}\) contains a rational component \(\rr_{4\tau(K)}\coloneqq\rr(4\tau(K))\) as well as a special component \(\s_{2\ell}\coloneqq\s_{2\ell}(\infty;\I,\II)\) and \(\sS_{2\ell}\coloneqq\s_{2\ell}(\infty;\III,\IV)\) for every \(\capR_{\ell}\) and \(\capL_{\ell}\) in \(\gamma_K\), respectively. 
		Note: In all pictures except the bottom left, we are looking at the manifolds from the outside.
}\label{fig:executive-summary}
\end{figure}

\subsection{%
	\texorpdfstring{%
		A brief review of the multicurve invariant \(\bm{\gamma_K}\)}{%
		A brief review of the multicurve invariant γ\_Κ
	}}
The knot invariant \(\gamma_K\) is a multicurve, a finite set of free homotopy types of closed curves%
\footnote{%
As part of the data of a multicurve, each curve is equipped with an orientation, a local system, and some grading data. 
Here, in the introduction, we will ignore any of these additional curve decorations. 
}, on the one-punctured torus. 
This torus is naturally identified with the boundary of the knot exterior \(S^3\smallsetminus\interior(\nu(K))\) and the puncture can be thought of as some basepoint \(z\) on said boundary.

To describe key properties of \(\gamma_K\), we introduce some notation:
Let \(\mu_K\) denote a meridian of \(K\) 
and \(\lambda_K\) a longitude of \(K\), both of which go through the basepoint \(z\). 
We may assume without loss of generality that \(\gamma_K\) intersects the meridian \(\mu_K\) and longitude \(\lambda_K\) transversally and minimally.  
Then by \cite[Example~6.8]{HRW}, the intersection points of \(\gamma_K\) with \(\mu_K\) generate the knot Floer homology of \(K\):
\[
\HFKhat(K)
\cong
\F\langle
\gamma_K \pitchfork \mu_K
\rangle.
\]
(We will work over the field \(\F\) of two elements throughout this paper.)
These intersection points cut the multicurve \(\gamma_K\) into \textit{curve segments}, which can be characterized by how their ends approach \(\mu_K\) and by how they intersect \(\lambda_K\). 
To make this more precise, let \(\R^2\) be the universal cover of the torus such that the integer lattice \(\Z^2\subset \R^2\) agrees with the preimage of the basepoint \(z\). 
This is illustrated on the left of \cref{fig:executive-summary}. 
The preimage of \(\mu_K\) under this covering map is equal to the union of the vertical lines \(\tilde{\mu}_K\coloneqq\Z\times\R\). 
Given a curve segment \(c\) of \(\gamma_K\), consider a lift \(\tilde{c}\) to \(\R^2\). 
It starts and ends on some connected components \(\tilde{\mu}^+_K\) and \(\tilde{\mu}^-_K\) of \(\tilde{\mu}_K\), respectively.
We now say the curve segment \(c\) is of type 
\[
\begin{cases*}
\capR_{\ell}
&
if \(\tilde{\mu}^+_K=\tilde{\mu}^-_K\) and \(\tilde{c}\) lies to the right of \(\tilde{\mu}^\pm_K\);
\\
\capL_{\ell}
&
if \(\tilde{\mu}^+_K=\tilde{\mu}^-_K\) and \(\tilde{c}\) lies to the left of \(\tilde{\mu}^\pm_K\);
\\
\ccon_\ell
&
if \(\tilde{\mu}^+_K\neq\tilde{\mu}^-_K\) (in which case \(\tilde{c}\) lies between \(\tilde{\mu}^+_K\) and \(\tilde{\mu}^-_K\)).
\end{cases*}
\]
In the first two cases, we call the subscript \(\ell\in\Z\) the \textit{length} of \(c\); in the third case, we call it the \textit{slope} of \(c\). 
In all three cases, \(|\ell|\) equals the number of intersection points of \(c\) with  \(\lambda_K\). 
In the first two cases, this already determines the curve segments up to homotopy (relative to their endpoints), so we choose \(\ell\geq0\). 
In fact, \(\ell>0\) for any such curve segment, since we assume that \(\gamma_K\) intersects \(\mu_K\) minimally.   
In the third case, we choose the sign of \(\ell\) to be positive if and only if the slope of \(\tilde{c}\) is positive (after some homotopy relative endpoints). 
Clearly, the type of a curve segment \(c\) is independent of the choice of lift \(\tilde{c}\) and invariant under homotopy of \(c\) relative endpoints. 

\begin{remark}\label{rem:UV=0}
	The multicurve invariant \(\gamma_K\) is equivalent to a version of knot Floer homology that takes the form of a chain complex over the polynomial ring \(\F[U,V]/(UV=0)\); see for example \cite[Theorem~2]{KWZ_mnemonic}. 
	The curve segments \(\capR_{\ell}\) and \(\capL_{\ell}\) respectively correspond to components \(U^\ell\) and \(V^\ell\) of the differential in this chain complex.
\end{remark}

The first part of the following structure theorem follows from \cite[Theorem~11.26]{LOT}. 
The second part follows from the usual conjugation symmetry on knot Floer homology; see also \cite[Theorem~7]{HRWcompanion}. 

\begin{theorem}\label{thm:structure:HFKcurve}
	For any knot \(K\subset S^3\), the multicurve \(\gamma_K\) contains exactly one curve segment \(\ccon_{2\tau(K)}\), where \(\tau(K)\) is the Ozsváth-Szabó concordance invariant, and every other curve segment is of type \(\capR_{\ell}\) or \(\capL_{\ell}\) for some \(\ell\in\Z^{>0}\).
	Moreover, for any \(\ell\in\Z^{>0}\), 
	\[
	\#\{\text{curve segments in \(\gamma_K\) of type \(\capR_{\ell}\)}\}
	=
	\#\{\text{curve segments in \(\gamma_K\) of type \(\capL_{\ell}\)}\}.
	\]
\end{theorem}


\subsection{%
	\texorpdfstring{%
		A brief review of the multicurve invariant \(\bm{\gamma_T}\)}{%
		A brief review of the multicurve invariant γ\_Τ
}}
The tangle invariant \(\gamma_T\) is also a multicurve; in this case, the multicurve lives on the four-punctured sphere which is naturally identified with the boundary of the three-dimensional ball containing the tangle minus the four tangle ends. 
Again, it is useful to consider the curves in a certain planar cover \(\R^2\smallsetminus\Z^2\) of the four-punctured sphere, which is illustrated on the top right of \cref{fig:executive-summary}. 
The integer lattice corresponds to the four punctures of the sphere, labelled \(\I\), \(\II\), \(\III\), and \(\IV\), the vertical lines \(\Z\times\R\) are the preimage of the arcs connecting tangle ends (\(\I\) and  \(\II\)) and (\(\III\) and \(\IV\)) and similarly, the horizontal lines \(\R\times\Z\) correspond to the arcs connecting (\(\I\) and \(\IV\)) and (\(\II\)~and~\(\III\)).

The components of the multicurve \(\gamma_T\) are completely classified \cite[Theorem~0.5]{pqSym}. 
We do not need the full classification here. 
We only describe those curves that show up as components of~\(\gamma_{T_K}\). 
First, an embedded simple closed curve in the four-punctured sphere is called a \textit{rational curve of slope \(k\in\Z\)}, denoted by \(\rr(k)\) if it lifts to a straight line of slope \(k\) in the cover \(\R^2\smallsetminus\Z^2\). 
Second, let \((\mathtt{i},\mathtt{j})=(\I,\II)\) or \((\III,\IV)\). 
Consider an infinite straight line in \(\R^2\smallsetminus\Z^2\) of slope \(\infty\) going through some lattice points corresponding to the punctures \(\mathtt{i}\) and \(\mathtt{j}\). 
The lattice points divide this line into intervals of equal length. For a fixed integer \(\ell\in2\Z^{>0}\), let us mark every \(\ell\)th interval of the line.
Then consider a small equivariant push-off of this line such that it intersects only the marked intervals and
each of them exactly once. 
Finally, let \(\s_{\ell}(\infty;\mathtt{i},\mathtt{j})\) be the immersed, primitive curve in the four-punctured sphere that lifts to this push-off. 
We call \(\s_{\ell}(\infty;\mathtt{i},\mathtt{j})\) the \textit{special curve of slope \(\infty\) and length \(\ell\)} through
the punctures \(\mathtt{i}\) and \(\mathtt{j}\). 
In this paper, we will use the following simplified notation:
\[
\rr_k\coloneqq \rr(k),
\quad
\s_{\ell}\coloneqq\s_{\ell}(\infty;\I,\II),
\quad\text{and}\quad
\sS_{\ell}\coloneqq\s_{\ell}(\infty;\III,\IV).
\]

The following structure theorem for \(\gamma_{T_K}\) is a special case of a general result for cap-trivial tangles, analogous in both statement and proof to \cite[Theorem~3.1]{KWZ_strong_inversions}. 

\begin{theorem}\label{thm:structure:HFTcurve}
	For any knot \(K\subset S^3\), the multicurve \(\gamma_{T_K}\) contains exactly one curve segment \(\rr_{k}\), for some \(k\in2\Z\), and every other curve segment is of type \(\s_{\ell}\) or \(\sS_{\ell}\) for some \(\ell\in2\Z^{>0}\).
	Moreover, for any \(\ell\in2\Z^{>0}\), 
	\[
	\#\{\text{components \(\s_{\ell}\) of \(\gamma_{T_K}\)}\}
	=
	\#\{\text{components \(\sS_{\ell}\) of \(\gamma_{T_K}\)}\}.
	\]
\end{theorem}

\subsection{The Main Theorem}

The similarity between the two structure theorems for \(\gamma_K\) and \(\gamma_{T_K}\) suggests a close relationship between the two invariants.
And indeed, it is as simple as it can possibly be!

\begin{named}{Main Theorem}\label{thm:main:intro}
	The components of \(\gamma_{T_K}\) are in one-to-one correspondence with the curve segments obtained by splitting the curve \(\gamma_K\) along its intersection points with the meridian \(\mu_K\).
	More specifically, we have the following correspondence:
	\[
	\text{in~}\gamma_K~
	\left\{
	\begin{array}{ccc}
	\ccon_{2\tau(K)}
	&\longleftrightarrow&
	\rr_{4\tau(K)}
	\\
	\capR_{\ell} 
	&\longleftrightarrow&
	\s_{2\ell}
	\\
	\capL_{\ell}
	&\longleftrightarrow&
	\sS_{2\ell}
	\end{array}
	\right\}
	~\text{in~}\gamma_{T_K}.
	\]
\end{named}
For a relatively bigraded version of this result, see \cref{thm:main:graded}.

\begin{example}\label{exa:main:intro}
	\cref{fig:executive-summary} illustrates the \ref{thm:main:intro} for the right-handed trefoil knot \(K=T_{2,3}\).
	The curve \(\gamma_K\) decomposes into three curve segments, namely \(\ccon_{2}\), \(\capR_{1}\), and \(\capL_{1}\). 
	Hence, \(\gamma_{T_K}\) consists of three components: \(\rr_4\), \(\s_{2}\), and \(\sS_{2}\). 
	Interestingly, as an ungraded curve, this agrees with the invariant of the quotient tangle \(T_{K,h}\) of \(K\) under the unique strong inversion \(h\) on \(K\) \cite[Figure~4]{KWZ_strong_inversions}. 
	(The tangle \(T_{K,h}\) is equal to the \((2,-3)\)-pretzel tangle, up to some number of twists, and its invariant \(\gamma_T\) was computed in \cite[Example~2.26]{pqMod}.)
	However, as graded invariants, \(\gamma_{T_K}\) and \(\gamma_{T_{K,h}}\) are distinct; see \cref{exa:main:graded}. 
\end{example}

\begin{remark}
The multicurve invariants \(\gamma_K\) are very computable thanks to Szabó's computer program \cite{hfkcalc} and Hanselman's implementation of the arrow sliding algorithm \cite{hanselman_program}. 
Hanselman has compiled a list of the invariants for all knots up to 15 crossings, which is available on github along with the source code of his program. 
While there also exists software \cite{PQM.m} for computing the tangle invariant \(\gamma_{T}\), it is neither as efficient nor as practical. 
Consequently, \(\gamma_T\) had so far only been known for very few tangles \(T\) other than the family of two-stranded pretzel tangles considered in \cite[Theorem~6.9]{pqMod}.
The \ref{thm:main:intro} now provides an infinite family of new tangles for which the invariants are very easy to compute. 
The python script \cite{hf2hft} automates this computation, taking as input the multicurves \(\gamma_K\) in the format used in Hanselman's database \cite{hanselman_program}. 
\end{remark}

\subsection{\texorpdfstring{A reformulation of the \ref{thm:main:intro}}{A reformulation of the Main Theorem}}
Recall the following well-known structure theorem for \(\HFKminus(K)\), a version of knot Floer homology taking the form of a bigraded module over \(\F[U]\); see for example \cite[Chapter~7]{GridHomologyBook}. 

\begin{theorem}\label{thm:structure:HFKminus}
	For any knot \(K\subset S^3\), there exist integers \(n\in\Z^{\geq0}\) and \(\ell_1,\dots,\ell_n\in\Z^{>0}\) such that
	\[
	\HFKminus(K)
	\cong
	\F[U]
	\oplus
	\bigoplus_{i=1}^{n}
	\F[U]/U^{\ell_i}.
	\]	
	The integers \(\ell_1,\dots,\ell_n\) are unique up to permutation.
\end{theorem}

The module \(\HFKminus(K)\) is generally not sufficient to determine \(\gamma_K\). 
However, the curve segments of \(\gamma_K\) can be determined from \(\HFKminus(K)\). 
First, the Alexander grading of the generator of \(\F[U]\) is equal to \(-\tau(K)\), by definition.
Second, using conjugation symmetry of knot Floer homology, one can show that for every \(U\)-torsion summand \(\F[U]/U^{\ell_i}\) in \(\HFKminus(K)\), \(\gamma_K\) contains a pair of curve segments \(\capR_{\ell_i}\) and \(\capL_{\ell_i}\).
We therefore obtain the following reformulation of the \ref{thm:main:intro}.

\begin{theorem}\label{thm:main:FUreformulation}
	Given a knot \(K\subset S^3\), let \(n\in\Z^{\geq0}\) and \(\ell_1,\dots,\ell_n\in\Z^{>0}\) be the integers determined by \cref{thm:structure:HFKminus}. 
	Then
	\[
		\pushQED{\qed} 
		\gamma_{T_K}
		=
		\rr_{4\tau(K)}
		\cup 
		\bigcup_{i=1}^{n}
		\s_{2\ell_i}\cup \sS_{2\ell_i}.
		\qedhere
		\popQED
	\]
\end{theorem}

\subsection{Consequences and questions: satellites, thinness, and A-links}\label{subsec:intro:thinness}
In 2021, Steve Boyer, Cameron Gordon, and Ying Hu showed that the Khovanov homology of any satellite knot is not thin, i.e.\ it is supported in more than one \(\delta\)-grading \cite{boyer2021slope}. 
Up to now, the corresponding result for knot Floer homology was only known to be true for some specific patterns such as cables and the Mazur pattern \cite{dey2019cable,petkova2021twisted}. 
In fact, it is false in general: 

\begin{theorem}\label{thm:intro:thin-satellite}
	There exists a satellite knot \(K\subset S^3\) for which \(\HFKhat(K)\) is thin. 
\end{theorem}

\begin{figure}[t]
	\centering	
	\includegraphics[width=0.29\textwidth]{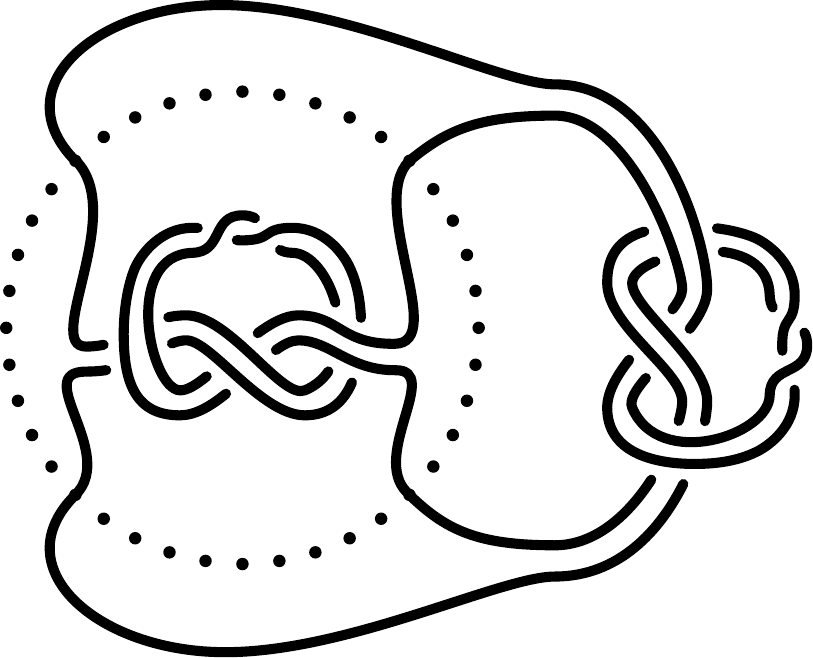}
	\caption{%
		A satellite knot with thin knot Floer homology. 
		The dotted circle indicates a Conway sphere splitting the knot into two tangles that agree with the double tangles of the left- and right-handed trefoil knots up to some twisting.
	}\label{fig:thin-satellite}
\end{figure}

The knot in question is shown in \cref{fig:thin-satellite}. 
At the time of writing, I do not know if this is the only such example.

\begin{question}\label{que:intro:thin-satellite}
	Is this the only satellite knot with thin knot Floer homology?
\end{question}

By applying thinness criteria for unions of Conway tangles from joint work with Artem Kotelskiy and Liam Watson \cite{KWZ_thin} to the graded version of the \ref{thm:main:intro} (\cref{thm:main:graded}), we can give a partial answer to \cref{que:intro:thin-satellite}:
\begin{proposition}\label{prop:intro:thick-satellites}
	Let \(K\subset S^3\) be a knot and \(P\) a pattern of wrapping number 2. 
	Suppose \(K\) is neither the unknot nor a trefoil knot.
	Then \(\HFKhat(P(K))\) is not thin.
\end{proposition}

Following \cite{KWZ_thin}, we call links whose knot Floer homology is supported in \(\delta\)-gradings of the same parity \textit{Heegaard Floer A-links}, or simply \textit{A-links} for short.
The set of A-links contains all alternating links, and more generally the set of thin links as a proper subset. 
This definition is inspired by the definition of L-spaces, which include all lens spaces \cite{HFKlens}. 
Recall that an L-space knot is a knot that admits any non-trivial L-space surgery.
Similar to \cref{prop:intro:thick-satellites}, we show:

\begin{proposition}\label{prop:A-knot-satellites}
	Let \(K\subset S^3\) be a knot and \(P\) a pattern of wrapping number 2. 
	Suppose \(K\) is not an L-space knot. 
	Then the satellite knot \(P(K)\) is not an A-link. 
\end{proposition}

Given \(\nicefrac{p}{q}\in\QPI\), let \(Q_{p/q}\) denote the rational tangle of slope \(\nicefrac{p}{q}\). 
We define the \textit{\(\nicefrac{p}{q}\)-rational filling} \(T(\nicefrac{p}{q})\) of a Conway tangle \(T\) as the union \(Q_{-p/q}\cup T\). 
If a rational filling of a tangle is an A-link, we call it a \textit{rational A-link filling} \cite{KWZ_thin}. 
Note that for any knot \(K\), \(T_K(\infty)\) is the unknot, so \(T_K\) always admits at least one rational A-link filling.
This is analogous to the fact that there is at least one L-space surgery for any knot \(K\subset S^3\), namely the \(\infty\)-surgery. 
This analogy between rational A-link fillings and L-space surgeries goes much further:

\begin{theorem}\label{thm:A-links}
	A knot \(K\subset S^3\) admits a non-trivial L-space surgery if and only if its double tangle \(T_K\) admits a non-trivial rational A-link filling.
\end{theorem}
We compare the relevant filling slopes in \cref{thm:comparison:fillings}.

\subsection{Consequences and questions: growth under cabling}\label{subsec:intro:growth}
An open question about the Heegaard Floer tangle invariant \(\gamma_T\) (in fact, any of the multicurve invariants now existing in low-dimensional topology) is how to interpret the number of their components.  
Thanks to the \ref{thm:main:intro}, double tangles \(T=T_K\) are the first family of tangles for which we can give a fairly satisfying answer:

\begin{corollary}
	For any knot \(K\subset S^3\), 
	\[
	\pushQED{\qed} 
	\#\{\text{connected components in \(\gamma_{T_K}\)}\}
	=
	\dim\HFKhat(K).
	\qedhere
	\popQED
	\]
\end{corollary}

This has the following consequence for the growth of knot Floer homology of cable knots \(C_{2,2t+1}(K)\) of winding number 2.


\begin{proposition}\label{prop:growth:HFK:knots}
	Given a knot \(K\subset S^3\), let \(d=\dim\HFKhat(K)\) and let \(\overline{\ell}\) be the average torsion order of \(\HFKminus(K)\), i.e.\ the arithmetic mean of the integers \(\ell_i\) in \cref{thm:structure:HFKminus,thm:main:FUreformulation}.
	Then for any \(t\in\Z\),
	\[
	\dim \HFKhat(C_{2,2t+1}(K))
	=
	2(d-1)\overline{\ell}+|2t+1-4\tau(K)|.
	\]
	In particular,
	\[
	2 (d-1) \ellmax
	+
	|2t+1-4\tau(K)|
	\geq
	\dim \HFKhat(C_{2,2t+1}(K))
	\geq
	2d-1,
	\]
	where \(\ellmax\) is the maximum torsion order of \(\HFKminus(K)\). 
\end{proposition}

The average torsion order \(\overline{\ell}\) may be arbitrarily large, as demonstrated by the fact that for any positive integer \(n\), the value of \(\overline{\ell}\) for the torus knot \(K=T_{n,n+1}\) is equal to \(\tfrac{n}{2}\).
Usually, however, \(\overline{\ell}\) is very close to~1. 
In fact, among all the knots up to 15 crossings, the average value of \(\overline{\ell}\) is about 1.00077 \cite{hanselman_program,hf2hft}.
Its maximum value is 2 among those knots (realized by \(13n_{4639}=C_{2,5}(T_{2,3})\) and \(15n_{41185}=T_{4,5}\)). 
So one might interpret \cref{prop:growth:HFK:knots} as saying that \(\dim \HFKhat(C_{2,2t+1}(K))\) depends approximately linearly on \(d=\dim\HFKhat(K)\).
This is in stark contrast to the non-linear behaviour of reduced Khovanov homology:

\begin{proposition}\label{prop:growth:Khr}
	Given a knot \(K\subset S^3\), let \(d=\dim \Khr(K)\). 
	Then for any \(t\in\Z\), 
	\[
	\dim \Khr(C_{2,2t+1}(K))
	\geq
	2d^2 - 2 + |2t+1-2\vartheta_2(K)|,
	\]
	where \(\vartheta_2(K)\) is the concordance homomorphism introduced in \cite{LZ}. 
\end{proposition}

This result is basically an application of the link splitting spectral sequence in Khovanov homology of Batson and Seed \cite{BatsonSeed}. 
Computations suggest that this behaviour extends to the number of components in the reduced Khovanov multicurve invariant \(\Khr(T_K)\) of double tangles \cite{KWZ}. 
We will explore the following conjecture in upcoming work. 

\begin{conjecture}\label{conj:Kh:double-tangle}
	For any knot \(K\subset S^3\), the number of special components in \(\Khr(T_K)\) is bigger or equal to \(\tfrac{1}{2}(d^2-1)\), where \(d\) is the dimension of the reduced Khovanov homology of \(K\). 
	For two-bridge knots, this bound is sharp. 
\end{conjecture}

On a perhaps more speculative note, the decomposition of the multicurve \(\gamma_K\) into curve segments might also to be relevant in relation to the Khovanov multicurve invariant \(\Khr(T)\), specifically in view of the spectral sequence \cite{OSspectralsequence}. 

\begin{observation}
	For all strongly invertible knots \((K,h)\) with up to 9 crossings, the number of special components of \(\Khr(T_{K,h})\) is bounded below by the number of curve segments \(\capR_\ell\) in \(\gamma_K\). 
	In fact we have equality for all such knots, except \(9_{46}\), the only such knot with two strong inversions \(h_1\) and \(h_2\) for which the (ungraded) Khovanov multicurve invariants do not agree \cite{KWZ_strong_inversions}.
\end{observation}

It would be interesting to compute the Heegaard Floer multicurve invariant \(\gamma_T\) of quotient tangles of strongly invertible knots to see if the identity of the tangle invariants in \cref{exa:main:intro} is part of a general pattern. 

\subsection*{Outline of the paper}
In \cref{sec:conventions}, we recall the construction of the multicurve invariants \(\gamma_K\) and \(\gamma_{T_K}\) and we discuss gradings in both settings. 
In \cref{sec:BS:computations}, we compute the bordered sutured type~AD bimodule that turns the type~D structure corresponding to the knot complement into the type~D structure for the corresponding double tangle.
In \cref{sec:proof:main_thm}, we interpret the action of this bimodule in terms of multicurves, proving the \ref{thm:main:intro}. 
In \cref{sec:thinness,sec:growth}, we prove the results from \cref{subsec:intro:thinness,subsec:intro:growth}, respectively. 
Throughout this paper, we will assume familiarity with bordered and bordered sutured Heegaard Floer homology \cite{LOT,ZarevThesis} as well as basic techniques for modifying chain complexes over categories within their homotopy class; see for instance \cite[Section~1]{pqMod}.

\subsection*{Acknowledgements}
This paper was motivated by joint work with Lukas Lewark \cite{LZ}.
I am very grateful to him for his generous support over the past two years and for giving me the freedom to pursue my own research as his postdoc. 
I would also like to thank Jake Rasmussen, Liam Watson, and Paul Wedrich for helpful conversations about the behaviour of knot Floer and Khovanov homology under cabling. 
Finally, I thank Jen Hom and Biji Wong for their helpful questions and comments after a talk I gave about this work at the low-dimensional topology workshop in Budapest in March 2023, which motivated me to explore some further applications of the Main Theorem and which allowed me to improve the exposition.
During the course of this work, I was funded through the Emmy Noether Programme of the DFG, project number 412851057, and an individual research grant of the DFG, project number 505125645.

\section{Review and conventions}\label{sec:conventions}

\subsection{%
	\texorpdfstring{%
		A more detailed review of the multicurve invariant \(\bm{\gamma_K}\)}{%
		A more detailed review of the multicurve invariant γ\_Κ
}}

The invariant \(\gamma_K\) is a topological interpretation of a certain bordered Heegaard Floer invariant.
We start with a description of the latter and then explain how it is related to \(\gamma_K\). 

\subsubsection*{The bordered structure on the knot exterior}
The exterior \(X_K\coloneqq S^3\smallsetminus\interior(\nu(K))\) of an oriented knot \(K\subset S^3\) is a three-manifold with torus boundary. 
On this torus \(\partial X_T\), we fix two simple closed curves, namely a meridian \(\mu_K\) and a longitude \(\lambda_K\) of \(K\), that intersect in a single point \(z\), which we may regard as a basepoint on \(\partial X_K\). 
The longitude is a push-off of \(K\), inheriting its orientation. 
The orientation of the meridian is determined by the right-hand rule. 
This decoration of \(\partial X_K\) can be turned into a bordered structure as follows:
Around \(z\), we add a small circle that intersects \(\mu_K\) and \(\lambda_K\) minimally and which one may think of as a suture. 
The restriction of \(\mu_K\) to the complement of the interior of the disc bounded by the suture is an oriented arc, which we denote by \(\mathbf{a}_{\bullet}\). 
Similarly, \(\lambda_K\) gives rise to an oriented arc \(\mathbf{a}_{\circ}\) with ends on the suture. 
We now add a basepoint to the suture; by convention, it sits immediately between the endpoints of \(\mathbf{a}_{\bullet}\) and \(\mathbf{a}_{\circ}\). 
This is illustrated at the bottom left of \cref{fig:executive-summary}. 
The arc diagram \(\Za\) consisting of the suture together with this basepoint and the arcs \(\mathbf{a}_{\bullet}\) and \(\mathbf{a}_{\circ}\) embedded on \(\partial X_K\) is a bordered structure for \(X_K\). 
Now the algebraic invariant corresponding to the multicurve \(\gamma_K\) is the type~D structure \(\CFD(X_K)^{\algT}\) associated with \(X_K\) equipped with this bordered structure.
For the construction of \(\CFD(X_K)^{\algT}\), we refer the reader to \cite{LOT}; we only recall its formal properties, following \cite[Section~11]{LOT}.

\subsubsection*{The torus algebra}
We define \(\CFD(X_K)^{\algT}\) as a graded right type~D structure over the bordered algebra \(\mathcal{A}(\Za,1)\) with one moving strand, which we identify with the quiver algebra 
\[
\mathcal{A}
\coloneqq
\F\left.\left[
	\begin{tikzcd}[column sep=2cm]
	\bullet
	\arrow[bend left ]{r}[description]{\sigma_1}
	\arrow[leftarrow ]{r}[description]{\sigma_2}
	\arrow[bend right]{r}[description]{\sigma_3}
	&
	\circ
	\end{tikzcd}
\right]
\right/
\Big(
\sigma_1\sigma_2
=
0
=
\sigma_2\sigma_3
\Big).
\]
We denote the idempotents in \(\mathcal{A}\) by \(\iota_\bullet\) and \(\iota_\circ\), respectively. 
We also write \(\sigma_{12}\coloneqq \sigma_2\sigma_1\), \(\sigma_{23}\coloneqq \sigma_3\sigma_2\), and \(\sigma_{123}\coloneqq \sigma_3\sigma_2\sigma_1\).
The algebra \(\mathcal{A}\) is graded by some non-commutative groups, which we will discuss later. 
For the moment, we will also ignore the grading on \(\CFD(X_K)\), which takes values in the set of certain cosets of these groups.

\begin{remark}\label{rem:torus_algebra}
	We have renamed the elements of the algebra \(\algT\), usually denoted by the letter \(\rho\) instead of \(\sigma\), to avoid confusion with the elements \(p_i\) of the algebra \(\algS\) that we will introduce below. 
	More importantly, the invariant of \(X_K\) was originally defined as a \emph{left} type~D structure \(^{\algT^{op}}\CFD(X_K)\) over the opposite algebra \(\algT^{op}\) of \(\algT\) \cite{LOT,HRW}. 
	Our conventions have the advantage that the multiplication in the torus algebra \(\algT\) is read from right to left, like morphisms in a category. 
	In fact, we usually consider \(\algT\) as a category with two objects and we regard type~D structures simply as chain complexes over this category; see \cite[Section~1]{pqMod}.
\end{remark}

\begin{example}\label{exa:first_complex}
	For the unknot \(U\) and the right-handed trefoil knot \(T_{2,3}\), we have
	\[
	\CFD(X_U)^{\algT}
	\cong
	\left[
	\begin{tikzcd}[inner sep=2pt]
	\bullet
	\arrow[in=-30,out=30,looseness=5]{rl}{\sigma_{12}}
	\end{tikzcd}
	\right]
	\quad
	\text{and}
	\quad
	\CFD(X_{T_{2,3}})^{\algT}
	\cong
	\left[
	\begin{tikzcd}[inner sep=2pt]
	\bullet
	\arrow{d}[swap]{\sigma_{1}}
	&
	\circ
	\arrow{l}[swap]{\sigma_{2}}
	&
	\bullet
	\arrow{l}[swap]{\sigma_{3}}
	\arrow{d}{\sigma_{1}}
	\\
	\circ
	&
	&
	\circ
	\\
	&
	\circ
	\arrow{ul}{\sigma_{23}}
	&
	\bullet
	\arrow{l}{\sigma_{3}}
	\arrow{u}[swap]{\sigma_{123}}
	\end{tikzcd}
	\right].
	\]
	These labelled graphs should be understood as follows: 
	Each vertex represents a generator/basis element of the underlying chain module of the chain complex. 
	Each arrow corresponds to a non-zero entry of the square matrix representing the differential when written with respect to the chosen basis. 
	The entries of this square matrix are the elements of \(\algT\) specified by the arrow labels. 
\end{example}
	
The complexes in the above example have the remarkable property that each of them arises from an immersed curve on a one-punctured torus via the following general construction.

\subsubsection*{From curves to chain complexes}
Let \(\Sigma\) be a one-punctured torus \(S^1\times S^1\smallsetminus\interior(D^2)\), decorated by two arcs \(\mathbf{a}_{\bullet}\) and \(\mathbf{a}_{\circ}\), just like \(\partial X_K\) above. 
In fact, if \(K\) is given, we identify \(\Sigma\) with \(\partial X_K\) minus a disc bounded by the suture. 
Let \(\gamma\) be a collection of non-contractible and primitive immersed curves \(S^1\looparrowright\Sigma\) that intersects \(\mathbf{a}_{\bullet}\) and \(\mathbf{a}_{\circ}\) minimally and transversely. 
For the second complex in the above example, \(\gamma\) is the blue curve \(\gamma_K\) on the bottom left of \cref{fig:executive-summary}; for the first complex, the curve agrees with the longitude of \(U\). 
Each point in \(\gamma\cap\mathbf{a}_{\bullet}\) corresponds to a generator in idempotent \(\iota_\bullet\), and similarly, each point in \(\gamma\cap\mathbf{a}_{\circ}\) corresponds to a generator in idempotent \(\iota_\circ\). 
The arcs cut \(\gamma\) into curve segments, each of which corresponds to an arrow (= a component of the differential) in the complex, connecting the two generators for the ends of the curve segment. 
To determine the direction and labels of these arrows, we first need to fix a labelling of some paths on \(\partial\Sigma\). 
Up to reparametrization, there are four elementary paths on \(\partial\Sigma\), i.e.\ paths that start and end on \(\partial(\mathbf{a}_\bullet\cup\mathbf{a}_\circ)\), that follow the induced orientation on \(\partial\Sigma\), and whose interiors avoid \(\partial(\mathbf{a}_\bullet\cup\mathbf{a}_\circ)\).  
By convention, we label three of these consecutive paths by \(\sigma_1\), \(\sigma_2\), and \(\sigma_3\) as shown in \cref{fig:executive-summary}. 
Now, every curve segment of \(\gamma\) is homotopic (relative to \(\mathbf{a}_\bullet\cup\mathbf{a}_\circ\)) to either \(\sigma_1\), \(\sigma_2\), or \(\sigma_3\), or one of their compositions, and each curve segment inherits its orientation and label from this path on \(\partial\Sigma\). 

Chain complexes over \(\mathcal{A}\) that (up to chain homotopy) arise from this construction and that can be appropriately graded (see below) are called \emph{loop-type} \cite{HW1}. 
However, as we will see in a moment, not every chain complex is loop-type. 

\subsubsection*{Curves with local systems}\label{local_systems}
Let \(\gamma\) be a curve in the torus as above, but consisting of a single component only. 
Also, let \(n\) be some positive integer and let \((C,d)\) be a chain complex that arises from \(\gamma\) using the previous construction. 
Then \((C\otimes \mathbb{F}_2^n,d\otimes\id_{\mathbb{F}_2^n})\) is also a well-defined chain complex. 
We now modify its differential as follows:
First, we fix an orientation of \(\gamma\). 
Next, let \(x\) and \(y\) be two basis elements for \((C,d)\) such that the differential \(d\) has a component \(a\) from \(x\) to \(y\). 
In the new chain complex, this component corresponds to a component \(a\otimes\id_{\mathbb{F}_2^n}\). 
We now change this component to \(a\otimes X\) or \(a\otimes X^{-1}\), depending on certain orientation conventions, where \(X\in\mathrm{GL}_n(\F)\).
One can show that the chain homotopy type of the resulting chain complex is independent of the choice of basis elements \(x\) and \(y\) (see the remarks at the end of the proof of \cite[Theorem~1.5]{HRW}).
We can therefore denote this chain complex simply by \(C(\gamma,X)\). 
We call \(X\) the local system of the oriented curve~\(\gamma\). 
For \(n>0\), the local systems \(X=\id_{\mathbb{F}_2^n}\) are called trivial.

\begin{definition}\label{def:equivalence:local_systems}
Given a collection \(\{(\gamma_i,X_i)\}_i\) of curves with local systems, we define the associated chain complex as the direct sum of the chain complexes \(C(\gamma_i,X_i)\). 
Two collections \(\{(\gamma_i,X_i)\}_i\) and \(\{(\gamma'_j,X'_j)\}_j\) of curves with local systems are considered equal (and can be shown to induce the same chain complexes up to homotopy) if for any curve~\(\gamma\), the matrices
\[
\bigoplus_{\gamma=\gamma_i} X_i \oplus \bigoplus_{-\gamma=\gamma_i} X^{-1}_i
\quad
\text{and}
\quad
\bigoplus_{\gamma=\gamma'_j} X'_j \oplus \bigoplus_{-\gamma=\gamma'_j} X'^{-1}_j
\]
are similar. 
Here, \(-\gamma\) denotes the curve \(\gamma\) with the opposite orientation and two oriented immersed curves are considered equal if they are homotopic. 
\end{definition}

\subsubsection*{Classification result}
Chain complexes arising from curves with local systems are \emph{extendable}, which is a certain algebraic property that is immaterial to our discussion. 
We only need to know the following result of Hanselman-Rasmussen-Watson \cite[Theorem~1.4]{HRW}, which is essentially due to Lipshitz-Ozsváth-Thurston \cite[Chapter~11]{LOT}.  

\begin{theorem}
	For every knot \(K\subset S^3\), \(\CFD(X_K)^{\algT}\) is extendable. 
\end{theorem}

This is relevant in the context of the following classification result of Hanselman-Rasmussen-Watson
\cite[Theorem~1.5]{HRW}.

\begin{theorem}\label{thm:classification:HF}
	Every extendable graded chain complex over \(\mathcal{A}\) is chain homotopic to a chain complex corresponding to a collection of curves with (possibly trivial) local systems. 
	This collection of curves with local systems is unique up to the equivalence from \cref{def:equivalence:local_systems}. 
\end{theorem}

Combining these two results, we can now finally define:

\begin{definition}
	For any knot \(K\subset S^3\), let \(\gamma_K\) denote the collection of immersed curves with local systems on \(\partial X_K\) corresponding to \(\CFD(X_K)^{\algT}\).
\end{definition}

\begin{definition}\label{def:local_systems}
	Given a collection \(\gamma=\{(\gamma_i,X_i)\}_i\) of curves with local systems, we define the curve segments \(\Sset(\gamma)\) as the set of (graded) type~D structures obtained as follows:
	Let \(D_i\) be the type~D structure corresponding to the curve \(\gamma_i\). 
	Split \(D_i\) along each generator in idempotent \(\bullet\) like so:
	\[
	(\cdots\longleftrightarrow \bullet \longleftrightarrow  \cdots)
	\quad\text{becomes}\quad
	(\cdots\longleftrightarrow \bullet)\quad(\bullet \longleftrightarrow \cdots ).
	\] 
	Let \(\{D_{i,j}\}_{j}\) be the connected components of the resulting type~D structure. 
	We now define \(\Sset(\gamma)\) as the set of type~D structures that for each pair \((i,j)\) contains \(\dim X_i\) copies of \(D_{i,j}\). 
\end{definition}

	Clearly, the elements of \(\Sset(\gamma)\) correspond to the curve segments \(\capR_\ell\), \(\capL_\ell\), and \(\ccon_\ell\) from the introduction.

\subsection{%
	\texorpdfstring{%
		A more detailed review of the multicurve invariant \(\bm{\gamma_T}\)}{%
		A more detailed review of the multicurve invariant γ\_T
}}

Like the multicurve \(\gamma_K\), the tangle invariant \(\gamma_{T_K}\) is a topological interpretation of a certain bordered (in fact: bordered sutured) Heegaard Floer invariant. 
But before describing any of this, we fix some conventions regarding Conway tangles, largely following
\cite[Section~2.1]{pqMod} and \cite[Section~2.1]{LZ}.

\subsubsection*{Conway tangles}
An \textit{oriented Conway tangle} \(T\) is a proper embedding of two intervals into a three-dimensional ball \(B^3\), considered up to ambient isotopy. 
A \textit{framing} of \(B^3\) is a circle \(\Phi\) on \(S^2=\partial B^3\) together with a choice of four points on this circle that are consecutively labelled \(\I\), \(\II\), \(\III\), and \(\IV\). 
We consecutively label the four arcs \(\Phi\smallsetminus\{\I,\II,\III,\IV\}\) by \(\aaa\), \(\aab\), \(\aac\), and \(\aad\), starting with the arc between \(\I\) and \(\II\).
A \textit{framed} oriented Conway tangle is a Conway tangle together with a choice of framing of \(B^3\) which identifies the endpoints of the two intervals with the four points \(\I\), \(\II\), \(\III\), and \(\IV\). 
We consider framed tangles up to ambient isotopy which identifies the framings. 
In the following, we will refer to framed, oriented Conway tangles simply as \textit{Conway tangles}. 

The \textit{double tangle} \(T_K\) of an oriented knot \(K\) is the Conway tangle that is defined as follows: 
Let \(L\) be the oriented link obtained as the union of \(K\) with one of its zero-framed longitudes, regarded as a small push-off of \(K\). 
Then \(L\) bounds an embedded (unoriented) annulus in \(S^3\), namely the image of a suitably chosen homotopy between \(K\) and the longitude.
Now choose a small open three-ball \(b\) which intersects this annulus in a trivially embedded band. 
Then \(T_K\) is defined as the intersection of \(L\) with the complement \(B^3\) of \(b\). 
We choose the framing of \(T_K\) that is (up to homotopy and up to interchanging the two components of \(T_K\)) uniquely characterized by the following properties:
\begin{enumerate}
	\item the intersection of \(\partial B^3\) with the band is equal to the arcs \(\aaa\) and \(\aac\), 
	\item in the closure of the three-ball \(b\), the restriction of \(L\) and the arcs \(\aab\) and \(\aad\) cobound two discs (ensuring that the linking number of the two strands in \(T_K\) vanishes),
	and 
	\item\label{enu:orientation} the orientation of the tangle points out of \(B^3\) at the tangle ends \(\I\) and \(\II\).
\end{enumerate}
In particular, if follows that 
\begin{enumerate}[resume]
\item\label{enu:connectivity} the tangle connects the tangle end \(\I\) with \(\IV\) and \(\II\) with \(\III\). 
\end{enumerate}
A Conway tangle satisfying \eqref{enu:connectivity} is said to have \textit{horizontal connectivity}.
We implicitly assume from now that all Conway tangles that we consider satisfy \eqref{enu:orientation} and~\eqref{enu:connectivity}.

\subsubsection*{The bordered sutured structure on the tangle exterior}
The framing \(\Phi\) of a Conway tangle \(T\) enables us to distinguish between the two components of \(\partial B^3\smallsetminus \Phi\): The \textit{front} is the component where the labelling of the points \(\I,\II,\III,\IV\) increases in counter-clockwise direction when viewed from outside \(B^3\); 
we call the other component the \textit{back} of \(T\). 
On both front and back, we now choose a small disc whose boundary we may consider as a suture.
We connect the two discs by four arcs \(\ba\), \(\bb\), \(\bc\), and \(\bd\) that each intersect the circle \(\Phi\) only once, namely within the arcs \(\aaa\), \(\aab\), \(\aac\), and \(\aad\), respectively. 
As in the case of the suture on the knot complement \(X_K\) in the previous subsection, we label the elementary paths on the front (resp.\ back) suture by \(p_i\) (resp.\ \(q_i\)) if they are adjacent to the tangle end labelled \(\mathtt{i}\). 
We now place basepoints on the elementary paths labelled \(p_4\) and \(q_3\). 
Finally, we remove a sufficiently small open tubular neighbourhood of the tangle \(T\) that avoids the arcs \(\ba\), \(\bb\), \(\bc\), and \(\bd\) and the sutures. 
We denote the result by \(X_T\coloneqq B^3\smallsetminus\interior(\nu(T))\), the exterior of the tangle \(T\). 
Since \(T\) is a tangle of horizontal connectivity, the basepoints lie in different components of the complement of the sutures and arcs on \(\partial X_T\). 
Hence the arc diagram \(\Zb\) consisting of the sutures with basepoints and the \(\beta\)-arcs \(\ba\), \(\bb\), \(\bc\), and \(\bd\) embedded on \(\partial X_T\) gives rise to a well-defined bordered sutured structure on \(X_T\). 
This is illustrated in \cref{fig:doubling:T_K}. 
Now the algebraic invariant corresponding to the multicurve \(\gamma_{T}\) is the associated type~D structure \(\BSD(X_T)^{\algS}\). 
For the construction of \(\BSD(X_T)^{\algS}\), we refer the reader to \cite{ZarevThesis}; we only recall its formal properties, following \cite{pqMod,pqSym} and specializing to the case that the tangle~\(T\) satisfies \eqref{enu:orientation} and~\eqref{enu:connectivity}.

\subsubsection*{The peculiar algebra}
\(\BSD(X_{T})^{\algS}\) is a graded chain complex over the bordered sutured algebra \(\mathcal{A}(\Zb,1)\) with one moving strand, which can be identified with the quiver algebra 
\[
\algS
\coloneqq
\mathcal{A}^\partial_{43}
=
\F\left.\left[
\begin{tikzcd}
&
\arcD
\arrow[bend right=10,swap]{dl}{q_1}
\\
\arcA
\arrow[bend right=10,swap]{ur}{p_1}
\arrow[bend right=10,swap]{dr}{q_2}
&
&
\arcC
\arrow[swap]{ul}{q_4}
\arrow{dl}{p_3}
\\
&
\arcB
\arrow[bend right=10,swap]{ul}{p_2}
\end{tikzcd}
\right]
\right/
\Big(
p_iq_i
=
0
=
q_ip_i
\Big)_{i=1,2}.
\]
We denote the idempotents in \(\algS\) by \(\iota_a\), \(\iota_b\), \(\iota_c\), and \(\iota_d\), respectively. 
We also write \(p_{12}\coloneqq p_1p_2\), \(p_{23}\coloneqq p_2p_3\), \(p_{123}\coloneqq p_1p_2p_3\), \(q_{21}\coloneqq q_2q_1\), \(q_{32}\coloneqq q_3q_2\), and \(q_{321}\coloneqq q_3q_2q_1\).
This algebra carries two gradings, a \(\hZ\)-grading \(\delta\) determined by
\[
\delta(p_1)
=
\delta(p_2)
=
\delta(p_3)
=
\half
=
\delta(q_1)
=
\delta(q_2)
=
\delta(q_4)
\] 
and a \(\Z^2\)-grading, called the Alexander grading, determined by 
\[
-A(p_1)
=
-A(q_1)
=
(1,0)
=
A(q_4)
\quad
-A(p_2)
=
-A(q_2)
=
(0,1)
=
A(p_3).
\] 
The chain complex \(\BSD(X_{T})^{\algS}\) is equipped with a \(\delta\)-grading, which increases along the differential by 1, and an Alexander grading, which is preserved along the differential. 
We sometimes denote elements in this bigrading by \(\HFKgr{r}{a_1,a_2}\), where \(r\in\hZ\) is the \(\delta\)-grading and \((a_1,a_2)\in\Z^2\) is the Alexander grading.  
We sometimes indicate the bigrading of a generator \(x\) using super- and subscripts: \(\HFTgr{r}{a_1}{a_2}{x}\).
Given a bigraded chain complex \(C\) over \(\algS\), let \(\delta^r t_1^{a_1} t_2^{a_2} C\) denote the chain complex obtained from \(C\) by increasing the Alexander and \(\delta\)-gradings of all generators by \((a_1,a_2)\) and \(r\), respectively.

\subsubsection*{Classification result}
To describe the classification of graded chain complexes over \(\algS\), we introduce a decoration of \(\partial X_T\) which is in some sense ``dual'' to the bordered sutured structure described above. 
We consider the tangle exterior with a meridional suture around each tangle end. 
The sutures decompose \(\partial X_T\) into a four-punctured sphere \(S^2_4\) and two annuli.  
The four arcs \(\aaa\), \(\aab\), \(\aac\), and \(\aad\) of the framing \(\Phi\) restrict to four arcs on \(\partial X_T\), which we denote by the same combination of letters. 
We label the elementary path on the suture corresponding to the tangle end \(\mathtt{i}\) along the front (resp.\ back) by \(p_i\) (resp.\ \(q_i\)). 
Finally, we place basepoints on the elementary paths labelled \(p_4\) and \(q_3\). 
This is illustrated on the right of \cref{fig:executive-summary}. 
(Note that with minor modifications, this can be turned into an actual bordered sutured structure:  First, one adds basepoints to the sutures at the tangle ends \(\I\) and \(\II\); second, one also adds two arcs around those basepoints such that the corresponding gluing surface remains a four-punctured sphere; for details, see the proof of \cite[Theorem~3.7]{pqMod}. 
The simpler bordered sutured structure on \(X_T\) that we use in this paper was first introduced in \cite[Section~5]{pqSym}.)

As for curves on the one-punctured torus and complexes over the algebra \(\algT\), every non-contractible and primitive immersed curve (with or without local system) on the four-punctured sphere \(S^2_4\) gives rise to a chain complex over \(\algS\). 
Chain complexes arising from curves with local systems are again \emph{extendable} and we have the following result, which is apparent from \cite[Theorem~5.3]{pqSym}. 

\begin{theorem}
	\(\BSD(X_T)^{\algS}\) is extendable for every Conway tangle \(T\) satisfying \eqref{enu:orientation} and~\eqref{enu:connectivity}.
\end{theorem}

The following result is \cite[Theorem~0.4]{pqMod}.

\begin{theorem}\label{thm:classification:HFT}
	Every extendable graded chain complex over \(\algS\) is chain homotopic to a chain complex corresponding to a collection of curves on \(S^2_4\) with (possibly trivial) local systems. 
	This collection of curves is unique up to the equivalence from \cref{def:equivalence:local_systems}. 
\end{theorem}

Combining these two results, we now define:

\begin{definition}
	For any Conway tangle \(T\) satisfying \eqref{enu:orientation} and~\eqref{enu:connectivity}, let \(\gamma_T\) denote the collection of immersed curves with local systems on \(\partial X_T\) corresponding to \(\BSD(X_T)^{\algS}\).
\end{definition}

\begin{proof}[Proof of \cref{thm:structure:HFTcurve}] 
	By the gluing formula \cite[Theorem~5.9]{pqMod}, the Lagrangian Floer homology of \(\gamma_{T_K}\) with a rational curve of slope \(\infty\) (a simple closed curve lifting to a straight line of slope \(\infty\) in the planar cover) computes \(\F^2\otimes\HFKhat(U)\cong \F^2\).
	Combining this with the classification of components \cite[Theorem~0.5]{pqSym} and detection of tangle connectivity \cite[Observation~6.1]{pqMod}, we see that \(\gamma_{T_K}\) consists of exactly one rational component of even integer slope with the trivial one-dimensional local system and a (potentially vanishing) number of special components of slope \(\infty\) with trivial local systems. 
	The second part follows from \cite[Theorem~0.10]{pqSym}.
\end{proof}

\subsection{Gradings}

Recall from \cite{ZarevThesis} that in bordered (sutured) Heegaard Floer theory, the unreduced grading \(\gr\) (aka unrefined grading \(\gr'\) in the terminology of \cite{LOT}) of the algebra \(\mathcal{A}(\mathcal{Z})\) associated with an arc diagram \(\mathcal{Z}=(\mathbf{Z},\mathbf{a},M)\) takes values in the non-commutative group \(\Gr(\mathcal{Z})\), which is defined as a certain semi-direct product \(\hZ\rtimes H_1(\mathbf{Z},\mathbf{a})\). 
For convenience, we enlarge the grading set by working with rational coefficients, i.e.\ we replace \(\Gr(\mathcal{Z})\) by \(\Gr(\mathcal{Z})\otimes\Q=\Q\rtimes H_1(\mathbf{Z},\mathbf{a};\Q)\). 
There is also a reduced (refined) grading \(\rgr\) of \(\mathcal{A}(\mathcal{Z})\), which takes values in \(\rGr(\mathcal{Z})\), which we define as \(\Q\rtimes H_1(\mathcal{F}(\mathcal{Z});\Q)\), where \(\mathcal{F}(\mathcal{Z})\) is the surface obtained by thickening the arc diagram \(\mathcal{Z}\). 
By default, we work with the gradings associated with type~A sides. 

\subsubsection*{Identification with the usual grading on the knot invariant}
From \cite[Section~11.1]{LOT}, recall the reduced grading  
\[
\rGr(\Za)
\coloneqq
\Q\rtimes H_1(\FZa;\Q) 
\cong
\Q\rtimes\Q^2
\]
with the group law defined by 
\[
(j_1;p_1,q_1)\cdot (j_2;p_2,q_2)
=
\Big(
j_1+j_2-\det
\left(\!\!
\begin{array}{cc}
p_1 & q_1\\
p_2 & q_2
\end{array}\!\!
\right)
;
p_1+p_2,
q_1+q_2
\Big)
\]
for any \((j_i;p_i,q_i)\in \rGr(\Za)\), \(i=1,2\). 
(When comparing this formula with \cite[Section~11.1]{LOT}, note that the minus sign in front of the determinant comes from the fact that our algebra \(\mathcal{A}\) is the opposite algebra of the usual torus algebra; see \cref{rem:torus_algebra}.)
The grading on \(\mathcal{A}\) is defined by
\[
\rgr(\sigma_1)= (-\half;\half,-\half)
\qquad
\rgr(\sigma_2)= (-\half;\half,\half)
\qquad
\rgr(\sigma_3)= (-\half;-\half,\half)
\]
The grading on \(\CFD(X_K)^{\algT}\) takes values in the coset \(\PI\backslash\rGr(\Za)\), where \(\PI\) is the subgroup generated by \((-\half;-1,0)\). 
Moreover, it suffices to specify the grading of the generators in idempotent \(\iota_\bullet\), which are in bijection with the (\(\Q^2\)-graded) generators of \(\HFKhat(K)\). 
(These two facts are explained in \cite[Theorem~11.26 and Remark~11.28]{LOT}, noting that by \cref{thm:classification:HF}, the twisting parameter \(n\) can be chosen to be 0.)
The first \(\Q\)-grading on \(\HFKhat(K)\) is called the \(\delta\)-grading and the second \(\Q\)-grading is called the Alexander grading \(A\). 
We will write \(\HFKgr{m}{n}\) for an element \((m,n)\in\Q^2\) of this bigrading. 
(Often, the Maslov grading \(M\) is considered instead of the \(\delta\)-grading; one can be determined from the other using the identity \(M+\delta=A\).)
Then if \(\mathbf{x}\) is a (homogeneous) generator of \(\HFKhat(K)\) corresponding to a generator \(\mathbf{x}_0\) of \(\CFD(X_K).\iota_\bullet\), then
\begin{align*}
\rgr(\mathbf{x}_0)
&=
(M(\mathbf{x})-\tfrac{3}{2}A(\mathbf{x});0,-A(\mathbf{x}))
\\
&=
(-\delta(\mathbf{x})-\half A(\mathbf{x});0,-A(\mathbf{x}))
\in 
\PI\backslash\rGr(\Za).
\end{align*}
Given a bigraded chain complex \(C\) over \(\algT\) and \(r,a\in\Q\), let \(\delta^r t^{a} C\) denote the chain complex obtained from \(C\) by increasing the \(\delta\)- and Alexander gradings of all generators by \(r\) and \(a\), respectively.

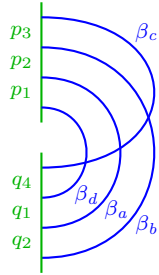
\begin{figure}[b]
	\centering
	\begin{tikzpicture}[scale=0.4, thick]
	\footnotesize
	\draw[darkgreen] (0,0.5) -- (0,4.5);
	\draw[blue] (0,1) .. controls  +(2,0) and +(2,0) .. (0,-2);
	\draw[blue] (0,2) .. controls  +(3.5,0) and +(3.5,0) .. (0,-3);
	\draw[blue] (0,3) .. controls  +(5,0) and +(5,0) .. (0,-4);
	\draw[blue] (0,4) .. controls  +(5,0) and +(5,0) .. (0,-1);
	\draw[darkgreen] (0,-0.5) -- (0,-4.5);
	\draw (0, 3.5) node[left,darkgreen] {\(p_3\)};
	\draw (0, 2.5) node[left,darkgreen] {\(p_2\)};
	\draw (0, 1.5) node[left,darkgreen] {\(p_1\)};
	\draw (0,-1.5) node[left,darkgreen] {\(q_4\)};
	\draw (0,-2.5) node[left,darkgreen] {\(q_1\)};
	\draw (0,-3.5) node[left,darkgreen] {\(q_2\)};

	\draw (1.5,-1.9) node[blue] {\(\bd\)};
	\draw (2.5,-2.4) node[blue] {\(\ba\)};
	\draw (3.5,-2.9) node[blue] {\(\bb\)};
	\draw (3.5,3.5)  node[blue] {\(\bc\)};
	\end{tikzpicture}
	\caption{The arc diagram \(\Zb\) for the bordered sutured structure on \(X_{T}\)
	}\label{fig:arcdiargam:T_K}
\end{figure}

\subsubsection*{Identification with the usual grading on the tangle invariant}
Let \(T\) be a Conway tangle satisfying \eqref{enu:orientation} and~\eqref{enu:connectivity}.
The arc diagram \(\Zb=(\mathbf{Z}_\beta,\mathbf{a}_\beta,M_\beta)\) on the tangle exterior \(X_T\) is shown in \cref{fig:arcdiargam:T_K}. 
The surface \(\FZb\) is a four-punctured sphere; in particular, it is planar. 
Hence, the intersection pairing on \(H_1(\FZb)\) is trivial and \(\rGr(\Zb)\) is actually identical to \(\Q\times H_1(\FZb;\Q)\) as a group. 
To determine the reduced grading on the algebra \(\algS=\mathcal{A}(\Zb)\), we choose the paths \(\{p_1,p_2,p_3,q_2,q_1,q_4\}\) as a basis for \(H_1(\mathbf{Z}_\beta,\mathbf{a}_\beta;\Q)\) so that
the unreduced gradings of the algebra generators are as follows:
\begin{align*}
\gr(p_1)
&=
(-\half;1,0,0;0,0,0),
&
\gr(q_2)
&=
(-\half;0,0,0;1,0,0),
\\
\gr(p_2)
&=
(-\half;0,1,0;0,0,0),
&
\gr(q_1)
&=
(-\half;0,0,0;0,1,0),
\\
\gr(p_3)
&=
(-\half;0,0,1;0,0,0),
&
\gr(q_4)
&=
(-\half;0,0,0;0,0,1).
\end{align*}
Recall from \cite[Definition~1.7]{pqSym} the usual grading group \(\mathfrak{A}\) for \(\algS\), which we will replace by \(\mathfrak{A}\otimes\Q\): 
\[
\mathfrak{A}
\coloneqq
\Q^4/\im(\Q\rightarrow\Q^4, n\mapsto(n,n,n,n)).
\]
We identify \(H_1(\FZb;\Q)\) with \(\mathfrak{A}\) 
via the homomorphism from \(\mathfrak{A}\) to \(H_1(\FZb;\Q)\) defined by
\begin{align*}
(1,0,0,0)
&\mapsto
(\half,0,0;0,\half,0),
&
(0,1,0,0)
&\mapsto
(0,\half,0;\half,0,0),
\\
(0,0,1,0)
&\mapsto
(0,0,\half;-\half,-\half,-\half),
&
(0,0,0,1)
&\mapsto
(-\half,-\half,-\half;0,0,\half).
\end{align*}
We now choose a grading reduction \(r\) with \(\iota_a\) as the base idempotent (so \(r(\iota_a)=0\)) and
\begin{align*}
r(\iota_b)
&=
(0;0,-\half,0;\half,0,0),
\\
r(\iota_c)
&=
(-\half;0,-\half,-\half;0,-\half,-\half),
\\
r(\iota_d)
&=
(0;\half,0,0;0,-\half,0).
\end{align*}
Given a basic algebra element \(\xi=\iota_t.\xi.\iota_s\in\algS\) for some \(t,s\in\{\arcA,\arcB,\arcC,\arcD\}\), the reduced grading is computed using the formula 
\begin{equation}\label{req:grading_reduction}
\rgr(\xi)=(r(\iota_t))^{-1}\gr(\xi) ~r(\iota_s).
\end{equation}
An elementary calculation shows that the induced reduced grading is given by
\begin{align*}
\rgr(p_1)
=
\rgr(q_1)
&
=
(-\half;1,0,0,0),
&
\rgr(p_3)
&=
(-\half;0,0,1,0),
\\
\rgr(p_2)
=
\rgr(q_2)
&=
(-\half;0,1,0,0),
&
\rgr(q_4)
&=
(-\half;0,0,0,1).
\end{align*}
(To determine \(\rgr(p_{3})\) and \(\rgr(q_{4})\), it might be easier to compute \(\rgr(p_{23})\) and \(\rgr(q_{14})\) first.)
Finally, since the tangle \(T\) satisfies \eqref{enu:orientation} and~\eqref{enu:connectivity}, we simplify the grading further via the homomorphism \(\varphi\co\mathfrak{A}\rightarrow \Q^2\) sending
\begin{align*}
(1,0,0,0)&\mapsto (-1,0), 
&
(0,0,0,1)&\mapsto (1,0),
\\
(0,1,0,0)&\mapsto (0,-1),
&
(0,0,1,0)&\mapsto (0,1).
\end{align*}
Then the gradings \((\delta,A)\) and \(((-\id_{\Q})\times\varphi)\circ\rgr\) on the algebra \(\algS\) agree. 

The grading on \(\BSD(X_T)\) takes values in \(\mathcal{P}_2\backslash\rGr(\Zb)\), where \(\mathcal{P}_2\) is the subgroup \(\rgr(\pi_2(\mathbf{x},\mathbf{x};\Q))\) defined as follows: 
First we fix a Heegaard diagram for \(X_T\) and we pick some base generator \(\mathbf{x}\) of \(\BSD(X_T).\iota_{a}\). 
(This is always possible.)
We now consider \(\pi_2(\mathbf{x},\mathbf{x};\Q)\), the group of periodic domains on a Heegaard diagram for \(X_T\) (with coefficients in \(\Q\)).
Each domain \(B\) has a certain unreduced grading \(\gr(B)\). 
Let \(R(\gr(B))\) be the element of \(\Gr(\Zb)\) obtained from \(\gr(B)\) by multiplying all components except the first by \(-1\). 
Then we define \(\rgr(B)\) as the element of \(\rGr(\Zb)\) corresponding to \(R(\gr(B))\) using the above grading reduction~\(r\). 
\(\mathcal{P}_2\) is now the subgroup generated by the elements \(\rgr(B)\) for all periodic domains \(B\in\pi_2(\mathbf{x},\mathbf{x};\Q)\).
To understand this subgroup better, we choose a Heegaard diagram with a single \(\alpha\)-circle parallel to the front suture. 
(This is always possible.)
The periodic domains of such a Heegaard diagram are generated by a single domain~\(\psi\).
An elementary computation (similar to the proof of \cite[Lemma~5.12]{HDsForTangles}) shows that the reduced grading of \(\psi\) is equal to \(\pm(0;1,0,0,1)\). 
Hence, we can use the above identification of the gradings on the algebra \(\algS\) to also identify the gradings on \(\BSD(X_T)\). 
(Alternatively, in the case that \(T\) is a double tangle \(T_K\) for some knot \(K\), this identification will also emerge from the proof of \cref{thm:BS_computations}).

\section{Computations in bordered sutured Heegaard Floer theory}\label{sec:BS:computations}

The goal of this section is to prove the following result.

\begin{theorem}\label{thm:BS_computations}
	Let \(X_K\) be the complement of an open tubular neighbourhood of a knot \(K\) in \(S^3\).
	We equip \(\partial X_K\) with the standard bordered parametrization by meridian and longitude. 
	Then
	\[
	\BSD(X_{T_K})^{\algS}
	\cong
	\CFD(X_K)^{\algT}
	\boxtimes
	\BSADY,
	\]
	where \(\BSADY\) is the type~AD bimodule shown in \cref{fig:doubling-bimodule}. 
	The reduced grading of \(\BSADY\) takes values in \((\rGr(\Za)\times\mathfrak{A})/\PIII\) where \(\PIII\) is the subgroup generated by 
	\[
	(0;0,0;0,1,1,0)
	\quad\text{and}\quad
	(-\half;0,1;0,0,2,2).
	\]
	Moreover, if \(\rgr(\arcA)=\PIII\) then \(\rgr(\arcC)=(-\half;-1,0;0,0,0,0) \cdot\PIII\). 
\end{theorem}
\begin{figure}[b]
	\centering
	\begin{tikzpicture}[xscale=1.7,yscale=1.5]
	\draw[gb] (-1.25,-1.35) -- (-1.25,1.35) -- (1.25,1.35) -- (1.25,-1.35) -- cycle;
	\draw[gd] (-3.7,0.5) -- %
						(-2.75,0.5) -- (-2,-1.5) -- (2,-1.5) -- (2.75,0.5) -- %
						( 3.7,0.5) -- %
						(3.7,-2.8) -- %
						(-3.7,-2.8) -- %
						cycle;
	\draw[gb,fill] (0,-1.35) circle [radius=1pt];
	\draw[gb,fill=white] (0,-1.5) circle [radius=1pt];
	
	\draw (-1,1)   node(lt){\textcolor{darkgreen}{\(c\)}};
	\draw (-1,-1)   node(lb){\textcolor{red}{			\(a\)}};
	\draw (1,1)   node(rt){\textcolor{blue}{			\(b\)}};
	\draw (1,-1)   node(rb){\textcolor{blue}{			\(b'\)}};
	
	\draw (-3,0)   node(LT){\textcolor{blue}{			\(B\)}};
	\draw (-3,-2)   node(LB){\textcolor{gold}{		\(D\)}};
	\draw (3,0)   node(RT){\textcolor{blue}{			\(B'\)}};
	\draw (3,-2)   node(RB){\textcolor{gold}{			\(D'\)}};
	
	\footnotesize
	\draw[->] (lt) [bend right=25] to node[t]{\((\sigma_1|p_3)\)} (LT);
	\draw[->] (rt) [bend left=25] to node[t]{\((\sigma_1|1)\)} (RT);
	
	\draw[->] (lt) to node[t]{\((-|q_{214})\)} (rt);
	\draw[->] (lb) to node[t]{\((-|q_2)\)} (rb);
	
	\draw[->] (LB) [bend left ] to node[t]{\((-|q_{21})\)} (LT);
	\draw[->] (RB) [bend right] to node[t]{\((-|q_{21})\)} (RT);
	
	\draw[->] (lb) [bend left =40] to node[t]{\((\sigma_3|p_1)\)} (LB);
	\draw[->] (rb) [bend right=40] to node[t]{\((\sigma_3|p_{12})\)} (RB);
	
	\draw[->] (LT) [bend left]  to node[t]{\((\sigma_{23}|p_{12})\)} (LB);
	\draw[->] (RT) [bend right] to node[t]{\((\sigma_{23}|p_{12})\)} (RB);
	
	\draw[->] (LT) [bend left] to node[t,pos=0.35] {\((\sigma_2|p_2)\)} (lb);
	\draw[->] (RT) [bend right] to node[t,pos=0.35]{\((\sigma_2|1)\)} (rb);
	
	\draw[->] (lt) [bend left=25]  to node[t]{\parbox{1.5cm}{\centering\((\sigma_{12}|p_{23})\)\\+\\\((\sigma_{12}|q_{14})\)}} (lb);
	\draw[->] (rt) [bend right=25] to node[t] {\((\sigma_{12}|1)\)} (rb);
	
	\draw[->] (lt) [bend left=20] to node[t,pos=0.55] {\((\sigma_{123}|p_{123})\)} (LB);
	\draw[->] (rt) [bend right=20] to node[t,pos=0.55] {\((\sigma_{123}|p_{12})\)} (RB);
	
	\draw (0.4,-0.65)   node [t](x3){\((\sigma_1|q_4)\)};
	\draw (lt) .. controls  +(0.6,-0.3) and +(-0.3,0.6) .. (x3);
	\draw[->] (x3) .. controls  +(1,-2) and +(-0.6,-0.5) .. (RB);

	\draw (-0.4,-1.8)   node [t](x4){\((\sigma_2|q_1)\)};
	\draw[<-] (lb) .. controls  +(0.15,-0.3) and +(-0.1,0.1) .. (x4);
	\draw (x4) .. controls  +(1,-1) and +(-0.5,-1) .. (RB);
	
	
	\end{tikzpicture}
	\vspace{-20pt}
	\caption{The type~AD bimodule \(\typeA{\algT}{\mathcal{Y}}^{\algS}\). 
		The right idempotents of the generators are determined by the letter of the generator names; for example, the right idempotent of the generator named \textcolor{blue}{\(B'\)} is \(\iota_b\). 
		Generators named by a lower-case letter sit in the left idempotent \(\iota_\bullet\), those named by an upper-case letter belong to the left idempotent \(\iota_\circ\). }\label{fig:doubling-bimodule}
\end{figure}
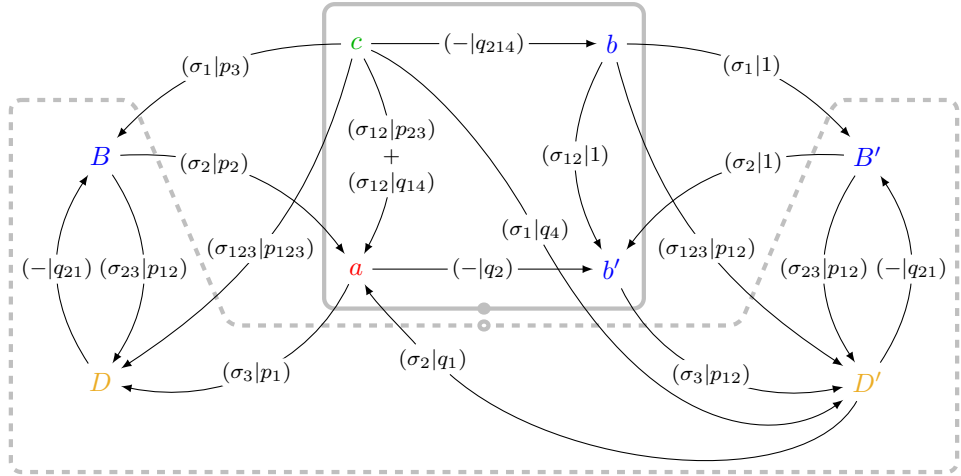

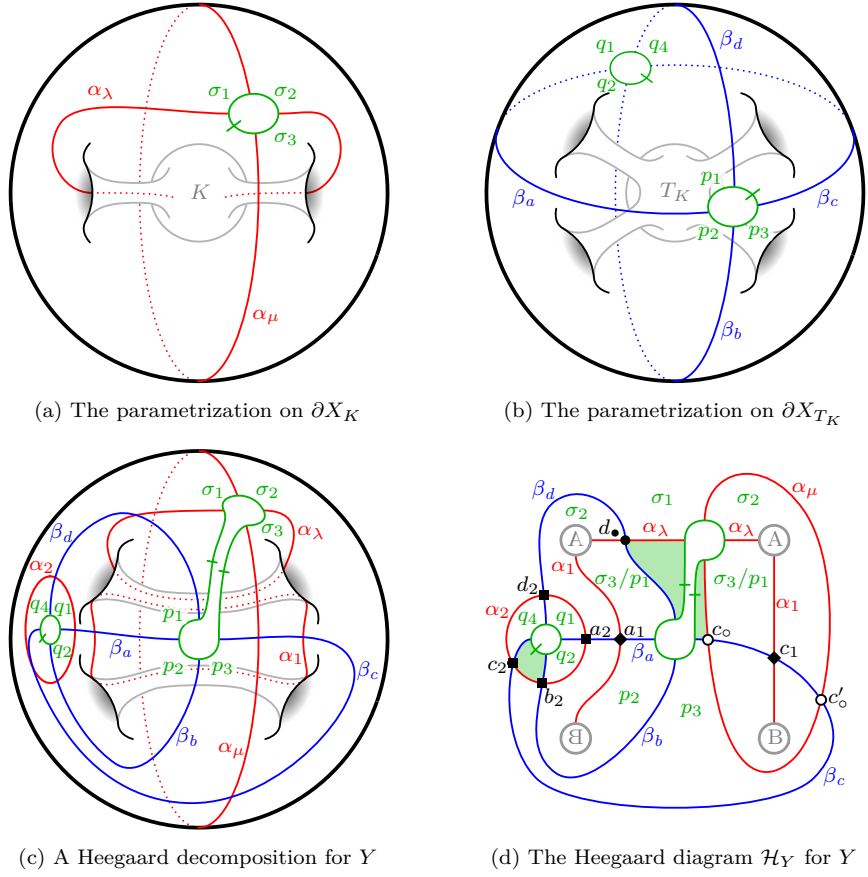
\begin{figure}[p]
	\centering\footnotesize
	\begin{subfigure}{0.5\textwidth}
		\centering
		\labellist
		\pinlabel \(\textcolor{gray}{K}\) at 145 145
		\pinlabel \(\textcolor{red}{\am}\) at 195 50
		\pinlabel \(\textcolor{red}{\al}\) at 70 220
		\pinlabel \(\textcolor{darkgreen}{\sigma_1}\) at 160 220
		\pinlabel \(\textcolor{darkgreen}{\sigma_2}\) at 210 220
		\pinlabel \(\textcolor{darkgreen}{\sigma_3}\) at 210 185
		\endlabellist
		\includegraphics[width=0.8\textwidth]{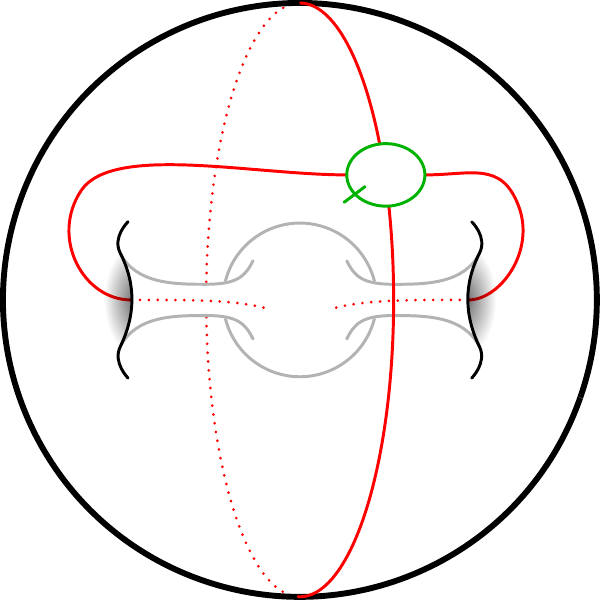}
		\caption{The parametrization on \(\partial X_K\)}\label{fig:doubling:X_K}
	\end{subfigure}%
	\begin{subfigure}{0.5\textwidth}		
		\centering
		\labellist
		\setlength{\fboxsep}{1pt}
		\pinlabel \(\textcolor{gray}{T_K}\) at 145 145
		\pinlabel \(\textcolor{blue}{\ba}\) at 30 140
		\pinlabel \(\textcolor{blue}{\bb}\) at 187 40
		\pinlabel \(\textcolor{blue}{\bc}\) at 260 140
		\pinlabel \(\textcolor{blue}{\bd}\) at 187 260
		\pinlabel \mybg{darkgreen}{\(p_1\)} at 175 155
		\pinlabel \mybg{darkgreen}{\(p_2\)} at 173 113
		\pinlabel \mybg{darkgreen}{\(p_3\)} at 211 114
		\pinlabel \mybg{darkgreen}{\(q_1\)} at 95 255
		\pinlabel \textcolor{darkgreen}{\(q_2\)} at 93 225
		\pinlabel \mybg{darkgreen}{\(q_4\)} at 135 255
		\endlabellist
		\includegraphics[width=0.8\textwidth]{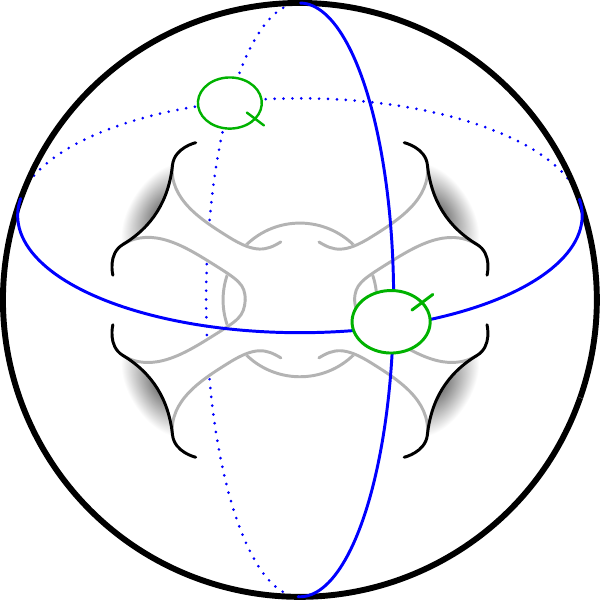}
		\caption{The parametrization on \(\partial X_{T_K}\) }\label{fig:doubling:T_K}
	\end{subfigure}
	\bigskip\\
	\begin{subfigure}{0.5\textwidth}
		\centering
		\labellist
		\setlength{\fboxsep}{1pt}
		\pinlabel \(\textcolor{blue}{\ba}\) at 85 137
		\pinlabel \(\textcolor{blue}{\bb}\) at 135 70
		\pinlabel \(\textcolor{blue}{\bc}\) at 272 120
		\pinlabel \(\textcolor{blue}{\bd}\) at 40 225
		\pinlabel \(\textcolor{red}{\am}\) at 169 60
		\pinlabel \(\textcolor{red}{\al}\) at 229 225
		\pinlabel \textcolor{red}{\(\alpha_1\)} at 214 130
		\pinlabel \textcolor{red}{\(\alpha_2\)} at 25 199
		\pinlabel \mybg{darkgreen}{\(p_1\)} at 128 162
		\pinlabel \mybg{darkgreen}{\(p_2\)} at 128 122
		\pinlabel \mybg{darkgreen}{\(p_3\)} at 165 122
		\pinlabel \textcolor{darkgreen}{\(q_1\)} at 42 167
		\pinlabel \textcolor{darkgreen}{\(q_2\)} at 41 135
		\pinlabel \textcolor{darkgreen}{\(q_4\)} at 24 169
		\pinlabel \textcolor{darkgreen}{\(\sigma_1\)} at 155 256
		\pinlabel \textcolor{darkgreen}{\(\sigma_2\)} at 196 256
		\pinlabel \textcolor{darkgreen}{\(\sigma_3\)} at 200 227
		\endlabellist
		\includegraphics[width=0.8\textwidth]{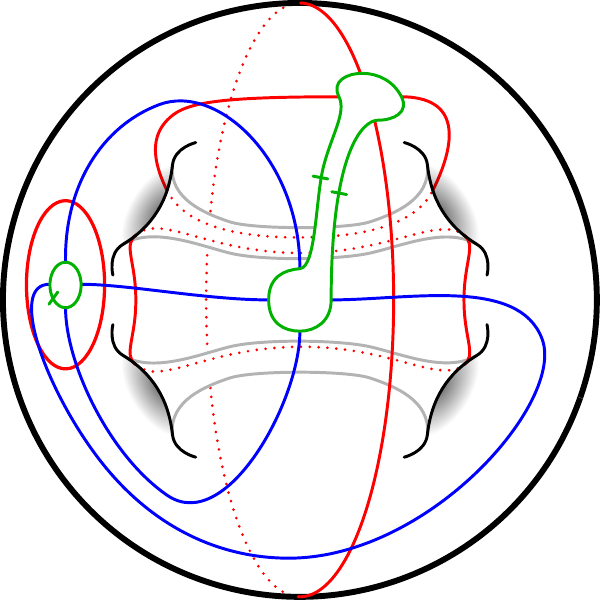}
		\caption{A Heegaard decomposition for \(Y\)}\label{fig:doubling:Y}
	\end{subfigure}%
	\begin{subfigure}{0.5\textwidth}
		\centering
		\labellist
		\setlength{\fboxsep}{1pt}
		\pinlabel \textcolor{gray}{A} at 220 220
		\pinlabel \reflectbox{\textcolor{gray}{A}} at 69 220
		\pinlabel \textcolor{gray}{B} at 220 70
		\pinlabel \reflectbox{\textcolor{gray}{B}} at 69 70
		\pinlabel \(\textcolor{blue}{\ba}\) at 120 134
		\pinlabel \(\textcolor{blue}{\bb}\) at 127 70
		\pinlabel \(\textcolor{blue}{\bc}\) at 265 40
		\pinlabel \(\textcolor{blue}{\bd}\) at 45 257
		\pinlabel \(\textcolor{red}{\am}\) at 243 257
		\pinlabel \(\textcolor{red}{\al}\) at 195 227
		\pinlabel \(\textcolor{red}{\al}\) at 130 227
		\pinlabel \textcolor{red}{\(\alpha_1\)} at 230 170
		\pinlabel \textcolor{red}{\(\alpha_1\)} at 60 198
		\pinlabel \textcolor{red}{\(\alpha_2\)} at 10 165
		\pinlabel \textcolor{darkgreen}{\(\sigma_2\)} at 70 240
		\pinlabel \textcolor{darkgreen}{\(\sigma_2\)} at 200 250
		\pinlabel \textcolor{darkgreen}{\(\sigma_3/p_1\)} at 195 190
		\pinlabel \textcolor{darkgreen}{\(\sigma_3/p_1\)} at 105 190
		\pinlabel \textcolor{darkgreen}{\(p_2\)} at 110 100
		\pinlabel \textcolor{darkgreen}{\(p_3\)} at 157 90
		\pinlabel \textcolor{darkgreen}{\(q_1\)} at 62 160
		\pinlabel \textcolor{darkgreen}{\(q_2\)} at 62 130
		\pinlabel \textcolor{darkgreen}{\(q_4\)} at 33 160
		\pinlabel \textcolor{darkgreen}{\(\sigma_1\)} at 135 250
		\pinlabel \(a_1\) at 114 150
		\pinlabel \(a_2\) at 88 150
		\pinlabel \(b_2\) at 53 100
		\pinlabel \(c_\circ\) at 180 150
		\pinlabel \(c'_\circ\) at 268 100
		\pinlabel \(c_1\) at 231 135
		\pinlabel \(c_2\) at 10 123
		\pinlabel \(d_2\) at 33 185
		\pinlabel \(d_\bullet\) at 95 228
		\endlabellist
		\includegraphics[width=0.8\textwidth]{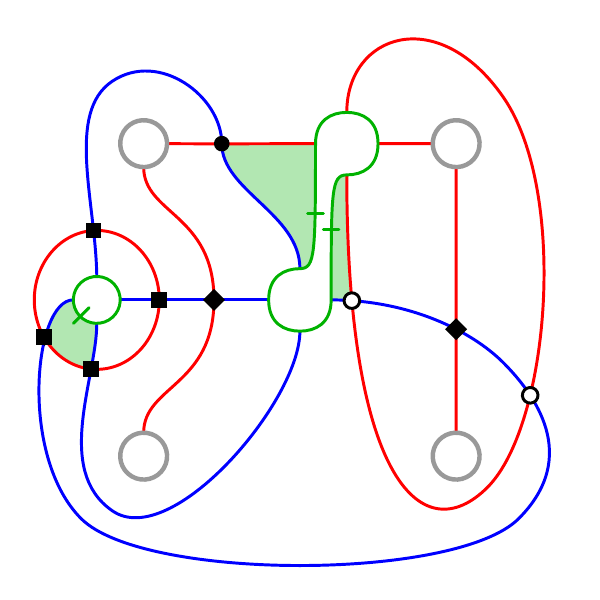}
		\caption{The Heegaard diagram \(\mathcal{H}_Y\) for \(Y\)}\label{fig:doubling:HD}
	\end{subfigure}
	\caption{The tangle complement \(X_{T_K}\) (b) is obtained from the knot complement \(X_K\) (a) by gluing the bordered sutured manifold \(Y\) (c) to it.  Figure (d) shows a Heegaard diagram from which we compute the type~AD bimodule \(\typeA{\algT}{\mathcal{Y}}^{\algS}\).}
\label{fig:doubling}
\end{figure}

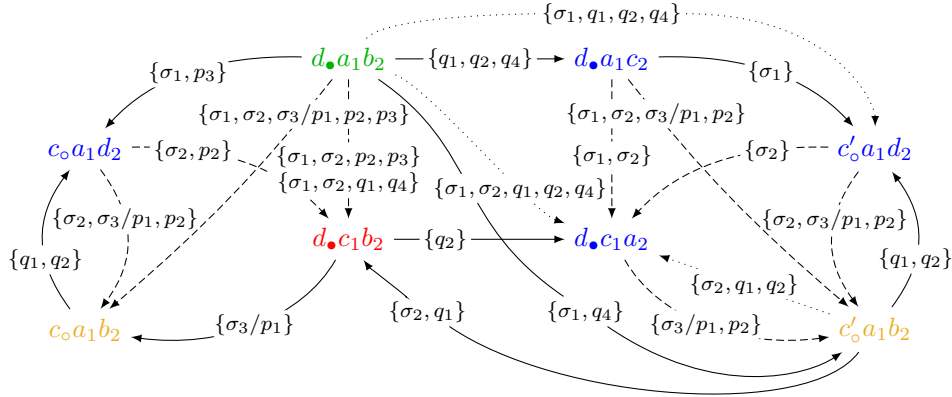
\begin{figure}[p]
	\centering
	\begin{tikzpicture}[xscale=1.75,yscale=1.2]
	\draw (-1,1)   node(lt){\textcolor{darkgreen}{\(d_\bullet a_1b_2\)}};
	\draw (-1,-1)   node(lb){\textcolor{red}{			\(d_\bullet c_1b_2\)}};
	\draw (1,1)   node(rt){\textcolor{blue}{			\(d_\bullet a_1c_2\)}};
	\draw (1,-1)   node(rb){\textcolor{blue}{			\(d_\bullet c_1a_2\)}};
	
	\draw (-3,0)   node(LT){\textcolor{blue}{			\(c_\circ  a_1d_2\)}};
	\draw (-3,-2)   node(LB){\textcolor{gold}{		\(c_\circ  a_1b_2\)}};
	\draw (3,0)   node(RT){\textcolor{blue}{			\(c'_\circ a_1d_2\)}};
	\draw (3,-2)   node(RB){\textcolor{gold}{			\(c'_\circ a_1b_2\)}};
	
	\footnotesize
	\draw[->] (lt) [bend right] to node[t]{\(\{\sigma_1,p_3\}\)} (LT);
	\draw[->] (rt) [bend left] to node[t]{\(\{\sigma_1\}\)} (RT);
	
	\draw[->] (lt) to node[t]{\(\{q_1,q_2,q_4\}\)} (rt);
	\draw[->] (lb) to node[t,pos=0.3]{\(\{q_2\}\)} (rb);
	
	\draw[->] (LB) [bend left =27] to node[t,pos=0.35]{\(\{q_1,q_2\}\)} (LT);
	\draw[->] (RB) [bend right=27] to node[t,pos=0.35]{\(\{q_1,q_2\}\)} (RT);
	
	\draw[->			 ] (lb) [bend left =40] to node[t]{\(\{\sigma_3/p_1\}\)} (LB);
	\draw[->,d] (rb) [bend right=40] to node[t]{\(\{\sigma_3/p_1,p_2\}\)} (RB);
	
	\draw[->,d] (LT) [bend left=27] to node[t,pos=0.35]{\(\{\sigma_2,\sigma_3/p_1,p_2\}\)} (LB);
	\draw[->,d] (RT) [bend right=27] to node[t,pos=0.35] {\(\{\sigma_2,\sigma_3/p_1,p_2\}\)} (RB);
	
	\draw[->,d] (LT) [bend left] to node[t,pos=0.25] {\(\{\sigma_2,p_2\}\)} (lb);
	\draw[->,d] (RT) [bend right] to node[t,pos=0.25]{\(\{\sigma_2\}\)} (rb);
	
	\draw[->,d] (lt) to node[t,pos=0.65]{\parbox{2.05cm}{\centering\(\{\sigma_1,\sigma_2,p_2,p_3\}\)\\ \(\{\sigma_1,\sigma_2,q_1,q_4\}\)}} (lb);
	\draw[->,d] (rt) to node[t,pos=0.55] {\(\{\sigma_1,\sigma_2\}\)} (rb);
	
	\draw[->,d] (lt) [bend left=5] to node[t,pos=0.13] {\(\{\sigma_1,\sigma_2,\sigma_3/p_1,p_2,p_3\}\)} (LB);
	\draw[->,d] (rt) [bend right=5] to node[t,pos=0.13] {\(\{\sigma_1,\sigma_2,\sigma_3/p_1,p_2\}\)} (RB);

	\draw (0.8,-1.8)   node [t](x3){\(\{\sigma_1,q_4\}\)};
	\draw (lt) .. controls  +(1,-1) and +(-1,1) .. (x3);
	\draw[->] (x3) .. controls  +(1,-1) and +(-0.5,-0.5) .. (RB);
	
	\draw (-0.4,-1.8)   node [t](x4){\(\{\sigma_2,q_1\}\)};
	\draw[<-] (lb) .. controls  +(0.15,-0.3) and +(-0.1,0.1) .. (x4);
	\draw (x4) .. controls  +(1,-1) and +(-0.5,-1) .. (RB);
	
	\draw[->,dotted] (RB) to node[t]{\(\{\sigma_2,q_1,q_2\}\)} (rb);
	
	\draw[->,dotted] (lt) .. controls  +(0.6,-0.3) and +(-0.6,0.3) .. node[t,pos=0.7] {\(\{\sigma_1,\sigma_2,q_1,q_2,q_4\}\)} (rb);
	
	\draw (1,1.5)   node [t](x2){\(\{\sigma_1,q_1,q_2,q_4\}\)};
	\draw[dotted,-] (lt) .. controls  +(0.3,0.5) and +(-0.5,0) .. (x2);
	\draw[dotted,->] (x2) .. controls  +(1.7,0) and +(0,0.5) .. (RT);
	
	\end{tikzpicture}
	\vspace{-0.5cm}
	\caption{Domains connecting generators}\label{fig:doubling:domains}
\end{figure}
	\begin{figure}[p]\small
		\begin{tabular}{ccccccccc}
			\toprule
			domain \(B\)
			&
			\(e(B)\)
			&
			\(n_{\mathbf{x}}(B)\)
			&
			\(n_{\mathbf{y}}(B)\)
			&
			\(\#\vec{\bm{\rho}}_1\)
			&
			\(\#\vec{\bm{\rho}}_2\)
			&
			\(\iota(\bm{\rho}_1)\)
			&
			\(\iota(\bm{\rho}_2)\)
			&
			\(\operatorname{ind}(B,\vec{\bm{\rho}}_1,\vec{\bm{\rho}}_2)\)
			\\\midrule
			\(\{\sigma_1,q_1,q_2,q_4\}\)
			&
			\(-\tfrac{1}{2}\)
			&
			\(\tfrac{1}{2}\)
			&
			\(1\)
			&
			\(1\)
			&
			\(1\)
			&
			\(-\tfrac{1}{2}\)
			&
			\(-\tfrac{1}{2}\)
			&
			\textbf{2}%
			\smallskip\\
			\(\{\sigma_2,q_1,q_2\}\)
			&
			\(-1\)
			&
			\(\tfrac{3}{4}\)
			&
			\(\tfrac{5}{4}\)
			&
			\(1\)
			&
			\(1\)
			&
			\(-\tfrac{1}{2}\)
			&
			\(-\tfrac{1}{2}\)
			&
			\textbf{2}%
			\smallskip\\
			\(\{\sigma_1,\sigma_2,q_1,q_2,q_4\}\)
			&
			\(-\tfrac{3}{2}\)
			&
			\(1\)
			&
			\(\tfrac{3}{2}\)
			&
			\(1\)
			&
			\(\geq1\)
			&
			\(-\tfrac{1}{2}\)
			&
			\(-\tfrac{1}{2}\)
			&
			\(\mathbf{\geq2}\)
			\\
			\bottomrule
		\end{tabular} 
		\caption{Computation of the embedded indices of the domains on the dotted arrows in \cref{fig:doubling:domains} following \cite[Definition~8.4.2]{ZarevThesis}. 
			Here, \(B\in\pi_2(\mathbf{x},\mathbf{y})\) denotes a domain connecting two generators \(\mathbf{x}\) and \(\mathbf{y}\). 
			The symbols \(\vec{\bm{\rho}}_1\) and \(\vec{\bm{\rho}}_2\) denote sets of Reeb chords on the type~D side and type~A side, respectively.
		}
		\label{tab:index}
	\end{figure}
	\begin{figure}[p]
		\small
		\newcommand{\tpl}{2.9cm}
		\begin{tabular}{cccc}
			\toprule
			\(i\)
			&
			\(\mathcal{T}^{\algT}=\mathcal{T}_i^{\algT}\)
			&
			\(
			\mathcal{S}^{\algS}
			=
			\mathcal{T}^{\algT}
			\boxtimes
			\typeA{\algT}{\BSAD(\mathcal{H}_Y)}^{\algS}
			\)
			&
			\(\gamma_{Q_{-2i}}\)
			\\\midrule
			\(0\)
			&
			\begin{tikzcd}
			\bullet
			\arrow[in=30,out=-30,looseness=5]{rl}[swap]{\sigma_{12}}
			\end{tikzcd}
			&
			\(
			\vc{%
				\begin{tikzpicture}[ppscale]
				\draw (-1,1)   node(lt){\textcolor{darkgreen}{\(c\)}};
				\draw (-1,-1)   node(lb){\textcolor{red}{			\(a\)}};
				\draw (1,1)   node(rt){\textcolor{blue}{			\(b\)}};
				\draw (1,-1)   node(rb){\textcolor{blue}{			\(b'\)}};
				\draw [g] (rt) to (rb);
				\draw (1,1)   node(rt){\textcolor{blue}{			\(b\)}};
				\draw (1,-1)   node(rb){\textcolor{blue}{			\(b'\)}};

				\draw[gb] (-1.35,-1.35) -- (-1.25,1.35) -- (1.25,1.35) -- (1.25,-1.35) -- cycle;

				\footnotesize
				
				\draw[->] (lt) to node[below]{\(q_{214}\)} (rt);
				\draw[->] (lb) to node[above]{\(q_2\)} (rb);
				
				\draw[->] (lt) [d,bend right=15]  to node[left]{\(p_{23}\)} (lb);
				\draw[->] (lt) [d,bend left=15]  to node[right]{\(q_{14}\)} (lb);
				
				\draw[->] (rt) [d] to node[left] {\(1\)} (rb);
				
				\end{tikzpicture}
			}
			\)
			&
			\(\vc{\includegraphics[width=\tpl]{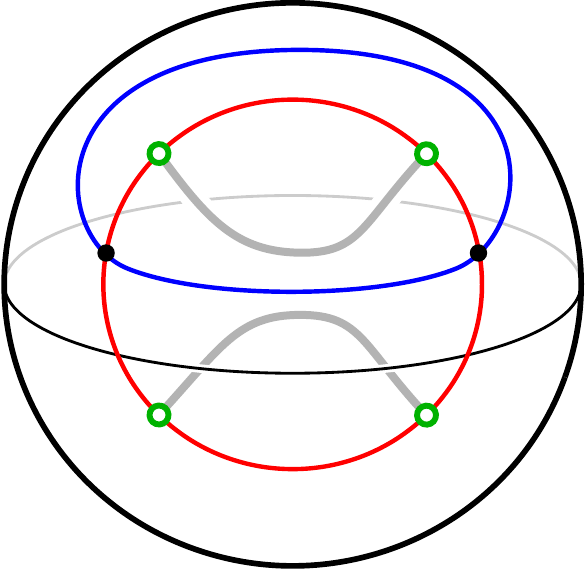}}\)
			\vspace{5pt}
			\\
			\(1\)
			&
			\begin{tikzcd}
			\bullet
			\arrow[bend left]{d}{\sigma_{123}}
			\\
			\circ
			\arrow[bend left]{u}{\sigma_{2}}
			\end{tikzcd}
			&
			\(
			\vc{%
				\begin{tikzpicture}[ppscale]
				
				\draw (-1,1)   node(lt){\textcolor{darkgreen}{\(c\)}};
				\draw (-1,-1)   node(lb){\textcolor{red}{			\(a\)}};
				\draw (1,1)   node(rt){\textcolor{blue}{			\(b\)}};
				\draw (1,-1)   node(rb){\textcolor{blue}{			\(b'\)}};
				
				\draw (-3,0)   node(LT){\textcolor{blue}{			\(B\)}};
				\draw (-3,-2)   node(LB){\textcolor{gold}{		\(D\)}};
				\draw (3,0)   node(RT){\textcolor{blue}{			\(B'\)}};
				\draw (3,-2)   node(RB){\textcolor{gold}{			\(D'\)}};
				
				\draw (RT) [g] to (rb);
				
				\draw[gb] (-1.25,-1.35) -- (-1.25,1.35) -- (1.25,1.35) -- (1.25,-1.35) -- cycle;
				
				\draw[gd] (-3.7,0.5) -- %
				(-2.75,0.5) -- (-2,-1.5) -- (2,-1.5) -- (2.75,0.5) -- %
				( 3.7,0.5) -- %
				(3.7,-2.5) -- %
				(-3.7,-2.5) -- %
				cycle;
				
				\draw (1,-1)   node(rb){\textcolor{blue}{			\(b'\)}};
				\draw (3,0)   node(RT){\textcolor{blue}{			\(B'\)}};
				
				\footnotesize
				
				\draw[->] (lt) to node[below]{\(q_{214}\)} (rt);
				\draw[->] (lb) to node[above]{\(q_2\)} (rb);
				
				\draw[->] (LB) to node[left]{\(q_{21}\)} (LT);
				\draw[->] (RB) to node[right]{\(q_{21}\)} (RT);

				\draw[->] (LT) [d] to node[above,pos=0.35] {\(p_2\)} (lb);
				\draw[->] (RT) [d] to node[above,pos=0.35]{\(1\)} (rb);
				
				\draw[->] (lt) [d,in=20,out=-135] to node[left,pos=0.1] {\(p_{123}\)} (LB);
				\draw[->] (rt) [d,in=160,out=-45] to node[right,pos=0.1] {\(p_{12}\)} (RB);
				
				\draw[->] (RB) [bend left=20] to node [below] {\(q_1\)} (lb);
				\end{tikzpicture}
			}
			\)
			&
			\(\vc{\includegraphics[width=\tpl]{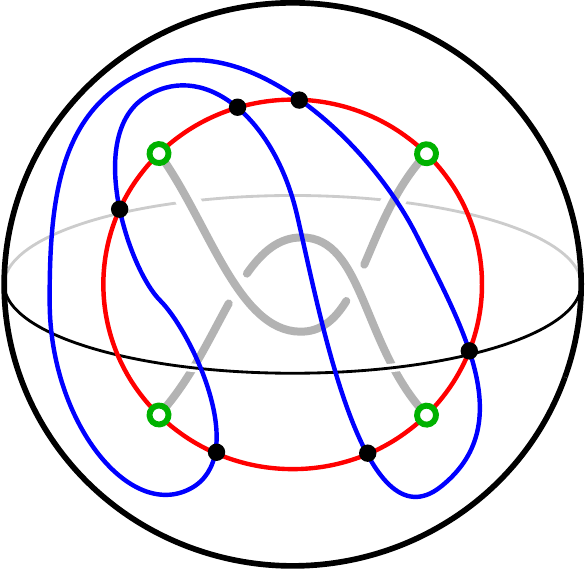}}\)
			\\
			\(-2\)
			&
			\begin{tikzcd}
			\bullet
			\arrow{d}{\sigma_{1}}
			\arrow[bend right]{dd}[swap]{\sigma_{3}}
			\\
			\circ
			\\
			\circ
			\arrow{u}[swap]{\sigma_{23}}
			\end{tikzcd}
			&
			\(
			\vc{%
				\begin{tikzpicture}[ppscale]
				\draw (-1,1)   node(lt){\textcolor{darkgreen}{\(c\)}};
				\draw (-1,-1)   node(lb){\textcolor{red}{			\(a\)}};
				\draw (1,1)   node(rt){\textcolor{blue}{			\(b\)}};
				\draw (1,-1)   node(rb){\textcolor{blue}{			\(b'\)}};
				
				\draw (-3,0)   node(LT){\textcolor{blue}{			\(B\)}};
				\draw (-3,-2)   node(LB){\textcolor{gold}{		\(D\)}};
				\draw (3,0)   node(RT){\textcolor{blue}{			\(B'\)}};
				\draw (3,-2)   node(RB){\textcolor{gold}{			\(D'\)}};
				
				\draw (-3,-4)   node(LTx){\textcolor{blue}{			\(B\)}};
				\draw (-3,-6)   node(LBx){\textcolor{gold}{		\(D\)}};
				\draw (3,-4)   node(RTx){\textcolor{blue}{			\(B'\)}};
				\draw (3,-6)   node(RBx){\textcolor{gold}{			\(D'\)}};
				
				\draw (rt) [g,bend left=25] to (RT);
				
				\draw[gb] (-1.25,-1.35) -- (-1.25,1.35) -- (1.25,1.35) -- (1.25,-1.35) -- cycle;
				
				\draw[gd] (-3.7,0.5) -- %
				(-2.75,0.5) -- (-2,-1.5) -- (2,-1.5) -- (2.75,0.5) -- %
				( 3.7,0.5) -- %
				(3.7,-2.5) -- %
				(-3.7,-2.5) -- %
				cycle;
				
				\draw[gd] (-3.7,-3.5) -- %
				( 3.7,-3.5) -- %
				(3.7,-6.5) -- %
				(-3.7,-6.5) -- %
				cycle;
				
				\draw (1,1)   node(rt){\textcolor{blue}{			\(b\)}};
				\draw (3,0)   node(RT){\textcolor{blue}{			\(B'\)}};
				
				\footnotesize
				\draw[->] (lt) [bend right=25] to node[above]{\(p_3\)} (LT);
				\draw[->] (rt) [bend left=25] to node[above]{\(1\)} (RT);
				
				\draw[->] (lt) to node[below]{\(q_{214}\)} (rt);
				\draw[->] (lb) to node[above]{\(q_2\)} (rb);
				
				\draw[->] (LB) to node[left]{\(q_{21}\)} (LT);
				\draw[->] (RB) to node[right]{\(q_{21}\)} (RT);
				
				\draw[->] (LBx) to node[left]{\(q_{21}\)} (LTx);
				\draw[->] (RBx) to node[right]{\(q_{21}\)} (RTx);
				
				\draw[->] (lb) [bend left =20  ] to node[right,pos=0.35]{\(p_1\)} (LBx);
				\draw[->] (rb) [bend right=20,d] to node[left ,pos=0.35]{\(p_{12}\)} (RBx);
				
				\draw[d,->] (LTx) to node[left]{\(p_{12}\)} (LB);
				\draw[d,->] (RTx) to node[right]{\(p_{12}\)} (RB);

				\draw[->] (lt) [out=-45,in=120] to node[pos=0.7,above]{~\(q_4\)} (RB);
				\end{tikzpicture}
			}
			\)
			&
			\(\vc{\includegraphics[width=\tpl]{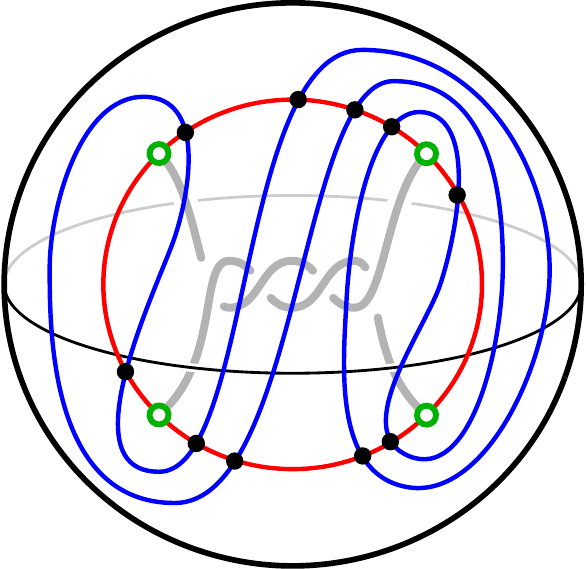}}\)
			\vspace{3pt}\\
			\bottomrule
		\end{tabular} 
		\caption{Test pairings to determine which pseudoholomorphic curves contribute to the differential in \(\typeA{\algT}{\BSAD(\mathcal{H}_Y)}^{\algS}\). 
			Specifically, the dashed arrows originate from the problematic differentials in \(\typeA{\algT}{\BSAD(\mathcal{H}_Y)}^{\algS}\). 
			Likewise, the solid arrows correspond to the unproblematic differentials. 
			The highlighted arrows labelled 1 can be cancelled using \cite[Lemma~1.22]{pqMod}.
			A comparison of the last two columns demonstrates that indeed, all dashed arrows have to be part of the differential of~\(\mathcal{S}^{\algS}\).
		}
		\label{tab:test-pairings}
	\end{figure}

\begin{proof}
	Consider the bordered sutured manifold \(Y\) shown in \cref{fig:doubling:Y}. 
	Gluing \(Y\) to the knot complement \(X_K\) shown in \cref{fig:doubling:X_K} produces the tangle complement \(X_{T_K}\) equipped with the bordered sutured structure described in \cref{sec:conventions} and shown in \cref{fig:doubling:T_K}.
	By the pairing theorem \cite[Theorem~12.3.2]{ZarevThesis}, 
	\[
	\BSD(X_{T_K})^{\algS}
	=
	\CFD(X_K)^{\algT}
	\boxtimes
	\typeA{\algT}{\BSAD(Y)}^{\algS}.
	\]
	(Here, \(\typeA{\algT}{\BSAD(Y)}^{\algS}\) denotes the direct summand of the type~AD structure \(\BSAD(Y)\) whose generators occupy exactly one of \(\am\) and \(\al\).)
	So it suffices to show that 
	\(
	\typeA{\algT}{\mathcal{Y}}^{\algS}
	\) 
	and 
	\(
	\typeA{\algT}{\BSAD(Y)}^{\algS}
	\)
	are graded chain homotopic. 
	In fact, we will show that
	\[
	\typeA{\algT}{\mathcal{Y}}^{\algS}
	=
	\typeA{\algT}{\BSAD(\mathcal{H}_Y)}^{\algS}
	\]
	for the Heegaard diagram \(\mathcal{H}_Y\) of \(Y\) shown in \cref{fig:doubling:HD}. 
	
	The type~AD structure \(\typeA{\algT}{\BSAD(\mathcal{H}_Y)}^{\algS}\) consists of eight generators in total. 
	We name the generators by the concatenation of the intersection points they are composed of. 
	They are shown as the vertices of the graph in \cref{fig:doubling:domains}. 
	The idempotents of \(\algT\) acting on the generators on the left can be read off the subscript of the first letter: 
	The four generators in middle two columns lie in the meridional idempotent \(\iota_\bullet\), the other four generators belong to the longitudinal idempotent \(\iota_\circ\). 
	The basic idempotents of \(\algS\) acting on the generators on the right can also be deduced from the name: A generator belongs to idempotent \(\iota_x\) for \(x\in\{\textcolor{red}{a},\textcolor{blue}{b},\textcolor{darkgreen}{c},\textcolor{gold}{d}\}\) if the letter \(x\) does not appear in the name.
	(This is also indicated by the colours.) 
	
	Next, we consider the regions in the Heegaard diagram. 
	As for the generators, we use suggestive notation and label them by algebra elements of \(\algS\) and \(\algT\), i.e.\ \(p_i\), \(q_j\), and \(\sigma_\ell\), or combinations thereof. 
	In the following, let us write domains \(D\) as formal differences \(D_+-D_-\) of unordered multisets of regions \(D_+\) and \(D_-\) with \(D_+\cap D_-=\varnothing\) such that 
	\[D=\sum_{r\in D_+}r-\sum_{r\in D_-}r.\]
	Let us calculate some connecting domains between the generators. 
	The solid arrows in \cref{fig:doubling:domains} indicate some \(n\)-gons with boundary punctures, which contribute to the type~AD structure. 
	Furthermore, elementary combinatorial arguments show that the group of periodic domains of \(\mathcal{H}_Y\) is freely generated by the domains
	\[
	B_1\coloneqq
	\{q_1,q_4\}-\{p_2,p_3\},
	\quad
	\text{and}
	\quad
	B_2\coloneqq
	\{\sigma_2,\sigma_3/p_1,p_2,q_1,q_2\}.
	\]
	\Cref{fig:doubling:domains} shows all connecting domains satisfying the following conditions: the multiplicity of each region is non-negative; the multiplicity of each region labelled by \(p_i\) and \(q_j\) is at most 1; the multiplicity 1 regions do not include \textit{both} those labelled by \(p_i\) and \(q_j\). 
	(This follows from routine calculations, most of which can be outsourced to a computer using the python script \cite{ancillary}.)
	These are the only domains that can possibly contribute to the differential in the type AD structure \(\typeA{\algT}{\BSAD(\mathcal{H}_Y)}^{\algS}\), but we do not know (yet) whether they do. 
	However, it turns out, this is all we need to compute from the Heegaard diagram.
	
	The domains on the three dotted arrows do not contribute to the differential \(\typeA{\algT}{\BSAD(\mathcal{H}_Y)}^{\algS}\), because the expected dimension of the corresponding moduli spaces is not 1; see \cref{tab:index}. 
	Comparing \cref{fig:doubling-bimodule,fig:doubling:domains}, it remains to see that all dashed arrows contribute to the differential. 
	We check this indirectly by pairing \(\typeA{\algT}{\BSAD(\mathcal{H}_Y)}^{\algS}\) with certain test type~D structures \(\mathcal{T}^{\algT}\) for which we already know the result 
	\[
	\mathcal{S}^{\algS}
	\coloneqq
	\mathcal{T}^{\algT}
	\boxtimes
	\typeA{\algT}{\BSAD(\mathcal{H}_Y)}^{\algS}
	\]
	up to homotopy.
	To construct such \(\mathcal{T}^{\algT}\), let \(X_K(i)\) be the knot complement equipped with the bordered structure given by the meridian \(\mu\) and the longitude \(\lambda+i\mu\), where \(\lambda\) is the homological longitude. 
	Then, setting \(K=U\) and \(\mathcal{T}=\mathcal{T}_{i}=\CFD(X_U(i))\), we know that \(\mathcal{S}^{\algS}\) is chain homotopic to \(\BSD(X_{Q_{-2i}})^{\algS}\), where \(Q_{-2i}\) is the rational tangle of slope \(-2i\). 
	\Cref{tab:test-pairings} compares the expected and the computed result for \(\mathcal{S}^{\algS}\) for some values of \(i\). 
	We conclude that all dashed arrows need to contribute, which concludes the identification of \(\typeA{\algT}{\BSAD(\mathcal{H}_Y)}^{\algS}\) with \(\BSADY\), up to gradings. 

	So it remains to consider gradings. 
	The unreduced grading of \(\typeA{\algT}{\BSAD(\mathcal{H}_Y)}^{\algS}\) takes values in \(\Q\rtimes H_1(\mathbf{Z}_\alpha,\mathbf{a}_\alpha;\Q)\times H_1(\mathbf{Z}_\beta,\mathbf{a}_\beta;\Q)\). 
	The unreduced grading \(\gr\) of the periodic domains \(B_1,B_2\in\pi_2(\textcolor{red}{d_\bullet c_1b_2},\textcolor{red}{d_\bullet c_1b_2})\) is equal to 
	\[
	\gr(B_1)=(0;0,0,0;0,-1,-1;0,1,1)
	\quad\text{and}\quad
	\gr(B_2)=(-\half;0,1,1;1,1,0;1,1,0)
	\]
	 and the unreduced grading of the domain \(B_3=\{\sigma_1,\sigma_2,p_2,p_3\}\in\pi_2(\textcolor{darkgreen}{d_\bullet a_1b_2}, \textcolor{red}{d_\bullet c_1b_2})\) is equal to
	\[
	\gr(B_3)=(0;1,1,0;0,1,1;0,0,0).
	\]
	Using the same grading reductions as in \cref{sec:conventions} (after multiplying the components of \( H_1(\mathbf{Z}_\beta,\mathbf{a}_\beta;\Q)\) by \(-1\)), the reduced gradings of these three domains is 
	\[
	\rgr(B_1)=(0;0,0;0,2,2,0)
	\quad\text{and}\quad
	\rgr(B_2)=(-\half;0,1;0,0,2,2)
	\]
	and
	\begin{align*}
	\rgr(B_3)
	&=
	(r(\iota_c))^{-1}R(\gr(B_3))
	\\
	&=
	(\half;1,1,0;0,-\half,-\half;0,\half,\half)
	=
	(\half;1,0;0,-1,-1,0);
	\end{align*}
	compare \cite[Formulas~3.44 and~10.36]{LOT} and \eqref{req:grading_reduction}. 
	Let \(\PIII=\langle\rgr(B_1),\rgr(B_2)\rangle\) and choose \(\textcolor{red}{a}\) as the base generator, i.e.\ \(\rgr(\textcolor{red}{a})=P_3\). 
	Then, since \(-B_3\in\pi_2(\textcolor{darkgreen}{c},\textcolor{red}{a})\), 
	\begin{align*} 
	\rgr(\textcolor{darkgreen}{c})
	&=
	\rgr(-B_3)\cdot\PIII
	\\
	&=
	(-\half;-1,0;0,1,1,0)\cdot\PIII
	=
	(-\half;-1,0;0,0,0,0)\cdot\PIII.
	\qedhere
	\end{align*}
\end{proof}

\section{Proof of the Main Theorem}\label{sec:proof:main_thm}

We now prove the following graded version of the \ref{thm:main:intro}.

\begin{theorem}\label{thm:main:graded}
	The components of \(\gamma_{T_K}\) are in one-to-one correspondence with the curve segments in \(\Sset(\gamma_K)\) (\cref{def:local_systems}).
	More specifically, for any \(r,a\in\hZ\) and \(\ell\in\Z^{>0}\), we have the following correspondence:
	\[
	\text{in~}\gamma_K~
	\left\{
	\begin{array}{ccc}
	\ccon_{2\tau(K)}
	&\longleftrightarrow&
	\rr_{4\tau(K)}
	\\
	\delta^r t^a\,\capR_{\ell} 
	&\longleftrightarrow&
	\delta^{r+a} t_1^{2a}t_2^{2a}\,\s_{2\ell}
	\\
	\delta^r t^{-a}\,\capL_{\ell}
	&\longleftrightarrow&
	\delta^{r+a} t_1^{-2a}t_2^{-2a}\,\sS_{2\ell}
	\end{array}
	\right\}
	~\text{in~}\gamma_{T_K}
	\]
	where the gradings on the curves/curve segments is as in \cref{fig:HFKgradings,fig:HFTgradings}. 
\end{theorem}

\begin{example}\label{exa:main:graded}
	For the right-handed trefoil knot \(K=T_{2,3}=3_1\), 
	\[
	\Sset(\gamma_K)=
	\{
	\ccon_{2},
	\delta^{\frac{3}{2}}t^{+\frac{1}{2}}\,\capR_{1},
	\delta^{\frac{3}{2}}t^{-\frac{1}{2}}\,\capL_{1}
	\};
	\]
	see for example \cite{hanselman_program}. 
	Hence
	\[
	\gamma_{T_K}
	=~
	\rr_4
	~\cup~
	\delta^{2}t_1^{+1}t_2^{+1}\,\s_{2}
	~\cup~
	\delta^{2}t_1^{-1}t_2^{-1}\,\sS_{2}.
	\]
	As one might expect \cite[Example~7.15]{KWZ_thin}, \(T_K\) is an example of a Heegaard Floer exceptional tangle, since \(\delta(\s,\rr_4)=2\) for both special components~\(\s\) of~\(\gamma_{T_K}\).
\end{example}
\begin{example}\label{exa:4_1}
	For the figure-eight knot \(K=4_1\), 
	\begin{align*}
	\Sset(\gamma_K)
	&=
	\{
	\ccon_{0},
	\delta^{\frac{1}{2}}t^{+\frac{1}{2}}\,\capR_{1},
	\delta^{\frac{1}{2}}t^{-\frac{1}{2}}\,\capR_{1},
	\delta^{\frac{1}{2}}t^{-\frac{1}{2}}\,\capL_{1},
	\delta^{\frac{1}{2}}t^{+\frac{1}{2}}\,\capL_{1}
	\},
	\\
	\gamma_{T_K}
	&=~
	\rr_0
	~\cup~
	\delta^{1}t_1^{+1}t_2^{+1}\,\s_{2}
	~\cup~
	\delta^{0}t_1^{-1}t_2^{-1}\,\s_{2}
	\\
	&
	\phantom{=~
	\rr_0 }
	~~\cup~
	\delta^{1}t_1^{-1}t_2^{-1}\,\sS_{2}
	~\cup~
	\delta^{0}t_1^{+1}t_2^{+1}\,\sS_{2}.
	\end{align*}
\end{example}
\begin{example}\label{exa:8_19}
	For the \((3,4)\)-torus knot \(K=T_{3,4}=8_{19}\), 
	\begin{align*}
	\Sset(\gamma_K)
	&=
	\{
	\ccon_{6},
	\delta^{\frac{7}{2}}t^{+\frac{5}{2}}\,\capR_{1},
	\delta^{3}t^{-1}\,\capR_{2},
	\delta^{\frac{7}{2}}t^{-\frac{5}{2}}\,\capL_{1},
	\delta^{3}t^{+1}\,\capL_{2}
	\},
	\\
	\gamma_{T_K}
	&=~
	\rr_{12}
	~\cup~
	\delta^{6}t_1^{+5}t_2^{+5}\,\s_{2}
	~\cup~
	\delta^{2}t_1^{-2}t_2^{-2}\,\s_{4}
	\\
	&
	\phantom{=~
		\rr_{12} }
	~~\cup~
	\delta^{6}t_1^{-5}t_2^{-5}\,\sS_{2}
	~\cup~
	\delta^{2}t_1^{+2}t_2^{+2}\,\sS_{4}.
	\end{align*}
\end{example}

\begin{remark}\label{rem:main:absolute_grading}
	The curve segments \(\ccon_{k}\) in \cref{fig:HFKgradings} are graded such that for any knot~\(K\),  \(\ccon_{2\tau(K)}\) is equal to an element of \(\Sset(\gamma_K)\) without any additional grading shift. 
	This is ensured by the symmetry with respect to the Alexander grading and the convention that the Maslov grading of the right generators \(\bullet\) in \cref{fig:HFKgradings:c-,fig:HFKgradings:c0,fig:HFKgradings:c+} vanishes. 
	(These are the generators that are mapped to the generators of \(\HFhat(S^3)\) under the spectral sequence corresponding to setting \(U=0\) and \(V=1\).)
\end{remark}
\begin{remark}
	The graded version of conjugation symmetry for the multicurve \(\gamma_K\) (\cref{thm:structure:HFKcurve}) corresponds precisely to the graded version of conjugation symmetry for the multicurve \(\gamma_{T_K}\) (\cref{thm:structure:HFTcurve}). 
	Indeed, the former says that for any \(\ell\in\Z^{>0}\) and \(r,a\in\hZ\), 
	\[
	\#\{\text{curve segments \(\delta^r t^{ a}\capR_{\ell}\) in \(\Sset(\gamma_K)\)}\}
	=
	\#\{\text{curve segments \(\delta^r t^{-a}\capL_{\ell}\) in \(\Sset(\gamma_K)\)}\}
	\]
	and the latter that for any \(\ell\in\Z^{>0}\) and \(r,a_1,a_2\in\hZ\), 
	\[
	\#\{\text{components \(\delta^{r} t_1^{a_1}t_2^{a_2}\,\s_{2\ell}\) of \(\gamma_{T_K}\)}\}
	=
	\#\{\text{components \(\delta^{r} t_1^{-a_1}t_2^{-a_2}\,\sS_{2\ell}\) of \(\gamma_{T_K}\)}\}.
	\]
\end{remark}

\newcommand{\ppc}{%
	(-0.5,1) -- 
	(-0.5,0.5) .. controls +(0,-0.5) and +(-0.2,0) ..%
	(0,0.3) .. controls +(0.1,0) and +(0,0.25) ..%
	(0.25,0) .. controls +(0,-0.25) and +(0.1,0) ..%
	(0,-0.3) .. controls +(-0.2,0) and +(0,0.5) ..%
	(-0.5,-0.5) -- (-0.5,-1)
}
\newcommand{\ppfl}{%
	\draw [line width=1pt,xshift=-3cm]%
	(3,1) -- \ppc -- (3,-1);%
}
\newcommand{\ppfr}{%
	\draw [line width=1pt,xshift=3cm,rotate=180]%
	(3,1) -- \ppc -- (3,-1);%
}
\newcommand{\ppmr}{%
	\draw [line width=1pt,xshift=3cm]%
	(-3,1) -- \ppc -- (-3,-1);%
}
\newcommand{\ppml}{%
	\draw [line width=1pt,xshift=-3cm,rotate=180]%
	(-3,1) -- \ppc -- (-3,-1);%
}

\begin{figure}[t]
	\centering
	\begin{subfigure}{0.2\textwidth}
		\centering
		\begin{tikzpicture}[scale=0.5]
		\footnotesize
		\draw (0,-4) node(L){\(\bullet\)};
		\draw (-0.5,-4.75) node [right]{\(\HFKgr{\tfrac{\ell}{2}-1}{-\tfrac{\ell}{2}}\)};
		\draw (2,-2) node(l){\(\circ\)};
		\draw (2, 0) node(c){};
		\draw (2, 2) node(r){\(\circ\)};
		\draw (0, 4) node(R){\(\bullet\)};
		\draw (-0.5, 4.75) node [right]{\(\HFKgr{-\tfrac{\ell}{2}}{\tfrac{\ell}{2}}\)};
		\draw[->,in=-90,out=0] (L) to node [above left] {\(\sigma_{3}\)} (l);
		\draw[->,d] (l) to node [left] {\(\sigma_{23}\)} (c);
		\draw[->,d] (c) to node [left] {\(\sigma_{23}\)} (r);
		\draw[->,in=0,out=90] (r) to node [below left] {\(\sigma_2\)} (R);
		\draw [very thick,decorate,decoration={calligraphic brace,amplitude=10pt},yshift=0pt]
		(2.5,2.3) -- (2.5,-2.3) node [black,midway,xshift=0.55cm] {\(\ell\)};
		\draw (-0.5,3.5) node [right]{\(x_u\)};
		\draw (-0.5,-3.5) node [right]{\(y_u\)};
		\end{tikzpicture}
		\caption{\(\capR_{\ell}\)}
		\label{fig:HFKgradings:u}
	\end{subfigure}%
	\begin{subfigure}{0.2\textwidth}
		\centering
		\begin{tikzpicture}[scale=0.5]
		\footnotesize
		\draw ( 0,-4) node(L){\(\bullet\)};
		\draw (0.5,-4.75) node [left,align=right]{%
			\(\HFKgr{-\tfrac{\ell}{2}}{-\tfrac{\ell}{2}}\)};
		\draw (-2,-2) node(l){\(\circ\)};
		\draw (-2, 0) node(c){};
		\draw (-2, 2) node(r){\(\circ\)};
		\draw ( 0, 4) node(R){\(\bullet\)};
		\draw (0.5,4.75) node [left,align=right]{%
			\(\HFKgr{\tfrac{\ell}{2}-1}{\tfrac{\ell}{2}}\)};
		\draw[->,in=-90,out=180] (L) to node [above right,pos=0.6] {\(\sigma_{123}\)} (l);
		\draw[->,d] (l) to node [right] {\(\sigma_{23}\)} (c);
		\draw[->,d] (c) to node [right] {\(\sigma_{23}\)} (r);
		\draw[<-,in=180,out=90] (r) to node [below right] {\(\sigma_1\)} (R);
		\draw [very thick,decorate,decoration={calligraphic brace,amplitude=10pt},yshift=0pt]
		(-2.5,-2.3) -- (-2.5,2.3) node [black,midway,xshift=-0.55cm] {\(\ell\)};
		\draw (-0.5,3.5) node [right]{\(x_v\)};
		\draw (-0.5,-3.5) node [right]{\(y_v\)};
		\end{tikzpicture}
		\caption{\(\capL_{\ell}\)}
		\label{fig:HFKgradings:v}
	\end{subfigure}%
	\begin{subfigure}{0.2\textwidth}
		\centering
		\begin{tikzpicture}[scale=0.5]
		\footnotesize
		\draw (0,-4) node(L){\(\bullet\)};
		\draw (0.5,-4.75) node [left]{\(\HFKgr{-\tfrac{\ell}{2}}{-\tfrac{\ell}{2}}\)};
		\draw (-2,-2) node(l){\(\circ\)};
		\draw (-2, 0) node(c){};
		\draw (-2, 2) node(r){\(\circ\)};
		\draw (-4, 4) node(R){\(\bullet\)};
		\draw (-4.5, 4.75) node [right]{\(\HFKgr{-\tfrac{\ell}{2}}{\tfrac{\ell}{2}}\)};
		\draw[->,in=-90,out=180] (L) to node [above right] {\(\sigma_{123}\)} (l);
		\draw[->,d] (l) to node [right] {\(\sigma_{23}\)} (c);
		\draw[->,d] (c) to node [right] {\(\sigma_{23}\)} (r);
		\draw[->,in=0,out=90] (r) to node [below left] {\(\sigma_2\)} (R);
		\draw [very thick,decorate,decoration={calligraphic brace,amplitude=10pt},yshift=0pt]
		(-2.5,-2.3) -- (-2.5,2.3) node [black,midway,xshift=-0.55cm] {\(\ell\)};
		\end{tikzpicture}
		\caption{\(\ccon_{-\ell}\)}
		\label{fig:HFKgradings:c-}
	\end{subfigure}%
	\begin{subfigure}{0.2\textwidth}
		\centering
		\begin{tikzpicture}[scale=0.5]
		\footnotesize
		\draw (-4,0) node(L){\(\bullet\)};
		\draw (-4.5,-0.75) node [right]{\(\HFKgr{0}{0}\)};
		\draw ( 0, 0) node(R){\(\bullet\)};
		\draw ( 0.5, 0.75) node [left]{\(\HFKgr{0}{0}\)};
		\draw[->] (R) to node [above] {\(\sigma_{12}\)} (L);
		\draw (-4.5,-4.75) node [right]{\phantom{\((0;0,\tfrac{\ell}{2})\)}};
		\draw (-4.5, 4.75) node [right]{\phantom{\((0;0,\tfrac{\ell}{2})\)}};
		\end{tikzpicture}
		\caption{\(\ccon_{0}\)}
		\label{fig:HFKgradings:c0}
	\end{subfigure}%
	\begin{subfigure}{0.2\textwidth}
		\centering
		\begin{tikzpicture}[scale=0.5]
		\footnotesize
		\draw (-4,-4) node(L){\(\bullet\)};
		\draw (-4.5,-4.75) node [right]{\(\HFKgr{\tfrac{\ell}{2}}{-\tfrac{\ell}{2}}\)};
		\draw (-2,-2) node(l){\(\circ\)};
		\draw (-2, 0) node(c){};
		\draw (-2, 2) node(r){\(\circ\)};
		\draw ( 0, 4) node(R){\(\bullet\)};
		\draw ( 0.5, 4.75) node [left]{\(\HFKgr{\tfrac{\ell}{2}}{\tfrac{\ell}{2}}\)};
		\draw[->,in=-90,out=0] (L) to node [above left] {\(\sigma_{3}\)} (l);
		\draw[->,d] (l) to node [right] {\(\sigma_{23}\)} (c);
		\draw[->,d] (c) to node [right] {\(\sigma_{23}\)} (r);
		\draw[<-,in=180,out=90] (r) to node [below right] {\(\sigma_1\)} (R);
		\draw [very thick,decorate,decoration={calligraphic brace,amplitude=10pt},yshift=0pt]
		(-2.5,-2.3) -- (-2.5,2.3) node [black,midway,xshift=-0.55cm] {\(\ell\)};
		\end{tikzpicture}
		\caption{\(\ccon_{\ell}\)}
		\label{fig:HFKgradings:c+}
	\end{subfigure}
	\caption{The graded chain complexes corresponding to the curve segments that can occur in \(\Sset(\gamma_K)\). Note \(\ell\in\Z^{>0}\).}
	\label{fig:HFKgradings}
\end{figure}
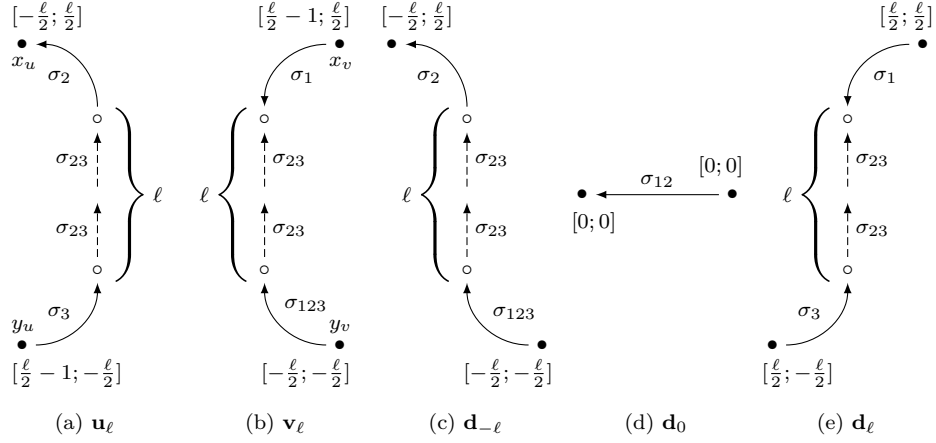
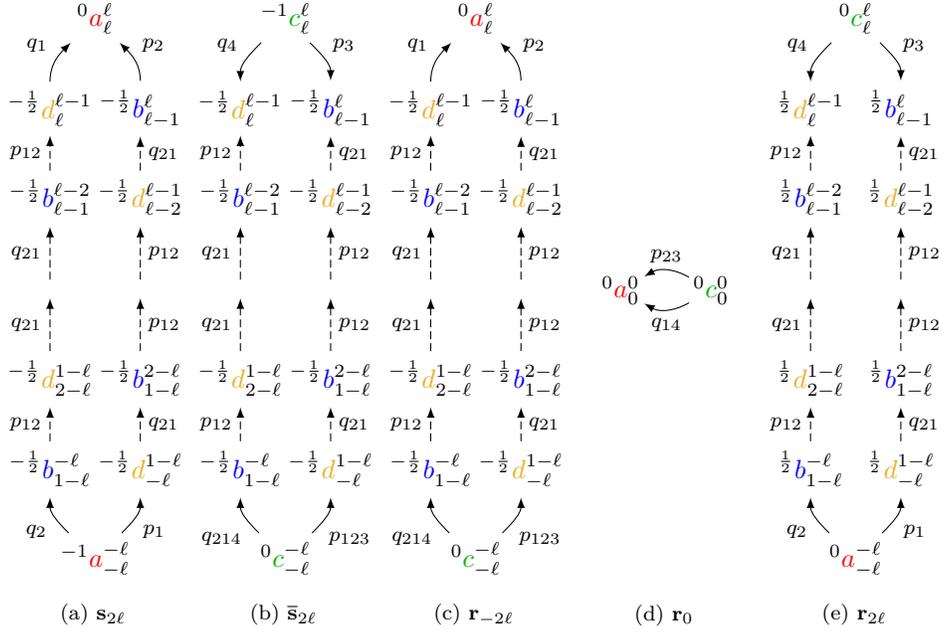
\begin{figure}[t]
	\centering
	\newcommand{\allthesame}{%
		\draw (-1, 4) node(tt1){\(\HFTgr{-\hhalf}{\ell-1}{\ell}{\arcD}\)};
		\draw (-1, 2) node(t1) {\(\HFTgr{-\hhalf}{\ell-2}{\ell-1}{\arcB}\)};
		\draw (-1, 0) node(c1) {};
		\draw (-1,-2) node(b1) {\(\HFTgr{-\hhalf}{1-\ell}{2-\ell}{\arcD}\)};
		\draw (-1,-4) node(bb1){\(\HFTgr{-\hhalf}{-\ell}{1-\ell}{\arcB}\)};
		
		\draw (1, 4) node(tt2){\(\HFTgr{-\hhalf}{\ell}{\ell-1}{\arcB}\)};
		\draw (1, 2) node(t2) {\(\HFTgr{-\hhalf}{\ell-1}{\ell-2}{\arcD}\)};
		\draw (1, 0) node(c2) {};
		\draw (1,-2) node(b2) {\(\HFTgr{-\hhalf}{2-\ell}{1-\ell}{\arcB}\)};
		\draw (1,-4) node(bb2){\(\HFTgr{-\hhalf}{1-\ell}{-\ell}{\arcD}\)};
		\footnotesize
		\draw[->,d] (t1)  to node [left] {\(p_{12}\)} (tt1);
		\draw[->,d] (t2)  to node [right] {\(q_{21}\)} (tt2);
		\draw[->,d] (c1)  to node [left] {\(q_{21}\)} (t1);
		\draw[->,d] (c2)  to node [right] {\(p_{12}\)} (t2);
		\draw[->,d] (b1)  to node [left] {\(q_{21}\)} (c1);
		\draw[->,d] (b2)  to node [right] {\(p_{12}\)} (c2);
		\draw[->,d] (bb1) to node [left] {\(p_{12}\)} (b1);
		\draw[->,d] (bb2) to node [right] {\(q_{21}\)} (b2);		
	}
	
	\begin{subfigure}{0.2\textwidth}
		\centering
		\begin{tikzpicture}[scale=0.6]
		\draw (0, 6) node(T){\(\HFTgr{0}{\ell}{\ell}{\arcA}\)};		
		\draw (0,-6) node(B){\(\HFTgr{-1}{-\ell}{-\ell}{\arcA}\)};
		\allthesame
		\draw[->,out=90,in=-135] (tt1) to node [above left] {\(q_{1}\)} (T);
		\draw[->,out=90,in= -45] (tt2) to node [above right] {\(p_{2}\)} (T);
		\draw[->,out=135,in=-90] (B) to node [below left] {\(q_{2}\)} (bb1);
		\draw[->,out= 45,in=-90] (B) to node [below right] {\(p_{1}\)} (bb2);
		\end{tikzpicture}
		\caption{\(\s_{2\ell}\)}
		\label{fig:HFTgradings:s12}
	\end{subfigure}%
	\begin{subfigure}{0.2\textwidth}
		\centering
		\begin{tikzpicture}[scale=0.6]
		\draw (0, 6) node(T){\(\HFTgr{-1}{\ell}{\ell}{\arcC}\)};
		\draw (0,-6) node(B){\(\HFTgr{0}{-\ell}{-\ell}{\arcC}\)};
		\allthesame
		\draw[<-,out=90,in=-135] (tt1) to node [above left] {\(q_{4}\)} (T);
		\draw[<-,out=90,in= -45] (tt2) to node [above right] {\(p_{3}\)} (T);
		\draw[->,out=135,in=-90] (B) to node [pos=0.3,below left] {\(q_{214}\)} (bb1);
		\draw[->,out= 45,in=-90] (B) to node [pos=0.3,below right] {\(p_{123}\)} (bb2);
		\end{tikzpicture}
		\caption{\(\sS_{2\ell}\)}
		\label{fig:HFTgradings:s34}
	\end{subfigure}%
	\begin{subfigure}{0.2\textwidth}
		\centering
		\begin{tikzpicture}[scale=0.6]
		\draw (0, 6) node(T){\(\HFTgr{0}{\ell}{\ell}{\arcA}\)};
		\draw (0,-6) node(B){\(\HFTgr{0}{-\ell}{-\ell}{\arcC}\)};
		\allthesame
		\draw[->,out=90,in=-135] (tt1) to node [above left] {\(q_{1}\)} (T);
		\draw[->,out=90,in= -45] (tt2) to node [above right] {\(p_{2}\)} (T);
		\draw[->,out=135,in=-90] (B) to node [pos=0.3,below left] {\(q_{214}\)} (bb1);
		\draw[->,out= 45,in=-90] (B) to node [pos=0.3,below right] {\(p_{123}\)} (bb2);
		\end{tikzpicture}
		\caption{\(\rr_{-2\ell}\)}
		\label{fig:HFTgradings:r-}
	\end{subfigure}%
	\begin{subfigure}{0.2\textwidth}
		\centering
		\begin{tikzpicture}[scale=0.6]
		\draw (0, 6) node(T){\phantom{\(\HFTgr{\hhalf}{\ell}{\ell}{\arcA}\)}};
		\draw (0, -6) node(T){\phantom{\(\HFTgr{\hhalf}{\ell}{\ell}{\arcA}\)}};
		\draw (-1, 0) node[t](L){\(\HFTgr{0}{0}{0}{\arcA}\)};
		\draw (1,0) node[t](R){\(\HFTgr{0}{0}{0}{\arcC}\)};
		\footnotesize
		\draw[->,out=150,in=30] (R) to node [above] {\(p_{23}\)} (L);
		\draw[->,out=-150,in=-30] (R) to node [below] {\(q_{14}\)} (L);
		\end{tikzpicture}
		\caption{\(\rr_0\)}
		\label{fig:HFTgradings:r0}
	\end{subfigure}%
	\begin{subfigure}{0.2\textwidth}
		\centering
		\begin{tikzpicture}[scale=0.6]
		\draw (0, 6) node(T){\(\HFTgr{0}{\ell}{\ell}{\arcC}\)};		
		\draw (0,-6) node(B){\(\HFTgr{0}{-\ell}{-\ell}{\arcA}\)};
		
		\draw (-1, 4) node(tt1){\(\HFTgr{\hhalf}{\ell-1}{\ell}{\arcD}\)};
		\draw (-1, 2) node(t1) {\(\HFTgr{\hhalf}{\ell-2}{\ell-1}{\arcB}\)};
		\draw (-1, 0) node(c1) {};
		\draw (-1,-2) node(b1) {\(\HFTgr{\hhalf}{1-\ell}{2-\ell}{\arcD}\)};
		\draw (-1,-4) node(bb1){\(\HFTgr{\hhalf}{-\ell}{1-\ell}{\arcB}\)};
		
		\draw (1, 4) node(tt2){\(\HFTgr{\hhalf}{\ell}{\ell-1}{\arcB}\)};
		\draw (1, 2) node(t2) {\(\HFTgr{\hhalf}{\ell-1}{\ell-2}{\arcD}\)};
		\draw (1, 0) node(c2) {};
		\draw (1,-2) node(b2) {\(\HFTgr{\hhalf}{2-\ell}{1-\ell}{\arcB}\)};
		\draw (1,-4) node(bb2){\(\HFTgr{\hhalf}{1-\ell}{-\ell}{\arcD}\)};
		\footnotesize
		\draw[->,d] (t1)  to node [left] {\(p_{12}\)} (tt1);
		\draw[->,d] (t2)  to node [right] {\(q_{21}\)} (tt2);
		\draw[->,d] (c1)  to node [left] {\(q_{21}\)} (t1);
		\draw[->,d] (c2)  to node [right] {\(p_{12}\)} (t2);
		\draw[->,d] (b1)  to node [left] {\(q_{21}\)} (c1);
		\draw[->,d] (b2)  to node [right] {\(p_{12}\)} (c2);
		\draw[->,d] (bb1) to node [left] {\(p_{12}\)} (b1);
		\draw[->,d] (bb2) to node [right] {\(q_{21}\)} (b2);	
		
		\draw[<-,out=90,in=-135] (tt1) to node [above left] {\(q_{4}\)} (T);
		\draw[<-,out=90,in= -45] (tt2) to node [above right] {\(p_{3}\)} (T);
		\draw[->,out=135,in=-90] (B) to node [below left] {\(q_{2}\)} (bb1);
		\draw[->,out= 45,in=-90] (B) to node [below right] {\(p_{1}\)} (bb2);
		\end{tikzpicture}
		\caption{\(\rr_{2\ell}\)}
		\label{fig:HFTgradings:r+}
	\end{subfigure}
	\caption{The graded chain complexes corresponding to curves that can appear as components of \(\gamma_{T_K}\). Note \(\ell\in\Z^{>0}\).}
	\label{fig:HFTgradings}
\end{figure}

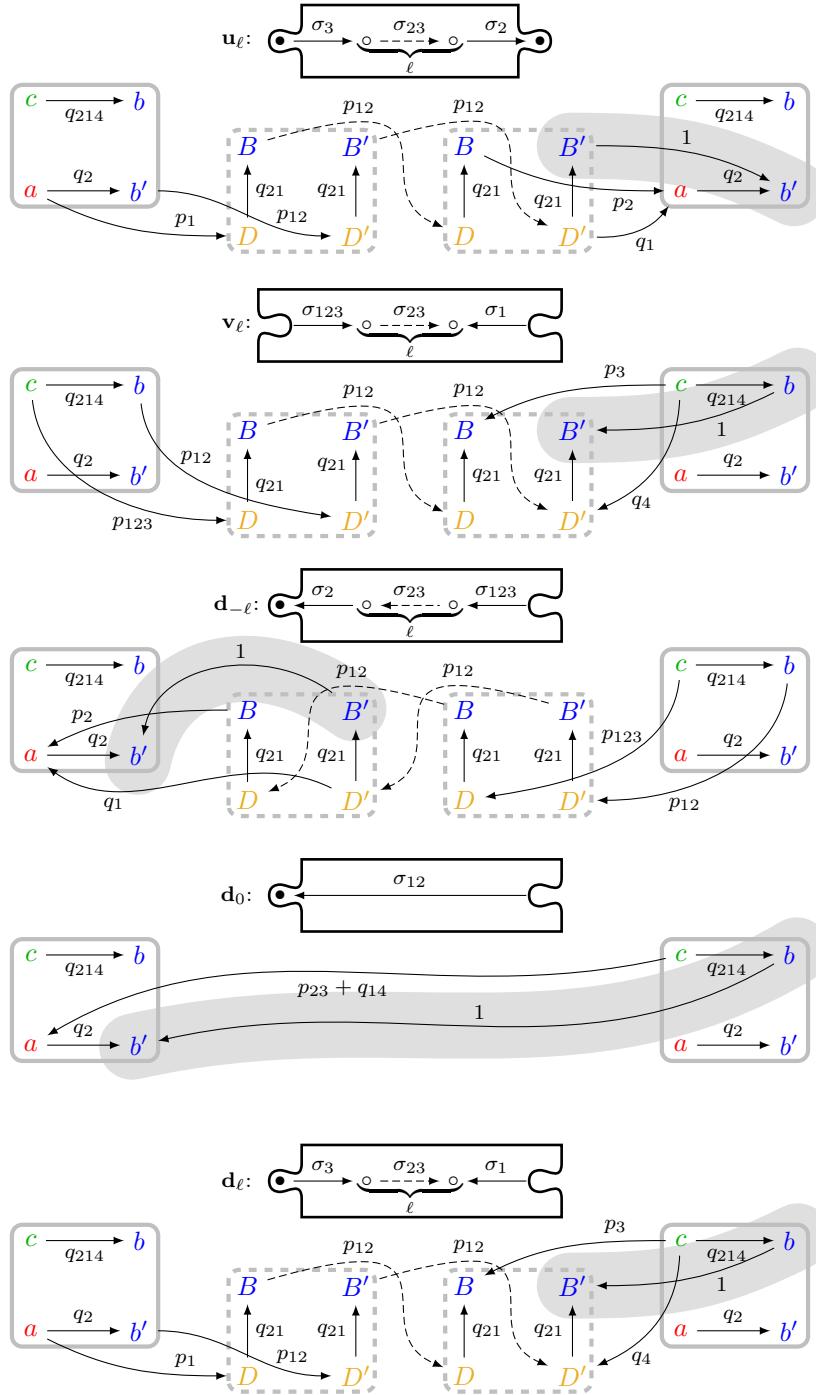
\begin{figure}[p]
	\centering
	\newcommand{\mybox}{(-1.35,-1.35) -- (-1.35,1.35) -- (1.35,1.35) -- (1.35,-1.35) -- cycle}
	\newcommand{\myvspace}{\vspace{10pt}}
	\newcommand{\myvvspace}{2.9cm}
		\begin{tikzpicture}[ppscale]
		\begin{scope}[scale=0.8,yshift=\myvvspace]
			\footnotesize
			\ppml\ppmr
			\draw (-4,0) node(W){\(\capR_{\ell}\):};
			\draw (-3,0) node(L){\(\bullet\)};
			\draw (-1,0) node(l){\(\circ\)};
			\draw (1,0) node(r){\(\circ\)};
			\draw (3,0) node(R){\(\bullet\)};
			\draw[->] (L) to node [above] {\(\sigma_{3}\)} (l);
			\draw[->,d] (l) to node [above] {\(\sigma_{23}\)} (r);
			\draw[->] (r) to node [above] {\(\sigma_2\)} (R);
			\draw (0,-0.5) node {\(\underbrace{\hspace{1.4cm}}_{\ell}\)};
		\end{scope}
		
		\draw (-7,1)   node(Llt){\textcolor{darkgreen}{\(c\)}};
		\draw (-7,-1)   node(Llb){\textcolor{red}{			\(a\)}};
		\draw (-5,1)   node(Lrt){\textcolor{blue}{			\(b\)}};
		\draw (-5,-1)   node(Lrb){\textcolor{blue}{			\(b'\)}};
		
		\draw (-3,0)   node(llt){\textcolor{blue}{			\(B\)}};
		\draw (-3,-2)   node(llb){\textcolor{gold}{		\(D\)}};
		\draw (-1,0)   node(lrt){\textcolor{blue}{			\(B'\)}};
		\draw (-1,-2)   node(lrb){\textcolor{gold}{			\(D'\)}};
		
		\draw (1,0)   node(rlt){\textcolor{blue}{			\(B\)}};
		\draw (1,-2)   node(rlb){\textcolor{gold}{		\(D\)}};
		\draw (3,0)   node(rrt){\textcolor{blue}{			\(B'\)}};
		\draw (3,-2)   node(rrb){\textcolor{gold}{			\(D'\)}};
		
		\draw (5,1)   node(Rlt){\textcolor{darkgreen}{\(c\)}};

		\draw (7,1)   node(Rrt){\textcolor{blue}{			\(b\)}};
		\draw (7,-1)   node(Rrb){\textcolor{blue}{			\(b'\)}};
		
		\draw[gt] (rrt) [out=0,in=150] to (Rrb);
		\draw (5,-1)   node(Rlb){\textcolor{red}{			\(a\)}};
		
		\draw[gb,xshift=-6cm] \mybox;
		\draw[gb,xshift= 6cm] \mybox;
		\draw[gd,xshift=-2cm,yshift=-1cm] \mybox;
		\draw[gd,xshift= 2cm,yshift=-1cm] \mybox;
		
		\draw (3,0)   node(rrt){\textcolor{blue}{			\(B'\)}};
		\draw (7,-1)   node(Rrb){\textcolor{blue}{			\(b'\)}};
		
		\footnotesize
		\draw[->] (Llt) to node[below]{\(q_{214}\)} (Lrt);
		\draw[->] (Llb) to node[above]{\(q_2\)} (Lrb);
		\draw[->] (Rlt) to node[below]{\(q_{214}\)} (Rrt);
		\draw[->] (Rlb) to node[above]{\(q_2\)} (Rrb);
		
		\draw[->] (llb) to node[right,pos=0.5]{\(q_{21}\)} (llt);
		\draw[->] (lrb) to node[left ,pos=0.5]{\(q_{21}\)} (lrt);
		\draw[->] (rlb) to node[right,pos=0.5]{\(q_{21}\)} (rlt);
		\draw[->] (rrb) to node[left ,pos=0.3]{\(q_{21}\)} (rrt);
		
		\draw[->,d] (llt) .. controls +(4.3,1.5) and +(-2,1) ..  node[above,pos=0.2]{\(p_{12}\)} (rlb);
		\draw[->,d] (lrt) .. controls +(4.3,1.5) and +(-2,1) ..  node[above,pos=0.2]{\(p_{12}\)} (rrb);
		
		\draw[->] (Llb) [out=-30,in=180] to node[above,pos=0.8]{\(p_1\)} (llb);
		\draw[->] (Lrb) [out=0,in=180] to node[above,pos=0.8]{\(p_{12}\)} (lrb);

		\draw[->] (rlt) [out=-30,in=180] to node[below,pos=0.8] {\(p_2\)} (Rlb);
		\draw[->] (rrt) [out=0,in=150] to node[above]{\(1\)} (Rrb);
		
		\draw[<-] (Rlb) [bend left] to  node [below]{~~\(q_1\)} (rrb);
		\end{tikzpicture}
		
		\myvspace
		\begin{tikzpicture}[ppscale]
		\begin{scope}[scale=0.8,yshift=\myvvspace]
			\footnotesize
			\ppfl\ppfr
			\draw (-4,0) node(W){\(\capL_{\ell}\):};
			\draw (-3,0) node(L){\phantom{\(\bullet\)}};
			\draw (-1,0) node(l){\(\circ\)};
			\draw (1,0) node(r){\(\circ\)};
			\draw (3,0) node(R){\phantom{\(\bullet\)}};
			\draw[->] (L) to node [above] {\(\sigma_{123}\)} (l);
			\draw[->,d] (l) to node [above] {\(\sigma_{23}\)} (r);
			\draw[<-] (r) to node [above] {\(\sigma_1\)} (R);
			\draw (0,-0.5) node {\(\underbrace{\hspace{1.4cm}}_{\ell}\)};
		\end{scope}
		
		\draw (-7,1)   node(Llt){\textcolor{darkgreen}{\(c\)}};
		\draw (-7,-1)   node(Llb){\textcolor{red}{			\(a\)}};
		\draw (-5,1)   node(Lrt){\textcolor{blue}{			\(b\)}};
		\draw (-5,-1)   node(Lrb){\textcolor{blue}{			\(b'\)}};
		
		\draw (-3,0)   node(llt){\textcolor{blue}{			\(B\)}};
		\draw (-3,-2)   node(llb){\textcolor{gold}{		\(D\)}};
		\draw (-1,0)   node(lrt){\textcolor{blue}{			\(B'\)}};
		\draw (-1,-2)   node(lrb){\textcolor{gold}{			\(D'\)}};
		
		\draw (1,0)   node(rlt){\textcolor{blue}{			\(B\)}};
		\draw (1,-2)   node(rlb){\textcolor{gold}{		\(D\)}};
		\draw (3,0)   node(rrt){\textcolor{blue}{			\(B'\)}};
		\draw (3,-2)   node(rrb){\textcolor{gold}{			\(D'\)}};

		\draw (5,-1)   node(Rlb){\textcolor{red}{			\(a\)}};
		\draw (7,1)   node(Rrt){\textcolor{blue}{			\(b\)}};
		\draw (7,-1)   node(Rrb){\textcolor{blue}{			\(b'\)}};
		
		\draw[gt] (Rrt) [out=-150,in=0] to (rrt);
		\draw (5,1)   node(Rlt){\textcolor{darkgreen}{\(c\)}};
		
		\draw[gb,xshift=-6cm] \mybox;
		\draw[gb,xshift= 6cm] \mybox;
		\draw[gd,xshift=-2cm,yshift=-1cm] \mybox;
		\draw[gd,xshift= 2cm,yshift=-1cm] \mybox;
		
		\draw (3,0)   node(rrt){\textcolor{blue}{			\(B'\)}};
		\draw (7,1)   node(Rrt){\textcolor{blue}{			\(b\)}};
		
		\footnotesize
		\draw[->] (Rlt) [out=180, in=30] to node[above,pos=0.25]{\(p_3\)} (rlt);
		\draw[->] (Rrt) [out=-150,in=0] to node[below,pos=0.3]{\(1\)} (rrt);
		
		\draw[->] (Llt) to node[below]{\(q_{214}\)} (Lrt);
		\draw[->] (Llb) to node[above]{\(q_2\)} (Lrb);
		\draw[->] (Rlt) to node[below,pos=0.4]{\(q_{214}\)} (Rrt);
		\draw[->] (Rlb) to node[above]{\(q_2\)} (Rrb);
		
		\draw[->] (llb) to node[right,pos=0.3]{\(q_{21}\)} (llt);
		\draw[->] (lrb) to node[left ,pos=0.7]{\(q_{21}\)} (lrt);
		\draw[->] (rlb) to node[right,pos=0.5]{\(q_{21}\)} (rlt);
		\draw[->] (rrb) to node[left ,pos=0.5]{\(q_{21}\)} (rrt);
		
		\draw[->,d] (llt) .. controls +(4.3,1.5) and +(-2,1) ..  node[above,pos=0.2]{\(p_{12}\)} (rlb);
		\draw[->,d] (lrt) .. controls +(4.3,1.5) and +(-2,1) ..  node[above,pos=0.2]{\(p_{12}\)} (rrb);
		
		\draw[->] (Llt) [in=180,out=-85] to node[below,pos=0.7] {\(p_{123}\)~~~~} (llb);
		\draw[->] (Lrt) [in=170,out=-85] to node[above,pos=0.4] {~~~\(p_{12}\)} (lrb);
		
		\draw[->] (Rlt) [bend left] to node [right,pos=0.9]{~\,\(q_4\)} (rrb);
		\end{tikzpicture}
		
		\myvspace
		\begin{tikzpicture}[ppscale]
		\begin{scope}[scale=0.8,yshift=\myvvspace]
		\footnotesize
		\ppfr\ppml
		\draw (-4,0) node(W){\(\ccon_{-\ell}\):};
		\draw (-3,0) node(L){\(\bullet\)};
		\draw (-1,0) node(l){\(\circ\)};
		\draw (1,0) node(r){\(\circ\)};
		\draw (3,0) node(R){\phantom{\(\bullet\)}};
		\draw[<-] (L) to node [above] {\(\sigma_{2}\)} (l);
		\draw[<-,d] (l) to node [above] {\(\sigma_{23}\)} (r);
		\draw[<-] (r) to node [above] {\(\sigma_{123}\)} (R);
		\draw (0,-0.5) node {\(\underbrace{\hspace{1.4cm}}_{\ell}\)};
		\end{scope}
		
		\draw (-7,1)   node(Llt){\textcolor{darkgreen}{\(c\)}};
		\draw (-7,-1)   node(Llb){\textcolor{red}{			\(a\)}};
		\draw (-5,1)   node(Lrt){\textcolor{blue}{			\(b\)}};
		\draw (-5,-1)   node(Lrb){\textcolor{blue}{			\(b'\)}};

		\draw (-3,-2)   node(llb){\textcolor{gold}{		\(D\)}};
		\draw (-1,0)   node(lrt){\textcolor{blue}{			\(B'\)}};
		\draw (-1,-2)   node(lrb){\textcolor{gold}{			\(D'\)}};
		
		\draw (1,0)   node(rlt){\textcolor{blue}{			\(B\)}};
		\draw (1,-2)   node(rlb){\textcolor{gold}{		\(D\)}};
		\draw (3,0)   node(rrt){\textcolor{blue}{			\(B'\)}};
		\draw (3,-2)   node(rrb){\textcolor{gold}{			\(D'\)}};
		
		\draw (5,1)   node(Rlt){\textcolor{darkgreen}{\(c\)}};
		\draw (5,-1)   node(Rlb){\textcolor{red}{			\(a\)}};
		\draw (7,1)   node(Rrt){\textcolor{blue}{			\(b\)}};
		\draw (7,-1)   node(Rrb){\textcolor{blue}{			\(b'\)}};
		
		\draw[gt] (lrt) [in=80,out=140,looseness=1.2] to (Lrb);	

		\draw (-3,0)   node(llt){\textcolor{blue}{			\(B\)}};

		\draw[gb,xshift=-6cm] \mybox;
		\draw[gb,xshift= 6cm] \mybox;
		\draw[gd,xshift=-2cm,yshift=-1cm] \mybox;
		\draw[gd,xshift= 2cm,yshift=-1cm] \mybox;

		\draw (-1,0)   node(lrt){\textcolor{blue}{			\(B'\)}};
		\draw (-5,-1)   node(Lrb){\textcolor{blue}{			\(b'\)}};
		
		\footnotesize
		
		\draw[->] (Llt) to node[below]{\(q_{214}\)} (Lrt);
		\draw[->] (Llb) to node[above,pos=0.7]{\(q_2\)} (Lrb);
		\draw[->] (Rlt) to node[below,pos=0.4]{\(q_{214}\)} (Rrt);
		\draw[->] (Rlb) to node[above]{\(q_2\)} (Rrb);
		
		\draw[->] (llb) to node[right]{\(q_{21}\)} (llt);
		\draw[->] (lrb) to node[left ]{\(q_{21}\)} (lrt);
		\draw[->] (rlb) to node[right]{\(q_{21}\)} (rlt);
		\draw[->] (rrb) to node[left ]{\(q_{21}\)} (rrt);

		\draw[->,d] (rlt) .. controls +(-4.3,1.5) and +(2,1) ..  node[above,pos=0.2]{\(p_{12}\)} (llb);
		\draw[->,d] (rrt) .. controls +(-4.3,1.5) and +(2,1) ..  node[above,pos=0.2]{\(p_{12}\)} (lrb);

		\draw[->] (llt) [out=180,in=30] to node[above,pos=0.8] {\(p_2\)} (Llb);
		\draw[->] (lrt) [in=80,out=140,looseness=1.15] to node[above,pos=0.4]{\(1\)} (Lrb);		
		
		\draw[->] (Rlt) [in=10,out=-95] to node[above,pos=0.4] {\(p_{123}\)~~~~} (rlb);
		\draw[->] (Rrt) [in=0,out=-95] to node[below,pos=0.7] {~~\(p_{12}\)} (rrb);
	
		\draw[<-] (Llb) [out=-40,in=150] to	node [below,pos=0.25]{\(q_1\)} (lrb);
		
		\end{tikzpicture}
		
		\myvspace
		\begin{tikzpicture}[ppscale]
		\begin{scope}[scale=0.8,yshift=\myvvspace]
			\footnotesize
			\ppfr\ppml
			\draw (-4,0) node(W){\(\ccon_0\):};
			\draw (-3,0) node(L){\(\bullet\)};
			\draw (3,0) node(R){\phantom{\(\bullet\)}};
			\draw[<-] (L) to node [above] {\(\sigma_{12}\)} (R);
		\end{scope}
		\draw (-7,1)   node(Llt){\textcolor{darkgreen}{\(c\)}};
		\draw (-7,-1)   node(Llb){\textcolor{red}{			\(a\)}};
		\draw (-5,1)   node(Lrt){\textcolor{blue}{			\(b\)}};
		\draw (-5,-1)   node(Lrb){\textcolor{blue}{			\(b'\)}};
		
		\draw (5,1)   node(Rlt){\textcolor{darkgreen}{\(c\)}};
		\draw (5,-1)   node(Rlb){\textcolor{red}{			\(a\)}};
		\draw (7,1)   node(Rrt){\textcolor{blue}{			\(b\)}};
		\draw (7,-1)   node(Rrb){\textcolor{blue}{			\(b'\)}};
		
		\draw[gt] (Rrt) [in=15,out=-145] to (Lrb);

		\draw[gb,xshift=-6cm] \mybox;
		\draw[gb,xshift= 6cm] \mybox;
		
		\draw (7,1)   node(Rrt){\textcolor{blue}{			\(b\)}};
		\draw (-5,-1)   node(Lrb){\textcolor{blue}{			\(b'\)}};
		
		\footnotesize
		
		\draw[->] (Llt) to node[below]{\(q_{214}\)} (Lrt);
		\draw[->] (Llb) to node[above]{\(q_2\)} (Lrb);
		\draw[->] (Rlt) to node[below,pos=0.4]{\(q_{214}\)} (Rrt);
		\draw[->] (Rlb) to node[above]{\(q_2\)} (Rrb);
		
		\draw[->] (Rlt) [in=35,out=-165]  to node[below]{\(p_{23}+q_{14}\)} (Llb);
		\draw[->] (Rrt) [in=15,out=-145] to node[above] {\(1\)} (Lrb);
		\end{tikzpicture}
		
		\myvspace
		\begin{tikzpicture}[ppscale]
		\begin{scope}[scale=0.8,yshift=\myvvspace]
			\footnotesize
			\ppfr\ppml
			\draw (-4,0) node(W){\(\ccon_{\ell}\):};
			\draw (-3,0) node(L){\(\bullet\)};
			\draw (-1,0) node(l){\(\circ\)};
			\draw (1,0) node(r){\(\circ\)};
			\draw (3,0) node(R){\phantom{\(\bullet\)}};
			\draw[->] (L) to node [above] {\(\sigma_{3}\)} (l);
			\draw[->,d] (l) to node [above] {\(\sigma_{23}\)} (r);
			\draw[<-] (r) to node [above] {\(\sigma_1\)} (R);
			\draw (0,-0.5) node {\(\underbrace{\hspace{1.4cm}}_{\ell}\)};
		\end{scope}
		\draw (-7,1)   node(Llt){\textcolor{darkgreen}{\(c\)}};
		\draw (-7,-1)   node(Llb){\textcolor{red}{			\(a\)}};
		\draw (-5,1)   node(Lrt){\textcolor{blue}{			\(b\)}};
		\draw (-5,-1)   node(Lrb){\textcolor{blue}{			\(b'\)}};
		
		\draw (-3,0)   node(llt){\textcolor{blue}{			\(B\)}};
		\draw (-3,-2)   node(llb){\textcolor{gold}{		\(D\)}};
		\draw (-1,0)   node(lrt){\textcolor{blue}{			\(B'\)}};
		\draw (-1,-2)   node(lrb){\textcolor{gold}{			\(D'\)}};
		
		\draw (1,0)   node(rlt){\textcolor{blue}{			\(B\)}};
		\draw (1,-2)   node(rlb){\textcolor{gold}{		\(D\)}};
		\draw (3,0)   node(rrt){\textcolor{blue}{			\(B'\)}};
		\draw (3,-2)   node(rrb){\textcolor{gold}{			\(D'\)}};
		
		\draw (5,-1)   node(Rlb){\textcolor{red}{			\(a\)}};
		\draw (7,1)   node(Rrt){\textcolor{blue}{			\(b\)}};
		\draw (7,-1)   node(Rrb){\textcolor{blue}{			\(b'\)}};
		
		\draw[gt] (Rrt) [out=-150,in=0] to (rrt);
		\draw (5,1)   node(Rlt){\textcolor{darkgreen}{\(c\)}};
		
		\draw[gb,xshift=-6cm] \mybox;
		\draw[gb,xshift= 6cm] \mybox;
		\draw[gd,xshift=-2cm,yshift=-1cm] \mybox;
		\draw[gd,xshift= 2cm,yshift=-1cm] \mybox;
		
		\draw (3,0)   node(rrt){\textcolor{blue}{			\(B'\)}};
		\draw (7,1)   node(Rrt){\textcolor{blue}{			\(b\)}};
		
		\footnotesize
		\draw[->] (Rlt) [out=180, in=30] to node[above,pos=0.25]{\(p_3\)} (rlt);
		\draw[->] (Rrt) [out=-150,in=0] to node[below,pos=0.3]{\(1\)} (rrt);
		
		\draw[->] (Llt) to node[below]{\(q_{214}\)} (Lrt);
		\draw[->] (Llb) to node[above]{\(q_2\)} (Lrb);
		\draw[->] (Rlt) to node[below,pos=0.4]{\(q_{214}\)} (Rrt);
		\draw[->] (Rlb) to node[above]{\(q_2\)} (Rrb);
		
		\draw[->] (llb) to node[right]{\(q_{21}\)} (llt);
		\draw[->] (lrb) to node[left ]{\(q_{21}\)} (lrt);
		\draw[->] (rlb) to node[right]{\(q_{21}\)} (rlt);
		\draw[->] (rrb) to node[left ]{\(q_{21}\)} (rrt);
		
		\draw[->,d] (llt) .. controls +(4.3,1.5) and +(-2,1) ..  node[above,pos=0.2]{\(p_{12}\)} (rlb);
		\draw[->,d] (lrt) .. controls +(4.3,1.5) and +(-2,1) ..  node[above,pos=0.2]{\(p_{12}\)} (rrb);
		
		\draw[->] (Llb) [out=-30,in=180] to node[above,pos=0.8]{\(p_1\)} (llb);
		\draw[->] (Lrb) [out=0,in=180] to node[above,pos=0.8]{\(p_{12}\)} (lrb);
		
		\draw[->] (Rlt) [bend left] to node [right,pos=0.9]{~\,\(q_4\)} (rrb);
		\end{tikzpicture}
	\caption{The main computation for the proof of \cref{thm:main:graded}, not including gradings. 
		Each puzzle piece represents a curve segment in \(\Sset(\gamma_K)\), which, when tensored with  the type~AD structure \(\BSADY\), becomes the chain complex over \(\algS\) shown below the puzzle piece. 
		Note \(\ell\in\Z^{>0}\).}
	\label{fig:pp-pairing}
\end{figure}

\begin{proof}[Proof of \cref{thm:main:graded}]
	Let us assume first that all local systems on \(\gamma_K\) are one-dimensional and thus trivial.
	With a view towards applying \cref{thm:BS_computations}, we first compute the effect of applying \(-\boxtimes \BSADY\) to the chain complexes over \(\algT\) corresponding to curve segments, ignoring gradings for a moment. 
	We carry out these computations in \cref{fig:pp-pairing}. 
	Clearly,
	\(\capR_{\ell}\boxtimes\BSADY\) 
	contains a direct summand that is chain homotopic to (the chain complex over \(\algS\) corresponding to)
	\(\s_{2\ell}\), 
	and	the same is true for 
	\(\capL_{\ell}\boxtimes\BSADY\)
	and 
	\(\sS_{2\ell}\), 
	as well as for
	\(\ccon_k\boxtimes\BSADY\)
	and 
	\(\rr_{2k}\)
	for all \(\ell\in\Z^{>0}\) and \(k\in\Z\). 
	However, in each case, four additional generators remain, two on each end, so to speak. 
	This is where we make use of Hanselman and Watson's puzzle piece notation \cite[Figure~1]{HW1}. 
	This notation is chosen such that two puzzle pieces fit together if and only if the corresponding curve segments can be joined at the respective ends to a valid curve on the torus.
	The following observation is crucial: 
	At each male end of a puzzle piece, the remaining two generators are the same, namely
	\textcolor{darkgreen}{\(c\)} and \textcolor{blue}{\(b\)}. 
	Similarly, at the female ends, the generators \textcolor{red}{\(a\)} and \textcolor{blue}{\(b'\)} remain. 
	This justifies a posteriori that the computation of 
	\(\CFD(X_K)^{\algT}\boxtimes\BSADY\) 
	can indeed be done locally on each curve segment. 
	In other words, 
	\(\CFD(X_K)^{\algT}\boxtimes\BSADY\)
	is indeed a direct sum of chain complexes over \(\algS\), namely one for each curve segment, as stated in the correspondence. 
	We now apply \cref{thm:BS_computations} to conclude.
	
	It remains to verify the gradings and to consider curves with non-trivial local systems. 
	We start with the latter.
	Let us suppose that \(\gamma_K\) contains some component with higher-dimensional local system~\(X\).
	For convenience, we may put this local system on an arrow adjacent to a generator in idempotent \(\bullet\); see the discussion before \cref{def:equivalence:local_systems}. 
	By inspecting the complexes in \cref{fig:pp-pairing}, we see that the local system on the curve corresponding to the curve segment containing this arrow is equal to \(XX^{-1}=\id\). 
	So any local system has the same effect on the result as a trivial local system of the same dimension. 
	
	To check the gradings, it clearly suffices to compute the grading of the generators in idempotents \(\arcA\) and \(\arcC\). 
	Note that any element the double coset \(\PI\backslash(\rGr(\Za)\times\mathfrak{A})/\PIII\) can be uniquely written as \(\PI\cdot(x;0,0;0,0,y,z)\cdot\PIII\) for some \(x,y,z\in\Q\).
	Identifying this element with \(\HFKgr{-x}{y,z}\), this recovers the usual bigrading on the tangle invariant.
	So, given a generator of \(\CFD(X_K)^{\algT}.\iota_\bullet\) in grading \(\HFKgr{m}{n}\), the corresponding generator of \(\CFD(X_K)^{\algT} \boxtimes \BSADY\) in idempotent \(\arcA\) has grading 
	\begin{align*}
	&
	\PI\cdot
	(-m-\tfrac{n}{2};0,-n;0,0,0,0)
	\cdot\PIII
	\\
	=~&
	\PI\cdot
	(-m-n;0,0;0,0,2n,2n)
	(\tfrac{1}{2};0,-1;0,0,-2,-2)^n
	\cdot\PIII
	\\
	=~&
	\PI\cdot
	(-m-n;0,0;0,0,2n,2n)
	\cdot\PIII
	\\
	=~&
	\HFKgr{m+n}{2n,2n}
	\end{align*}
	and the corresponding generator in idempotent \(\arcC\) has grading 
	\begin{align*}
	&
	\PI\cdot
	(-m-\tfrac{n}{2};0,-n;0,0,0,0)
	(-\half;-1,0;0,0,0,0)
	\cdot\PIII
	\\
	=~&
	\PI\cdot
	(-m-\tfrac{n}{2}-\half+n;-1,-n;0,0,0,0)
	\cdot\PIII
	\\
	=~&
	\PI\cdot
	(-\half;-1,0;0,0,0,0)
	(-m+n+\tfrac{n}{2};0,-n;0,0,0,0)
	\cdot\PIII
	\\
	=~&
	\PI\cdot
	(-m+n;0,0;0,0,2n,2n)
	(\tfrac{1}{2};0,-1;0,0,-2,-2)^n
	\cdot\PIII
	\\
	=~&
	\PI\cdot
	(-m+n;0,0;0,0,2n,2n)
	\cdot\PIII
	\\
	=~&
	\HFKgr{m-n}{2n,2n}.
	\end{align*}
	So it suffices to check the identities above for \(r=0=a\). 
	This is elementary to check, we simply compare the gradings in \cref{fig:HFKgradings,fig:HFTgradings}. 
\end{proof}

\section{Satellites, thinness, and A-links}\label{sec:thinness}

We now prove the results from \cref{subsec:intro:thinness} of the introduction. 

\begin{proof}[Proof of \cref{thm:intro:thin-satellite}]
	Zoltan Szabó's program \cite{hfkcalc} computes the following knot Floer homology of the knot \(K\) shown in \cref{fig:thin-satellite}:
	\[
	\HFKhat(K)
	\cong
	\delta^0t^1\mathbb{F}^4
	\oplus
	\delta^0t^0\mathbb{F}^9
	\oplus
	\delta^0t^{-1}\mathbb{F}^4.
	\]
	A conceptually more interesting proof, which also explains how this example was found, uses the multicurve invariant of the double tangle of the trefoil knot computed in \cref{exa:main:graded}. 
	Note that the Conway tangle \(T_\star\) highlighted on the left of \cref{fig:thin-satellite} agrees with said double tangle up to two full twists. 
	Using the naturality of the multicurve invariant under twisting \cite[Theorem~3.1]{pqSym}, we see that \(\gamma_{T_\star}\) consists of a single rational component \(\rr_0\) of slope 0 and a pair of conjugate special components of slope \(\infty\), denoted here by \(\mathbf{S}_\infty\). 
	The complementary tangle \(T'_\star\) on the right of \cref{fig:thin-satellite} is obtained from \(T_\star\) by rotation in the plane about 90 degrees and reversing all crossings. 
	Thus, the corresponding multicurve invariant consists of a single rational component \(\rr_\infty\) of slope \(\infty\) and a pair \(\mathbf{S}_0\) of conjugate special components of slope \(0\).
	The Lagrangian Floer homology between a rational and a special component of the same slope is zero \cite[Lemma~4.20]{KWZ_thin}. 
	Hence, by the gluing theorem \cite[Theorem~5.9]{pqMod},
	\[
	\HFKhat(T'_\star\cup T_\star)\otimes \mathbb{F}^2
	\cong
	\HF(-\gamma_{T'_\star},\gamma_{T_\star})
	\cong
	\HF(-\rr_\infty,\rr_0)
	\oplus
	\HF(-\mathbf{S}_0,\mathbf{S}_\infty).
	\]
	These two summands are thin and supported in the same \(\delta\)-grading. 
	This can be seen either by direct computation or as an application of \cite[Lemma~4.17]{KWZ_thin}. 
	%
	%
	%
	%
\end{proof}

\begin{proposition}\label{prop:two-specials}
	For any knot \(K\subset S^3\) except the unknot and the trefoil knots, the multicurve \(\gamma_{T_K}\) contains two special components in distinct \(\delta\)-gradings. 
\end{proposition}


\begin{figure}[t]\small
	\begin{tabular}{ccc}
		\toprule
		generator \(x\)
		&
		component of \(\Sset(\gamma_K)\)
		&
		component of \(\gamma_{T_K}\)
		\\\midrule
		\(x_u\)
		&
		\(\delta^{r+\tfrac{\ell}{2}}t^{a-\tfrac{\ell}{2}}\capR_{\ell}\)
		&
		\(\delta^{r+a}t_1^{2a-\ell}t_2^{2a-\ell}\s_{2\ell}\)
		\smallskip\\
		\(y_u\)
		&
		\(\delta^{r+1-\tfrac{\ell}{2}}t^{a+\tfrac{\ell}{2}}\capR_{\ell}\)
		&
		\(\delta^{r+1+a}t_1^{2a+\ell}t_2^{2a+\ell}\s_{2\ell}\)
		\smallskip\\
		\(x_v\)
		&
		\(\delta^{r+1-\tfrac{\ell}{2}}t^{a-\tfrac{\ell}{2}}\capL_{\ell}\)
		&
		\(\delta^{r+1-a}t_1^{2a-\ell}t_2^{2a-\ell}\sS_{2\ell}\)
		\smallskip\\
		\(y_v\)
		&
		\(\delta^{r+\tfrac{\ell}{2}}t^{a+\tfrac{\ell}{2}}\capL_{\ell}\)
		&
		\(\delta^{r-a}t_1^{2a+\ell}t_2^{2a+\ell}\sS_{2\ell}\)
		\\
		\bottomrule
	\end{tabular} 
	\caption{Computations for the proof of \cref{prop:two-specials}. 
	}
	\label{tab:thick-satellites}
\end{figure}

\begin{proof}
	We first introduce some notation and make some preliminary computations.  
	Given a curve segment \(\capR_{\ell}\) or \(\capL_{\ell}\) in \(\Sset(\gamma_K)\), consider the two generators of \(\HFKhat(K)\) at the ends of this curve segment. 
	If the segment is of type \(\capR_{\ell}\), we denote these two generators by \(x_u\) and \(y_u\) as shown in \cref{fig:HFKgradings:u}; 
	if it is of type \(\capL_{\ell}\), we denote them by \(x_v\) and \(y_v\) as shown in \cref{fig:HFKgradings:v}. 
	The second column of the table in \cref{tab:thick-satellites} shows the graded curve segments whose generators \(x\in\{x_u,y_u,x_v,y_v\}\), determined in the first column, sit in bigrading \(\HFKgr{r}{a}\).
	The third column shows the corresponding components of \(\gamma_{T_K}\) computed using \cref{thm:main:graded}. 
	Note that the \(\delta\)-grading of the components of \(\gamma_{T_K}\) are independent of the components'
	length \(\ell\). 
	
	We now study how two such curve segments can be joined up along a generator \(x\) of \(\HFKhat(K)\) to form part of the curve \(\gamma_K\).
	We distinguish two cases:
	\begin{enumerate}
		\item \label{enu:extremal} Suppose \(x=x_u=x_v\) or \(x=y_u=y_v\). 
		We compare the first and third, respectively second and fourth entry of the last column in the table in \cref{fig:thin-satellite}. 
		We see that in this case, \(\gamma_{T_K}\) contains two special components in distinct \(\delta\)-gradings, since \(a\), being the Alexander grading of a generator of \(\HFKhat(K)\), is an integer. 
		\item \label{enu:mixed} Suppose \(x=x_u=y_v\) or \(x=y_u=x_v\). 
		In this case, \(\gamma_{T_K}\) also contains two special components in distinct \(\delta\)-gradings, unless \(a=0\). 
	\end{enumerate}
	Let \(K\subset S^3\) be a knot and suppose \(\gamma_{T_K}\) does not contain two special components in distinct \(\delta\)-gradings. 
	Then \(\gamma_{K}\) consists of a single connected component. 
	Indeed, if there were two components, then in particular there would be a component of \(\gamma_{K}\) not containing the curve segment \(\ccon_{2\tau(K)}\); 
	a generator of maximal Alexander grading in that component would be a generator of type \eqref{enu:extremal}.
	Similarly, 
	one can see that \(\gamma_{K}\) has the structure of the knot Floer homology of an L-space knot: 
	All generators sit in distinct Alexander gradings and every curve segment in \(\Sset(\gamma_K)\) except \(\ccon_{2\tau(K)}\) connects generators in consecutive Alexander gradings. 
	(This is sometimes known as the ``staircase structure''.)
	If \(\dim\HFKhat(K)>3\) this implies that there is more than one generator of type \eqref{enu:mixed}, and only one of them can be in Alexander grading 0. 
	So we deduce that \(\dim\HFKhat(K)\leq3\). 
	But the only knots with such small knot Floer homology are the unknot and the two trefoil knots \cite{OSgenus,ghiggini2006knot,hedden2017geography}. 
\end{proof}

\begin{proof}[Proof of \cref{prop:intro:thick-satellites}]
	Let \(K\) be as in the statement of \cref{prop:intro:thick-satellites} and suppose for contradiction that \(P(K)\) is a satellite knot whose knot Floer homology is thin for some pattern \(P\) with wrapping number 2. 
	Then \(P(K)=T_P\cup T_K\) for some pattern tangle \(T_P\) and we may write 
	\[
	\HFKhat(P(K))\otimes \mathbb{F}^2
	\cong
	\HF(-\gamma_{T_P},\gamma_{T_K}).
	\]
	By \cref{prop:two-specials}, \(T_K\) contains two special components in distinct \(\delta\)-grading.
	Thus, \(\gamma_{T_P}\) only contains component of slope \(\infty\), because otherwise, \(\HF(-\gamma_{T_P},\gamma_{T_K})\) would not be thin by \cite[Lemma~4.17]{KWZ_thin}. 
	Furthermore, by the same reasoning, using \cite[Lemma~4.20]{KWZ_thin}, no component of \(\gamma_{T_P}\) can be special. 
	So \(\gamma_{T_P}\) consists of rational components of slope \(\infty\) only. 
	By \cite[Theorem 4.1]{LMZ}, the tangle \(T_P\) is vertically split and hence the wrapping number of the pattern \(P\) is 0, contrary to our assumptions.
\end{proof}

\begin{lemma}\label{lem:A-links}
	A knot \(K\subset S^3\) is an L-space knot if and only if all special components in \(\gamma_{T_K}\) sit in \(\delta\)-gradings of the same parity.
\end{lemma}

\begin{proof}
	This follows from the computations in the proof of \cref{prop:two-specials}, observing the fact that generators of type \eqref{enu:extremal} cannot occur. 
\end{proof}

\begin{proof}[Proof of \cref{prop:A-knot-satellites}] 
Let \(P\) and \(K\) be as in the statement of \cref{prop:A-knot-satellites} and suppose for contradiction that \(P(K)\) is an A-link. 
We now argue as in the proof of \cref{prop:intro:thick-satellites}, using the fact that by \cref{lem:A-links}, \(\gamma_{T_K}\) contains two special components in \(\delta\)-gradings of different parity.
\end{proof}

\begin{proof}[Proof of \cref{thm:A-links}]
	By \cref{lem:A-links}, it remains to see that there exists a non-trivial rational A-link filling of \(T_K\) if and only if all special components in \(\gamma_{T_K}\) sit in \(\delta\)-gradings of the same parity.
	But this follows from \cite[Lemma~4.17]{KWZ_thin}.
\end{proof}

	It is interesting to compare the following spaces \cite{RR,HRW,KWZ_thin}:
	\begin{align*}
		\mathcal{L}(K)
		&\coloneqq
		\{
		\nicefrac{p}{q}\in\QPI
		\mid
		S^3_{p/q}(K) \text{ is an L-space}
		\},
		\\
		\Aa{T_K}
		&\coloneqq
		\{
		\nicefrac{p}{q}\in\QPI
		\mid
		T_K(\nicefrac{p}{q}) \text{ is an A-link}
		\}.
	\end{align*}
	
\begin{theorem}\label{thm:comparison:fillings}
	For any knot \(K\subset S^3\), exactly one of the following holds:
	\begin{enumerate}
		\item \(\mathcal{L}(K)=\{\infty\}=\Aa{T_K}\);
		\item\label{item:2} \(\mathcal{L}(K)=\QPI\smallsetminus\{0\}=\Aa{T_K}\); in this case, \(K\) is the unknot;
		\item\label{item:3} \(\mathcal{L}(K)=[2\tau(K)-1,\infty]\) and \(\Aa{T_K}=(4\tau(K),\infty]\); in this case, \(\tau(K)>0\);
		\item\label{item:4} \(\mathcal{L}(K)=[\infty,2\tau(K)+1]\) and \(\Aa{T_K}=[\infty,4\tau(K)]\); in this case, \(\tau(K)<0\). 
	\end{enumerate}
\end{theorem}
	
\begin{proof}	
	Let \(K\subset S^3\) be a knot. 
	By \cref{thm:A-links}, \(\mathcal{L}(K)=\{\infty\}\) if and only if \(\Aa{T_K}=\{\infty\}\). 
	So it suffices to see that if \(K\) is an L-space knot, then exactly one of \eqref{item:2}--\eqref{item:4} holds. 
	We claim that the three cases correspond precisely to the cases \(\tau(K)=0\), \(\tau(K)>0\), and \(\tau(K)<0\), respectively. 
	Note that the unknot is the only L-space knot with \(\tau(K)=0\). 
	
	The values of \(\mathcal{L}(K)\) in these three cases can be easily determined from the multicurve \(\gamma_K\); this is explained in \cite[Section~7.5]{HRW}.  So it remains to determine the value of \(\Aa{T_K}\) in these cases. 
	
	For case \eqref{item:2}, observe that \(T_K\) is the rational tangle of slope 0, and so any rational filling of \(T_K\) is a 2-bridge link.
	All two-bridge links are alternating, except the unlink, which corresponds to the 0-rational filling and whose knot Floer homology is supported in two consecutive \(\delta\)-gradings. 
	So indeed \(\Aa{T_K}=\QPI\smallsetminus\{0\}\).
		
	For cases \eqref{item:3} and \eqref{item:4}, we first note that \(\Aa{T_K}\) contains some \(\nicefrac{p}{q}\neq\infty\) by \cref{thm:A-links}. 
	So \(\Aa{T_K}\) must be an interval \cite[Theorem~1.8]{KWZ_thin}. 
	In fact, as explained in \cite[Sections~2.1 and 4.2]{KWZ_thin}, the boundary points of the interval must be equal to \(\infty\) and \(4\tau(K)\), since \(\gamma_{T_K}\) is supported in these slopes. 
	Moreover, as the curve of slope \(4\tau(K)\) is rational, this slope is not contained in the interval and we already know that \(\infty\in\Aa{T_K}\).  
	So \(\Aa{T_K}\) is either equal to \((4\tau(K),\infty]\) or \([\infty,4\tau(K))\). 
	We determine which one it is by observing \(0\not\in\mathrm{A}(T_K)\). 
	This is because the determinant of the link \(T_K(0)=C_{2,0}(K)\) is zero.
\end{proof}
\begin{corollary}
	For any knot \(K\subset S^3\) and \(n\gg0\), \(S^3_{n}(K)\) is an L-space if and only if \(C_{2,n}(K)\) is an A-link.
\end{corollary}
\begin{proof}
	For any integer \(n\), \(T_K(n)=C_{2,n}(K)\). 
	Now choose \(n>4|\tau(K)|\) and apply \cref{thm:comparison:fillings}. 
\end{proof}
It is interesting to compare this to Hanselman and Watson's cabling formula for \(\gamma_K\) \cite{hanselman2019cabling}. 

\section{Growth of knot Floer and Khovanov homology under cabling}\label{sec:growth}

We now compare the Heegaard Floer multicurve invariant \(\gamma_{T}\) and the Khovanov multicurve invariant \(\Khr(T;\F)\) \cite{KWZ} for double tangles \(T=T_K\). 
To emphasize the similarity between the two invariants, we will in this section denote the former by \(\HFcurve{T}\) and the latter by \(\Khcurve{T}\). 
Like \(\HFcurve{T}\), the multicurve invariant \(\Khcurve{T}\) satisfies a gluing theorem that allows one to compute the Khovanov homology of the union of two tangles in terms of Lagrangian Floer homology \cite[Theorem~1.9]{KWZ}. 
Even more remarkably, the components of \(\Khcurve{T}\) are also subject to very similar geography restrictions \cite[Theorem~1.2]{KWZ_linear}. 
We do not need the full statement of this classification; it suffices to say that there is a class of \textit{rational} curves \(\rKh(k)\), parametrized by their \textit{slope} \(k\in\Z\), and a class of \textit{special} curves \(\sKh_{\ell}(\infty)\), parametrized by their \textit{length} \(\ell\in2\Z\).
To simplify notation, we will write 
\[
\sKh_{\ell}\coloneqq\sKh_{\ell}(\infty)
\quad\text{and}\quad
\rKh_k\coloneqq\rKh(k).
\]
For clarity, we will in this section write \(\sHF_{\ell}\), \(\sSHF_{\ell}\), and \(\rHF_k\) for the curves previously denoted by \(\s_{\ell}\), \(\sS_{\ell}\), and \(\rr_k\), respectively, for all \(\ell\in\Z^{>0}\) and \(k\in\Z\). 

We have the following structure theorem for \(\Khcurve{T_K}\), which is implicit in \cite{LZ}. 

\begin{theorem}\label{thm:structure:Khcurve}
	For any knot \(K\subset S^3\), the multicurve \(\gamma_{T_K}\) contains a single rational component \(\rKh_{2\vartheta_2(K)}\), where \(\vartheta_2(K)\in\Z\) is the concordance homomorphism from \cite{LZ}, and every other curve segment is of type \(\sKh_{\ell_i}\) for some \(\ell_i>0\).
\end{theorem}
\begin{proof}
	The tangle \(T_K\) is the quotient tangle of a strongly invertible knot, namely \(K\#r(K)\), where \(r(K)\) is the reverse of \(K\). 
	So by \cite[Theorem~3.1]{KWZ_strong_inversions}, every component except one is of type \(\sKh_{\ell_i}\) for some \(\ell_i\in2\Z^{>0}\), and the remaining one is of type \(\rKh_{k}\) for some slope \(k\in\Z\). 
	The slope \(k\) is divisible by 2 because rational components (of odd length) detect how tangle ends are connected \cite[Theorem~5.6]{KWZ_linear}. 
	(This implies that also in \cite[Theorem~3.1]{KWZ_strong_inversions}, the slope \(k\) is always divisible by 2.)
	
	It remains to relate the slope \(k\) to the concordance invariant \(\vartheta_2(K)\). 
	By \cite[Corollary~5.14]{LZ}, \(2\vartheta_2(K)\) is the slope of the multicurve invariant \(\BNr(T_K)\) near the bottom right tangle end.
	By naturality of the mapping class group action on \(\BNr(T_K)\) and \(\Khr(T_K)\) \cite[Theorem~1.13]{KWZ}, it suffices to show that the rational component of the twisted curve \(\Khr(T_K+Q_{-2\vartheta_2(K)})\) is 0.
	This follows from a simple computation carried out in \cite[Proof of Lemma~5.11]{LZ}.
\end{proof}

\begin{remark}
	Combining \cref{thm:main:graded,thm:structure:Khcurve}, we see that the slopes of the rational components in \(\HFcurve{T_K}\) and \(\Khcurve{T_K}\) are \(4\tau(K)\) and \(2\vartheta_2(K)\), respectively. 
	So, as already observed in \cite[Section~1.3]{LZ}, \(2\tau(K)\) plays the same role in knot Floer homology as the invariant \(\vartheta_2(K)\) in Khovanov homology.
\end{remark}

For the proofs of \cref{prop:growth:HFK:knots,prop:growth:Khr}, we need the following basic result.

\begin{lemma}\label{lem:pairing_linear_curves:dimension_formula}
	Let \(k,k'\in\Z\) with \(k\neq k'\) and \(\ell\in2\Z^{>0}\).
	Then 
	\begin{align*}
	\dim\HF(\gamma,\gamma')
	&=
	\begin{cases*}
	2\ell & if \((\gamma,\gamma')\in\{(\rKh_{k},\sKh_{\ell}), (\rHF_{k},\sHF_{\ell}), (\rHF_{k},\sSHF_{\ell})\}\)
	\\
	2|k'-k| & if \((\gamma,\gamma')\in\{(\rKh_{k},\rKh_{k'}),(\rHF_{k},\rHF_{k'})\}\)
	\\
	4 & if \(\gamma=\gamma'=\rKh_{k}\)
	\\
	2 & if \(\gamma=\gamma'=\rHF_{k}\)
	\end{cases*}
	\end{align*}
\end{lemma}

\begin{proof}
	These are straightforward computations.
\end{proof}

\begin{proof}[Proof of \cref{prop:growth:HFK:knots}]
	By the gluing theorem, combined with the \ref{thm:main:intro},
	\begin{align*}
	2\dim \HFKhat(C_{2,2t+1}(K))
	&=
	\dim\HF(\rHF_{2t+1},\HFcurve{T_K})
	\\
	&=
	\dim\HF(\rHF_{2t+1},\rHF_{4\tau(K)})
	+
	\sum_{\s}\dim\HF(\rHF_{2t+1},\s)
	\end{align*}
	where the sum is over all special connected components \(\s\) in \(\HFcurve{T_K}\). 
	The first summand is equal to \(2|2t+1-4\tau(K)|\) by \cref{lem:pairing_linear_curves:dimension_formula}. 
	Each term in the second summand is equal to twice the length of the special curve \(\s\). 
	The number of terms is \((d-1)\) by the \ref{thm:main:intro}, so the second term is equal to \(4(d-1)\overline{\ell}\).
  The desired identity follows.  
	
	The upper bound for \(\dim \HFKhat(C_{2,2t+1}(K))\) is clear.  
	The lower bound follows from the observation that \(\overline{\ell}\geq1\) and \(|2t+1-4\tau(K)|\) is odd.
\end{proof}

\begin{proof}[Alternative proof of \cref{prop:growth:HFK:knots}]
	We briefly outline an alternative argument that uses Hanselman and Watson's cabling formula for the multicurve invariant \(\gamma_K\) \cite{hanselman2019cabling}. 
	As noted in the introduction, for any knot \(K\subset S^3\), the dimension of \(\HFKhat(K)\) is equal to the minimal number of intersection points of \(\gamma_{K}\) with the meridian \(\mu_{K}\) (counted with multiplicity of the dimension of any local system). 
	Equivalently, \(\dim\HFKhat(K)\) is equal to the number of elements in \(\Sset(\gamma_{K})\). 
	The same is true, of course, if we replace \(K\) by \(K'\coloneqq C_{2,2t+1}(K)\). 	
	Using Hanselman and Watson's cabling formula, we can easily compute the multicurve \(\gamma_{K'}\) from \(\gamma_{K}\); in particular, we determine \(\Sset(\gamma_{K'})\) from \(\Sset(\gamma_{K})\) as follows: 
	Every component \(\capR_{\ell}\) contributes a single component \(\capR_{\ell'}\) (for some \(\ell'\in\Z^{>0}\)), \(\ell\) components \(\capR_{1}\) and \((\ell-1)\) components \(\capL_{1}\) in \(\Sset(\gamma_{K'})\). 
	Similarly, \(\capL_{\ell}\) contributes a total of \(2\ell\) components in \(\Sset(\gamma_{K'})\). 
	From this, we get the first summand in \cref{prop:growth:HFK:knots}.
	The second summand comes from the contribution of \(\ccon_{2\tau(K)}\). 
	This curve contributes a single component \(\ccon_{2\tau(K')}\) to \(\Sset(\gamma_{K'})\) as well as  \(\lfloor|\tfrac{2t+1}{2}-2\tau(K)|\rfloor\) pairs of components \(\capR_{1}\) and \(\capL_{1}\). 
\end{proof}

\begin{proof}[Proof of \cref{prop:growth:Khr}]
	We start by observing that each of the two components of the link \(C_{2,2t}(K)\) is equal to \(K\). 
	So by \cite[Corollary~1.6]{BatsonSeed},
	\[
	\dim\Kh(C_{2,2t}(K))
	\geq
	(\dim\Kh(K))^2,
	\]
	where \(\Kh(\cdot)\) denotes unreduced Khovanov homology over \(\F\). 
	Thus, for reduced Khovanov homology \(\Khr(\cdot)\), we obtain
	\begin{equation}\label{eqn:BatsonSeed}
	\dim\Khr(C_{2,2t}(K))
	\geq
	2(\dim\Khr(K))^2=2d^2,
	\end{equation}
	since \(\dim\Kh(J)=2\dim\Khr(J)\) for any link \(J\) \cite[Corollary~3.2.C]{Shumakovitch}. 
	By the gluing theorem, combined with \cref{thm:structure:Khcurve},
	\begin{align*}
	2\dim \Khr(C_{2,k}(K))
	&=
	\dim\HF(\rKh_k,\Khcurve{T_K})
	\\
	&=
	\dim\HF(\rKh_k,\rKh_{2\vartheta_2(K)})
	+
	\sum_{\s}\dim\HF(\rKh_k,\s)
	\end{align*}
	for any integer \(k\),	where the sum is over all special components \(\s\) in \(\Khcurve{T_K}\). 
	The first summand is equal to \(2|k-2\vartheta_2(K)|\), unless \(k=2\vartheta_2(K)\), in which case it is equal to 4. 
	The second summand is independent of \(k\).  
	Setting \(k=2t+1\) and \(k=2\vartheta_2(K)\), we obtain
	\[
	\dim \Khr(C_{2,2t+1}(K))
	= 
	\dim \Khr(C_{2,2\vartheta_2(K)}(K)) - 2 + |2t+1-2\vartheta_2(K)| 
	\]
	We now combine this identity with \eqref{eqn:BatsonSeed} and obtain the desired inequality.
\end{proof}

I close with a few remarks about \cref{conj:Kh:double-tangle}. 
I have verified this conjecture for a selection of about 30 knots with up to 11 crossings using the computer program \cite{khtpp}. 
The conjecture is also consistent with \cref{prop:growth:Khr}---at least, if we assume the following conjecture, which generalises \cite[Conjecture~3.10]{KWZ_strong_inversions} and for which there is now overwhelming computational evidence. 

\begin{conjecture}\label{conj:Kh:divisibility}
	For any Conway tangle \(T\), the length of every special component of \(\Khcurve{T}=\Khr(T;\F)\) is divisible by 4. 
\end{conjecture}

Note this conjecture is known to be false if \(\Khr(T;\F)\) is computed over a field of different characteristic. 
Now, assuming \cref{conj:Kh:double-tangle,conj:Kh:divisibility}, 
\cref{prop:growth:Khr} follows by the same argument as \cref{prop:growth:HFK:knots}. 
By the gluing theorem, 
\begin{align*}
2\dim \Khr(C_{2,2t+1}(K))
&=
\dim\HF(\rKh_{2t+1},\Khcurve{T_K})
\\
&=
\dim\HF(\rKh_{2t+1},\rKh_{2\vartheta_2(K)})
+
\sum_{\s}\dim\HF(\rKh_{2t+1},\s),
\end{align*}
where the sum is over all special connected components \(\s\) in \(\Khcurve{T_K}\). 
By assumption, there are at least \(\tfrac{1}{2}(d^2-1)\) such components (where \(d=\dim\Khr(K)\)) and each of them contributes at least 8 to the expression above.  
So the overall contribution of special components of \(\Khcurve{T_K}\) to \(\dim \Khr(C_{2,2t+1}(K))\) is \(2(d^2-1)\), i.e.\ the first two summands in \cref{prop:growth:Khr}. 
The remaining summand is equal to the contribution of the rational component of \(\Khcurve{T_K}\).

\bibliographystyle{amsalpha}
\newcommand{\arxiv}[1]{\relax\href{http://arxiv.org/abs/#1}{\tt arXiv:\penalty -100\unskip#1}}
\bibliography{main}

\providecommand{\bysame}{\leavevmode\hbox to3em{\hrulefill}\thinspace}
\providecommand{\MR}{\relax\ifhmode\unskip\space\fi MR }
\providecommand{\MRhref}[2]{%
  \href{http://www.ams.org/mathscinet-getitem?mr=#1}{#2}
}
\providecommand{\href}[2]{#2}
\begin{thebibliography}{KWZ22b}

\bibitem[BGH21]{boyer2021slope}
Steven Boyer, Cameron~McA Gordon, and Ying Hu, \emph{Slope detection and
  toroidal 3-manifolds}, 2021, \arxiv{2106.14378}.

\bibitem[BS15]{BatsonSeed}
Joshua Batson and Cotton Seed, \emph{A link-splitting spectral sequence in
  {K}hovanov homology}, Duke Mathematical Journal \textbf{164} (2015), no.~5,
  801--841, \arxiv{1303.6240}.

\bibitem[Dey19]{dey2019cable}
Subhankar Dey, \emph{Cable knots are not thin}, 2019, \arxiv{1904.11591}.

\bibitem[Ghi08]{ghiggini2006knot}
Paolo Ghiggini, \emph{Knot {F}loer homology detects genus-one fibred knots},
  Amer. J. Math. \textbf{130} (2008), no.~5, 1151--1169, \arxiv{math/0603445}.

\bibitem[Han19]{hanselman_program}
Jonathan Hanselman, \emph{github repository for immersed curve invariants of
  knot complements}, 2019,
  \url{https://github.com/hanselman/CFK-immersed-curves/}.

\bibitem[HRW16]{HRW}
Jonathan Hanselman, Jacob Rasmussen, and Liam Watson, \emph{Bordered {F}loer
  homology for manifolds with torus boundary via immersed curves}, 2016,
  \arxiv{1604.03466}.

\bibitem[HRW22]{HRWcompanion}
\bysame, \emph{{H}eegaard {F}loer homology for manifolds with torus boundary:
  properties and examples}, Proc.\ Lond.\ Math.\ Soc. \textbf{125} (2022),
  no.~4, 879--967, \arxiv{1810.10355}.

\bibitem[HW15]{HW1}
Jonathan Hanselman and Liam Watson, \emph{A calculus for bordered {F}loer
  homology}, 2015, \arxiv{1508.05445}, to appear in Geom.\ Topol.

\bibitem[HW18]{hedden2017geography}
Matthew Hedden and Liam Watson, \emph{On the geography and botany of knot
  {F}loer homology}, Selecta Math. (N.S.) \textbf{24} (2018), no.~2, 997--1037,
  \arxiv{1404.6913}.

\bibitem[HW19]{hanselman2019cabling}
Jonathan Hanselman and Liam Watson, \emph{Cabling in terms of immersed curves},
  2019, \arxiv{1908.04397}, to appear in Geom.\ Topol.

\bibitem[KWZ19]{KWZ}
Artem Kotelskiy, Liam Watson, and Claudius Zibrowius, \emph{{I}mmersed curves
  in {K}hovanov homology}, 2019, \arxiv{1910.14584}.

\bibitem[KWZ20]{KWZ_mnemonic}
\bysame, \emph{A mnemonic for the {L}ipshitz-{O}zsváth-{T}hurston
  correspondence}, 2020, \arxiv{2005.02792}, to appear in Algebr. Geom. Topol.

\bibitem[KWZ21]{KWZ_thin}
\bysame, \emph{{T}hin links and {C}onway spheres}, 2021, \arxiv{2105.06308}.

\bibitem[KWZ22a]{KWZ_strong_inversions}
\bysame, \emph{Khovanov homology and strong inversions}, Open Book Series
  \textbf{5} (2022), no.~1, 223--244, \arxiv{2104.13592}.

\bibitem[KWZ22b]{KWZ_linear}
\bysame, \emph{{K}hovanov multicurves are linear}, 2022, \arxiv{2202.01460}.

\bibitem[LMZ22]{LMZ}
Tye Lidman, Allison~H. Moore, and Claudius Zibrowius, \emph{{$L$}-space knots
  have no essential {C}onway spheres}, Geom. Topol. \textbf{26} (2022), no.~5,
  2065--2102, \arxiv{2006.03521}.

\bibitem[LOT18]{LOT}
Robert Lipshitz, Peter~S. Ozsv\'{a}th, and Dylan~P. Thurston, \emph{Bordered
  {H}eegaard {F}loer homology}, Mem. Amer. Math. Soc. \textbf{254} (2018),
  no.~1216, viii+279, \arxiv{0810.0687v5}.

\bibitem[LZ22]{LZ}
Lukas Lewark and Claudius Zibrowius, \emph{{R}asmussen invariants of
  {W}hitehead doubles and other satellites}, 2022, \arxiv{2208.13612}.

\bibitem[OS04]{OSgenus}
Peter Ozsvath and Zoltan Szabo, \emph{Holomorphic disks and genus bounds},
  Geometry {\&} Topology \textbf{8} (2004), no.~1, 311--334,
  \arxiv{math/0311496}.

\bibitem[OS05a]{HFKlens}
P.~Ozsv\'{a}th and Z.~Szab\'{o}, \emph{On knot {F}loer homology and lens space
  surgeries}, Topology \textbf{44} (2005), no.~6, 1281--1300,
  \arxiv{math/0303017}. \MR{2168576}

\bibitem[OS05b]{OSspectralsequence}
Peter Ozsváth and Zoltán Szabó, \emph{On the {H}eegaard {F}loer homology of
  branched double-covers}, Advances in Mathematics \textbf{194} (2005), no.~1,
  1--33, \arxiv{math/0309170}.

\bibitem[OS17]{hfkcalc}
\bysame, \emph{Knot {F}loer homology calculator},
  \url{https://web.math.princeton.edu/~szabo/HFKcalc.html}, 2017.

\bibitem[OSS15]{GridHomologyBook}
Peter~S. Ozsváth, András~I. Stipsicz, and Zoltán Szabó, \emph{Grid homology
  for knots and links}, vol. 208, Math.\ Surv.\ Monogr., Am.\ Math.\ Soc.,
  2015.

\bibitem[PW21]{petkova2021twisted}
Ina Petkova and Biji Wong, \emph{Twisted {M}azur pattern satellite knots and
  bordered {F}loer theory}, 2021, \arxiv{2005.12795}.

\bibitem[RR17]{RR}
Jacob Rasmussen and Sarah~Dean Rasmussen, \emph{Floer simple manifolds and
  {L}-space intervals}, Adv. Math. \textbf{322} (2017), 738--805. \MR{3720808}

\bibitem[Shu04]{Shumakovitch}
Alexander~N. Shumakovitch, \emph{Torsion of the {K}hovanov homology}, Fund.
  Math. (2004), 343--364, \arxiv{math/0405474}.

\bibitem[Zar11]{ZarevThesis}
Rumen Zarev, \emph{Bordered sutured {F}loer homology}, ProQuest LLC, Ann Arbor,
  MI, 2011, Thesis (Ph.D.)--Columbia University.
  (\href{https://doi.org/10.7916/D83R10V4}{doi:~10.7916/D83R10V4}).

\bibitem[Zib18]{PQM.m}
Claudius Zibrowius, \emph{{\texttt{PQM.m}}}, 2018, Mathematica package for
  computing the Heegaard Floer multicurve invariant for Conway tangles,
  available from the author's website.

\bibitem[Zib19a]{HDsForTangles}
\bysame, \emph{Kauffman states and {H}eegaard diagrams for tangles}, Algebr.
  Geom. Topol. \textbf{19} (2019), no.~5, 2233--2282, \arxiv{1601.04915}.

\bibitem[Zib19b]{pqSym}
\bysame, \emph{{O}n symmetries of peculiar modules; or, {\(\delta\)}-graded
  link {F}loer homology is mutation invariant}, 2019, \arxiv{1909.04267v2}, to
  appear in J. Eur. Math. Soc.

\bibitem[Zib20]{pqMod}
\bysame, \emph{Peculiar modules for 4-ended tangles}, J. Topol. \textbf{13}
  (2020), no.~1, 77--158, \arxiv{1712.05050v3}.

\bibitem[Zib21]{khtpp}
\bysame, \emph{\texttt{kht++}, a program for computing {K}hovanov invariants
  for links and tangles}, \url{https://cbz20.raspberryip.com/code/khtpp/docs/},
  2021.

\bibitem[Zib22]{ancillary}
\bysame, \emph{python script {\texttt{domains.py}}, ancillary file to this
  article}, 2022, available at
  \url{https://cbz20.raspberryip.com/code/ancillary/HF2HFT/domains.py}.

\bibitem[Zib23]{hf2hft}
\bysame, \emph{python script {\texttt{hf2hft.py}}, ancillary file to this
  article}, 2023, available at
  \url{https://cbz20.raspberryip.com/code/ancillary/HF2HFT/hf2hft.py}.

\end{thebibliography}
\end{document}